\documentclass{amsart}

\usepackage[latin1]{inputenc}
\usepackage{amsfonts,amssymb,amsmath,amsthm}
\usepackage{a4wide}
\usepackage[all]{xy}
\usepackage{textcomp}
\usepackage{graphicx}
\usepackage{url}
\usepackage{mathtools}
\usepackage{enumerate}
\usepackage{tikz}

\newif\ifdebug                                                      %
\debugfalse

\newcommand{\printname}[1]
   {\smash{\makebox[0pt]{\hspace{-.5in}\raisebox{-8pt}{\tiny #1}}}}
\newcommand{\labell}[1] {\ifdebug {\label{#1}\printname{#1}}
                        \else    {\label{#1}} \fi}

\swapnumbers
\numberwithin{equation}{section}

\newtheorem{thm}[equation]{Theorem}
\newtheorem{prop}[equation]{Proposition}
\newtheorem{Proposition}[equation]{Proposition}

\newtheorem{lemma}[equation]{Lemma}
\newtheorem*{lemma*}{Lemma}

\newtheorem{cor}[equation]{Corollary}

\theoremstyle{definition}
\newtheorem{defi}[equation]{Definition}
\theoremstyle{remark}
\newtheorem{Remark}[equation]{Remark}
\newtheorem{Remark*}[equation]{Remark}
\newtheorem{rmk}[equation]{Remark}
\newtheorem{ex}[equation]{Example}

\def\eor{\unskip\ \hglue0mm\hfill$\diamond$\smallskip\goodbreak}
\def\eoe{\unskip\ \hglue0mm\hfill$\between$\smallskip\goodbreak}
\def\eod{\unskip\ \hglue0mm\hfill$\diamond$\smallskip\goodbreak}

\newcommand{\bb}[1]{\mathbb{#1}}
\newcommand{\wt}[1]{\widetilde{#1}}
\DeclareMathOperator{\ind}{ind}

\DeclareMathOperator{\gr}{gr}
\DeclareMathOperator{\im}{im}

\DeclareMathOperator{\Cont}{Cont}
\DeclareMathOperator{\Sp}{Sp}
\DeclareMathOperator{\Id}{Id}
\DeclareMathOperator{\dist}{dist}

\DeclareMathOperator{\Giv}{Giv}

\def \Z {{\mathbb Z}}
\def \R {{\mathbb R}}

\def \C {{\mathbb C}}
\def \RP {{\mathbb R}{\mathbb P}}
\def \CP {{\mathbb C}{\mathbb P}}
\def \calE {{\mathcal E}}
\def \calS {{\mathcal S}}

\def \zt {{\mathbb Z}_2}
\def \zk {{{\mathbb Z}_k}}

\def \ssminus {\smallsetminus}
\def \ol {\overline}
\def \ul {\underline}

\def \calB {{\mathcal B}}
\def \calC {{\mathcal C}}

\def \calO {{\mathcal O}}

\def \tX {\wt{X}}
\def \tY {\wt{Y}}
\def \tA {\wt{A}}

\def \del {\partial}

\begin{document}


\title
[Givental's non-linear Maslov index on lens spaces]
{Givental's non-linear Maslov index on lens spaces}

\author{Gustavo Granja}
\address{Center for Mathematical Analysis, Geometry and Dynamical Systems, Instituto Superior Técnico,
Universidade de Lisboa, Av. Rovisco Pais, 1049-001 Lisboa, Portugal}
\email{ggranja@math.tecnico.ulisboa.pt}

\author{Yael Karshon}
\address{Department of Mathematics, University of Toronto, 40 St. George Street, Toronto, Ontario, M5S 2E4 Canada}
\email{karshon@math.toronto.edu}

\author{Milena Pabiniak}
\address{Mathematisches Institut, Universit\"{a}t zu K\"{o}ln, Cologne, Germany}
\email{pabiniak@math.uni-koeln.de}

\author{Sheila Sandon}
\address{Universit\'e de Strasbourg, CNRS, IRMA UMR 7501, F-67000 Strasbourg, France}
\email{sandon@math.unistra.fr}

\maketitle

\begin{abstract}
\noindent 
Givental's \emph{non-linear Maslov index}, constructed in 1990,
is a quasimorphism 
on the universal cover of the identity component of the contactomorphism group 
of real projective space. 
This invariant was used by several authors
to prove contact rigidity phenomena 
such as orderability, 
unboundedness of the discriminant and oscillation metrics,
and a contact geometric version of the Arnold conjecture.
In this article we give an analogue for lens spaces
of Givental's construction and its applications.
\end{abstract}

\setcounter{tocdepth}{1}
\tableofcontents


\section{Introduction}
\labell{section: introduction}

A \emph{discriminant point} of a contactomorphism $\phi$
of a co-oriented contact manifold $(V,\xi)$
is a point $p$ of $V$ such that $\phi(p) = p$
and $(\phi^{\ast} \alpha)_p = \alpha_p$
for some (hence any) contact form $\alpha$ for $\xi$;
the \emph{discriminant} of $(V,\xi)$
is the space of those contactomorphisms
that have at least one discriminant point.
Givental's non-linear Maslov index \cite{Givental - Nonlinear Maslov index}
assigns to every contact isotopy $\{\phi_t\}$ of real projective space 
$\mathbb{RP}^{2n-1}$
with its standard contact structure
an integer $\mu (\{\phi_t\})$
that is defined using generating functions and can be interpreted
as an intersection index of the path $\{\phi_t\}$ in the contactomorphism group
with (a certain subspace of) the discriminant.
This number only depends on the homotopy class of $\{\phi_t\}$
with fixed endpoints,
and thus defines a map
$$
\mu \colon 
\widetilde{\text{Cont}}_0(\mathbb{RP}^{2n-1}) \rightarrow \mathbb{Z}
$$
on the universal cover of the identity component
of the contactomorphism group.
It follows from \cite[Theorem 9.1]{Givental - Nonlinear Maslov index}
that $\mu$ is a quasimorphism,
i.e.\ a homomorphism up to a bounded error
(cf.\ Ben Simon \cite{Ben Simon}).
While quasimorphisms on Hamiltonian groups
were studied by several authors,
starting with Biran, Entov and Polterovich \cite{BEP}
and Entov and Polterovich \cite{EP - Calabi qm},
Givental's non-linear Maslov index
and its reductions studied by Borman and Zapolsky \cite{BZ}
are the only known non-trivial quasimorphisms on contactomorphism groups.

In \cite{Givental - Nonlinear Maslov index} Givental also studied
intersections with the discriminant in two other related settings.
One is a space of Legendrian submanifolds of $\mathbb{RP}^{2n-1}$,
with discriminant given
by those Legendrians that intersect a fixed one.
The second setting
(that was also studied by Th\'eret \cite{Theret - Rotation numbers})
is the Hamiltonian group
of complex projective space $\mathbb{CP}^{n-1}$
with the Fubini--Study symplectic form;
in this case the discriminant is formed by
Hamiltonian diffeomorphisms of $\mathbb{CP}^{n-1}$
that lift to contactomorphisms of $S^{2n-1}$ having discriminant points.
The applications of the non-linear Maslov index
that were discussed already in \cite{Givental - Nonlinear Maslov index}
include a proof of the Arnold conjecture
for fixed points of Hamiltonian diffeomorphisms
and for Lagrangian intersections in $\mathbb{CP}^{n-1}$,
results about existence of Reeb chords between Legendrians 
in $\mathbb{RP}^{2n-1}$ that are Legendrian isotopic to each other,
and a proof of the chord and Weinstein conjectures for $\mathbb{RP}^{2n-1}$.

After the work of Givental,
discriminant points appeared again more recently
in proofs (based on generating functions)
of other contact rigidity results.
In \cite{B} Bhupal used the rigidity of discriminant points
to define a partial order
on the identity component of the group
of compactly supported contactomorphisms of
the standard contact Euclidean space $\mathbb{R}^{2n+1}$.
Elaborating on the work of Bhupal,
the fourth author obtained a new proof
of the contact non-squeezing theorem of Eliashberg, Kim and Polterovich 
\cite{EKP},
and a construction of an integer-valued bi-invariant metric
on the identity component of the group
of compactly supported contactomorphisms of
$\mathbb{R}^{2n} \times S^1$ 
(\cite{S - Contact homology capacity and nonsqueezing}
and  \cite{S - metric} respectively).
In these works the role played by
discriminant and translated points is made more explicit,
and appears to be similar to the one
described in \cite{Givental - Nonlinear Maslov index}.
Recall from \cite{S - Contact homology capacity and nonsqueezing, S - Iterated}
that a point $p$ of a contact manifold $(V,\xi)$
is said to be a \emph{translated point}
of a contactomorphism $\phi$
with respect to a contact form $\alpha$ for $\xi$
if $p$ is a discriminant point of $\varphi^{\alpha}_{-\eta} \circ \phi$
for some real number $\eta$ (called the \emph{time-shift}),
where $\varphi^{\alpha}_t $ denotes the Reeb flow.
Consider the contact product
$\big( V \times V \times \mathbb{R} \,,\, \ker (e^{\theta}\alpha_1 - \alpha_2) \big)$,
where $\theta$ is the $\R$-coordinate
and $\alpha_1$ and $\alpha_2$ the pullbacks of $\alpha$
by the projection of $V \times V \times \mathbb{R}$
on the first and second factor respectively.
Denote by $\gr (\phi)$ the Legendrian graph of $\phi$,
i.e.\ the Legendrian submanifold of $V \times V \times \mathbb{R}$
defined by
$$
\gr(\phi) = \big\{ \big( p , \phi(p) , g(p) \big) \;\lvert\; p \in V \big\}
$$
where $g$ is the function determined by the relation $\phi^{\ast} \alpha = e^g \alpha$.
Discriminant points of $\phi$ correspond
to intersections between $\text{gr}(\phi)$ and the diagonal $\Delta \times \{0\} = \text{gr}(\text{id})$;
since the Reeb vector field of $e^{\theta} \alpha_1 - \alpha_2$ is $(0, R_{\alpha}, 0)$,
where $R_{\alpha}$ denotes the Reeb vector field of $\alpha$,
translated points of $\phi$
correspond to Reeb chords
between $\text{gr}(\phi)$ and the diagonal $\Delta \times \{0\}$.
By Weinstein's theorem, if $\phi$ is $\mathcal{C}^1$-close to the identity
then $\text{gr}(\phi)$ can be identified
with the $1$-jet $j^1f \subset J^1 ( \Delta \times \{0\} )$ of a function $f$ on $V \cong \Delta \times \{0\}$.
Critical points of $f$ correspond
to Reeb chords in $J^1 ( \Delta \times \{0\} )$ between the zero section and $j^1f$,
hence\footnote{
Recall that this correspondence is in general not 1--1.
Every Reeb chord in $J^1 ( \Delta \times \{0\} )$ between the zero section and $j^1f$
corresponds to a Reeb chord in $V \times V \times \mathbb{R}$
between $\Delta \times \{0\}$ and $\gr (\phi)$,
but there might be Reeb chords in $V \times V \times \mathbb{R}$
between $\Delta \times \{0\}$ and $\gr (\phi)$
that are not contained in the Weinstein neighborhood of $\Delta \times \{0\}$,
and thus do not correspond to a Reeb chord in $J^1 ( \Delta \times \{0\} )$
between the zero section and $j^1f$.}
to Reeb chords in $V \times V \times \mathbb{R}$ between $\Delta \times \{0\}$ and $\gr (\phi)$,
hence to translated points of $\phi$.
Moreover, critical points of critical value zero
correspond to discriminant points.
It follows from this local description that 
for such a $\phi$ translated points always exists,
while discriminant points can be removed by a small perturbation.
On the other hand, 
Givental's non-linear Maslov index for $\mathbb{RP}^{2n-1}$
and the bi-invariant metric for $\mathbb{R}^{2n} \times S^1$
defined in \cite{S - metric}
show that there exist contact isotopies
that must intersect the discriminant at least a certain number of times,
and these intersections cannot be removed
by perturbing the contact isotopy
in the same homotopy class with fixed endpoints.
This rigidity of discriminant points is also the main ingredient
in \cite{S - Contact homology capacity and nonsqueezing}
for the construction of a contact capacity for domains of $\mathbb{R}^{2n} \times S^1$,
and for the proof of the contact non-squeezing theorem.
Elaborating on this idea,
the fourth author, in collaboration with Colin,
defined \cite{CS - Discriminant metric} bi-invariant (pseudo)metrics
(the \emph{discriminant metric} and the \emph{oscillation pseudometric})
on the universal cover of the identity component
of the contactomorphism group of any compact contact manifold,
and, using results from \cite{Givental - Nonlinear Maslov index},
proved that both are unbounded
in the case of $\mathbb{RP}^{2n-1}$.
Previously, Givental's non-linear Maslov index
was also used by Eliashberg and Polterovich \cite{EP - Partially ordered groups}
to show that $\mathbb{RP}^{2n-1}$ is \textit{orderable},
i.e.\ it does not admit any positive contractible loop of contactomorphisms\footnote{
Recall from \cite{EP - Partially ordered groups}
that a contact isotopy of a co-oriented contact manifold
is said to be {\it positive (non-negative)}
if it moves every point in a direction positively transverse (or tangent)
to the contact distribution,
and that a compact co-oriented contact manifold $(V,\xi)$
is said to be \emph{orderable} if
the relation $\leq$ on $\widetilde{\Cont}_0(V,\xi)$
defined by posing
$$
[\{\phi_t\}] \leq [\{\psi_t\}]
\; \text{ if } \; [\{\psi_t\}] \cdot [\{\phi_t\}]^{-1} 
\text{ can be represented by a non-negative contact isotopy}
$$
is a partial order.
By \cite[Criterion 1.2.C]{EP - Partially ordered groups}, this is equivalent to asking that $(V,\xi)$
does not admit any positive contractible loop of contactomorphisms.
Recall also that
the oscillation pseudometric on a compact co-oriented contact manifold $(V,\xi)$
is a metric if and only if $(V,\xi)$ is orderable \cite{CS - Discriminant metric}.},
and by the fourth author \cite{S - Morse estimate for translated points}
to prove that any contactomorphism of $\mathbb{RP}^{2n-1}$
isotopic to the identity has at least $2n$ translated points.

In the present article we give an analogue for lens spaces
of the construction of Givental's non-linear Maslov index
and its applications.
This is the first step of a more general program:
study the contact rigidity phenomena mentioned above
in the case of prequantizations of symplectic toric manifolds,
by extending to the contact case techniques from \cite{Givental - toric}.
See also \cite{Tervil} for work in this direction,
and \cite{AF} for other work related to the non-linear Maslov index.

Although $\mathbb{RP}^{2n-1}$ and $S^{2n-1}$
are both prequantizations of $\mathbb{CP}^{n-1}$,
and the former is the quotient of the latter
by the antipodal $\zt$-action,
$S^{2n-1}$ (for $n>1$) is not orderable \cite{EKP}
and does not admit non-trivial quasimorphisms
and unbounded bi-invariant metrics \cite{FPR}.
In this work we prove that lens spaces
(in particular those that are prequantizations of $\mathbb{CP}^{n-1}$)
behave rather as $\mathbb{RP}^{2n-1}$:
in spite of the fact that the ring structure of the cohomology
of general lens spaces is different
from that of projective space,
which affects the proofs of some of the key properties
of the topological invariant that is used in the construction,
we show that it is still possible to define the non-linear Maslov index
and use it to extend to lens spaces
the applications described above.

Let $k \geq 2$ be an integer and $\underline{w} = (w_1, \cdots, w_{n})$
an $n$-tuple of positive integers that are relatively prime to~$k$.
The lens space $L_k^{2n-1}(\underline{w})$ is
the quotient of the unit sphere $S^{2n-1}$ in $\C^{n}$ 
by the free $\Z_k$-action generated by the map
\begin{equation}\labell{equation: action}
(z_1,\cdots,z_n) \mapsto 
( e^{ \frac{2\pi i}{k} \cdot w_1} z_1, \cdots, e^{ \frac{2\pi i}{k} \cdot w_n} z_n ) \,.
\end{equation}
We equip the lens space $L_k^{2n-1}(\underline{w})$
with its standard contact structure,
i.e.\ the kernel of the contact form whose pullback to $S^{2n-1}$
is equal to the pullback from $\mathbb{R}^{2n}$
of the 1-form $\sum_{j=1}^n ( x_j dy_j - y_j dx_j)$.
We denote by $\{r_t\}$ the Reeb flow on $L_k^{2n-1}$
with respect to this contact form.

Throughout the article, when the weights are not relevant in the discussion
we denote the lens space $L_k^{2n-1}(\underline{w})$ simply by $L_k^{2n-1}$.
As usual, we see the universal cover $\widetilde{\Cont}_0(L_k^{2n-1})$
of the identity component of the contactomorphism group
as the space of contact isotopies starting at the identity
modulo smooth 1-parameter families with fixed endpoints;
the group operation is given by
$[\{\phi_t\}] \cdot [\{\psi_t\}] = [\{\phi_t \circ \psi_t\}]$.

Our main result is the following theorem.

\begin{thm}\labell{theorem: main}
For any lens space $L_k^{2n-1}$ with its standard contact structure
there is a map 
$$
\mu \colon \widetilde{\Cont}_0(L_k^{2n-1}) \rightarrow \mathbb{Z}
$$
such that $\mu \big([\{r_{2 \pi l t}\}_{t \in [0,1]}] \big) = 2nl$ for every integer $l$, 
and with the following properties:

\begin{enumerate}
\renewcommand{\labelenumi}{(\roman{enumi})}

\item 
\emph{(Quasimorphism.)} 
For any two elements $[\{\phi_t\}]$ and $[\{\psi_t\}]$ 
of $\widetilde{\Cont}_0(L_k^{2n-1})$ we have
$$
\left| \, \mu\big([\{\phi_t\}] \cdot [\{\psi_t\}]\big) - \mu\big([\{\phi_t\}]\big) - \mu\big([\{\psi_t\}]\big) \, \right| \leq 2n + 1\,.
$$

\item \emph{(Positivity.)} 
If $\{\phi_t\}$ is a non-negative contact isotopy
then $\mu \big( [\{\phi_t\}] \big) \geq 0$.
If $\{\phi_t\}$ is a positive contact isotopy
then $\mu \big( [\{\phi_t\}] \big) > 0$. 

\item \emph{(Relation with discriminant points.)} 
Let $\{\phi_t\}_{t \in [0,1]}$ be a contact isotopy of $L_k^{2n-1}$,
and $[t_0,t_1]$ a subinterval of $[0,1]$.
If $\mu \big( [\{\phi_t\}_{t \in [0,t_0]}] \big) \neq \mu \big( [\{\phi_t\}_{t \in [0,t_1]}] \big)$
then there is $\underline{t} \in [t_0,t_1]$ such that
$\phi_{\underline{t}}$ belongs to the discriminant.
If there is only one such $\underline{t}$
then the following holds:
if $\phi_{\underline{t}}$ has only finitely many discriminant points then
$$
\left| \, \mu \big( [\{\phi_t\}_{t \in [0,t_0]}] \big) - \mu \big( [\{\phi_t\}_{t \in [0,t_1]}] \big) \, \right| \leq 2 \,;
$$
if all discriminant points of $\phi_{\underline{t}}$
are non-degenerate\footnote{
Recall from \cite{S - Morse estimate for translated points} that a 
discriminant point $p$ of a contactomorphism $\phi$
of a contact manifold $\big(V,\xi = \text{ker}(\alpha)\big)$
is said to be \textit{non-degenerate}
if there are no vectors $X \in T_pV$ such that
$\phi_{\ast}X = X$ and $dg(X) = 0$,
where $g$ is the function defined by
$\phi^{\ast} \alpha = e^g \alpha$.
A translated point of $\phi$ of time-shift $\eta$
is said to be \emph{non-degenerate}
if it is a non-degenerate discriminant point
of $\varphi^{\alpha}_{-\eta} \circ \phi$
(where $\varphi_t^{\alpha}$ denotes the Reeb flow).}
then
$$
\left| \, \mu \big( [\{\phi_t\}_{t \in [0,t_0]}] \big) - \mu \big( [\{\phi_t\}_{t \in [0,t_1]}] \big) \, \right| \leq 1 \,.
$$
\end{enumerate}
\end{thm}

The map $\mu \colon \widetilde{\Cont}_0(L_k^{2n-1}) \rightarrow \mathbb{Z}$
(the \emph{non-linear Maslov index})
is defined at the beginning of Section \ref{sec:Maslov index},
using the material that is developed
in Sections \ref{sec:generating functions} and \ref{section: cohomological index}.
The quasimorphism property is proved in Proposition \ref{p:quasimorphism property},
positivity in Proposition \ref{proposition: positivity}
and the relation with discriminant points in Proposition \ref{discriminant points}.
The calculation for the Reeb flow is presented
in Example \ref{Reeb flow}.

Theorem \ref{theorem: main} allows to extend
the applications of the non-linear Maslov index to the case of lens spaces,
giving the following results (see Section \ref{section: applications}).

\begin{cor}\labell{cor1}
Consider a lens space $L_k^{2n-1}$ with its standard contact structure.
Then:
\begin{enumerate}
\renewcommand{\labelenumi}{(\roman{enumi})}
\item $L_k^{2n-1}$ is orderable.
\item The discriminant and oscillation metrics
on $\widetilde{\Cont}_0(L_k^{2n-1})$
are unbounded.
\item Any contactomorphism of $L_k^{2n-1}$ contact isotopic to the identity
has at least $n$ translated points with respect to the standard contact form.
Moreover, if all translated points are non-degenerate
then their number is at least $2n$.
\item Any contact form on $L_k^{2n-1}$ 
defining the standard contact structure
has at least one closed Reeb orbit.
\end{enumerate}
\end{cor}

Orderability of lens spaces was also proved
with different methods by Milin \cite{Milin}
and by the fourth author \cite{S - Lens spaces}.
Regarding part (iii),
this proves for the standard contact form of lens spaces
two statements in direction of the following conjecture:
if $(V,\xi)$ is a compact contact manifold
then any contactomorphism $\phi$ contact isotopic to the identity
should have at least as many translated points
(with respect to any contact form for $\xi$)
as the minimal number of critical points of a smooth function on $V$.
This conjecture, formulated in \cite{S - Morse estimate for translated points},
can be thought of as a contact analogue of the Arnold conjecture
on fixed points of Hamiltonian diffeomorphisms.
As in the Hamiltonian case,
one can consider weaker versions
obtained by replacing the lower bound on translated points by 
the Lusternik--Schnirelmann category or
(even weaker) the cup length,
or the version where the lower bound is the sum of the Betti numbers 
if all translated points are assumed to be non-degenerate.
Working with $\zk$-coefficients (with $k$ prime),
for any lens space $L_k^{2n-1}$
the sum of the Betti numbers is $2n$,
while the cuplength is $2n$ if $k = 2$
(i.e.\ for $\mathbb{RP}^{2n-1}$)
and $n + 1$ if $k > 2$.
While our bound in the general case is just $n$,
the one obtained in \cite{S - Morse estimate for translated points}
in the case of $\mathbb{RP}^{2n-1}$ is $2n$.
On the other hand,
since the Lusternik--Schnirelmann category of $L_k^{2n-1}$ is $2n$ for all $k$,
we should still have at least $2n$ translated points for all $L_k^{2n-1}$
also in the degenerate case.
It might be possible to prove this
using Massey products
(similarly to \cite{Viterbo - Massey products}),
but this goes beyond the scope of the present article
(see also Remark \ref{remark: Massey}).

\begin{rmk}\labell{remark: k prime}
In the case $k = 2$ our arguments prove,
as in \cite{Givental - Nonlinear Maslov index},
the following stronger form of Theorem \ref{theorem: main}:
in (i) the bound is $2n$, rather than $2n+1$,
and in (iii) the last bound holds also in the degenerate case.
Using this one recovers (see Section \ref{section: applications})
the stronger bound in Corollary \ref{cor1}(iii)
that holds in the case of $\RP^{2n-1}$:
any contactomorphism contact isotopic to the identity
has at least $2n$ translated points with respect to the standard contact form.
It is enough to prove Theorem \ref{theorem: main}
in the case when $k$ is prime.
Indeed, for any multiple $k'$ of $k$
one then obtains a quasimorphism on $\widetilde{\text{Cont}}_0(L_{k'}^{2n-1})$
with the required properties
by pulling back $\mu$ by the natural map 
$\widetilde{\text{Cont}}_0(L_{k'}^{2n-1})
\rightarrow \widetilde{\text{Cont}}_0(L_k^{2n-1})$.
Because of this,
if $k$ is even then Theorem \ref{theorem: main} and Corollary \ref{cor1}
hold in the same stronger form as in the case $k = 2$.
By a similar method
(i.e.\ pulling back $\mu$ via the natural map induced by taking a quotient in stages),
one also obtains quasimorphisms for quotients of $S^{2n-1}$
by a finite non-cyclic subgroup $G$ of $U(n)$
that acts freely on $S^{2n-1}$,
as any such $G$ contains a non-trivial cyclic subgroup.
\eor
\end{rmk}

Note that there exist contactomorphisms of $L_k^{2n-1}$
that have exactly $2n$ translated points.
Indeed, as above,
if a contactomorphism $\phi$ of $L_k^{2n-1}$
is sufficiently $\mathcal{C}^1$-close to the identity
then its graph $\gr(\phi)$ in the contact product $L_k^{2n-1} \times L_k^{2n-1} \times \R$
is contained in a Weinstein neighborhood of the diagonal $\Delta \times \{0\}$,
and can be identified to the 1-jet of a function $f$ on $\Delta \times \{0\} \cong L_k^{2n-1}$.
Critical points of $f$ correspond to Reeb chords in $J^1 (\Delta \times \{0\})$
between the zero section and $j^1f$,
hence to Reeb chords in $L_k^{2n-1} \times L_k^{2n-1} \times \R$
between $\Delta \times \{0\}$ and $\gr (\phi)$,
hence to translated points of $\phi$.
This correspondence is now 1--1.
Indeed, if $p$ is a translated point of $\phi$
with respect to the standard contact form $\alpha_0$ on $L_k^{2n-1}$
then there exists a Reeb chord $\gamma: [0,1] \rightarrow L_k^{2n-1}$ of $\alpha_0$
with $\gamma(0) = p$ and $\gamma(1) = \phi(p)$
so that all points $(p,\gamma(t),0)$
are in the considered Weinstein neighborhood.
Since the Reeb flow on $L_k^{2n-1} \times L_k^{2n-1} \times \R$
is given by $\Id \times \varphi^{\alpha}_t \times \Id$
it follows that for every translated point of $\phi$
there is a Reeb chord in $L_k^{2n-1} \times L_k^{2n-1} \times \R$
between $(p,p,0) \in \Delta \times \{0\}$
and $(p,\phi(p),0) \in \gr(\phi)$
that is contained in the Weinstein neighborhood,
and thus any translated point of $\phi$
corresponds to a critical point of $f$.
Choosing a function $f$ on $L_k^{2n-1}$
with exactly $2n$ critical points
we thus obtain a contactomorphism with exactly $2n$ translated points.

Since the Reeb flow of the standard contact form $\alpha_0$ on $L_k^{2n-1}$ is 1-periodic,
if $p$ is a translated point with respect to $\alpha_0$
of a contactomorphism $\phi$
then there are infinitely many Reeb chords connecting $p$ and $\phi(p)$,
hence the translated point $p$ has infinitely many time-shifts.
Thus Corollary \ref{cor1}(iii)
implies that any contactomorphism of $L_k^{2n-1}$
contact isotopic to the identity
has infinitely many pairs (translated point, time-shift) with respect to $\alpha_0$.
Using the properties of the non-linear Maslov index
that are listed in Theorem \ref{theorem: main},
we will prove in Section \ref{section: applications} that this result remains true
for any contact form on $L_k^{2n-1}$ defining the standard contact structure.

\begin{cor}\label{cor 1bis}
Let $\alpha$ be any contact form on $L_k^{2n-1}$
defining the standard contact structure.
For any contactomorphism $\phi$ of $L_k^{2n-1}$ contact isotopic to the identity
there are infinitely many distinct real numbers
that are time-shifts of translated points of $\phi$ with respect to $\alpha$.
In particular, $\phi$ has at least one translated point
with respect to $\alpha$.
\end{cor}

As in the case of projective space,
it follows from Theorem \ref{theorem: main}
that the \emph{asymptotic non-linear Maslov index}
$$
\ol{\mu} \colon \widetilde{\text{Cont}}_0(L_k^{2n-1}) \rightarrow \mathbb{R} \, , \;
\ol{\mu}([\{\phi_t\}]) = \lim_{m \to \infty} \frac{\mu([\{\phi_t\}]^m)}{m}
$$
is \emph{monotone},
i.e.\ $\mu([\{\phi_t\}] )\leq \mu([\{\psi_t\}])$ whenever  $[\{\phi_t\}] \leq [\{\psi_t\}]$,
and has the \emph{vanishing property},
i.e.\ it vanishes on any element that can be represented
by a contact isotopy supported in a displaceable set
(see Propositions
\ref{proposition: monotonicity - asymptotic} and
\ref{proposition: vanishing property}).
In \cite{BZ} Borman and Zapolsky showed that, 
in analogy with the symplectic case \cite{Borman 2},
in certain situations monotone quasimorphisms descend under contact reduction;
starting from Givental's non-linear Maslov index on $\mathbb{RP}^{2n-1}$
they thus obtained induced quasimorphisms
on those contact toric manifolds that can be written 
in a certain way as contact reductions of $\mathbb{RP}^{2n-1}$.
Moreover, it is proved in \cite{BZ} that if a contact manifold
admits a non-trivial monotone quasimorphism
then it is orderable,
and if the prequantization of a symplectic toric manifold 
admits a non-trivial monotone quasimorphism with the vanishing property
(a property that is preserved under contact reduction)
then it has a non-displaceable pre-Lagrangian toric fibre.
As already observed in  \cite[Remark 1.5]{BZ}, 
our generalization to lens spaces of Givental's non-linear Maslov index
allows us to extend the class of contact toric manifolds
that inherit, by contact reduction, a quasimorphism.
We thus obtain the following result
(see Section \ref{section: applications}).

\begin{cor}\labell{cor2}
Let $(W^{2n}, \omega)$ be
a compact monotone symplectic toric manifold.
Write the moment polytope as
$\Delta = \{\, x \in \mathfrak{t}^{\ast} \text{ , } \langle \nu_j,x \rangle + \lambda \geq 0 \text{ for } j=1,\cdots,d\,\}$,
where $d$ is the number of facets and $\nu_j \in \mathfrak{t}$
are vectors normal to the facets and primitive in the integer lattice
$\mathfrak{t}_{\mathbb{Z}} = \text{\emph{ker}} \, (\text{\emph{exp}} \colon \mathfrak{t} \rightarrow \mathbb{T}^n)$.
Suppose that, for some $k \in \mathbb{N}$,
$\sum_{j=1}^d \nu_j \in k \, \mathfrak{t}_{\mathbb{Z}}$.
Then there is a rescaling $a \omega$
of the symplectic form such that
the prequantization of $(W,a\omega)$ 
admits a non-trivial monotone quasimorphism
with the vanishing property,
and so is orderable and contains
a non-displaceable pre-Lagrangian toric fibre.
\end{cor}

Note however that the link with discriminant points
seems to be lost in the reduction process.
Therefore, it is not clear if it is possible to use these induced quasimorphisms
to obtain the applications (other than orderability)
listed in Corollary \ref{cor1}.

The original idea of the construction of
Givental's non-linear Maslov index in $\mathbb{RP}^{2n-1}$
is as follows.
Given a contact isotopy $\{\phi_{t}\}_{t \in [0,1]}$ of $\RP^{2n-1}$,
one associates to it
a 1-parameter family of functions 
$f_t \colon \RP^{2M-1} \rightarrow \R$, for some large $M$,
so that critical points of $f_t$ of critical value zero
correspond to discriminant points of $\phi_t$.
In order to detect intersections of $\{\phi_{t}\}$ with the discriminant
one then analyzes the changes in topology of the sublevel sets $A_t = \{f_t \leq 0\}$.
This is done by studying the \emph{cohomological index} of these sets,
i.e.\ the dimension of the image of the homomorphism 
$H^{\ast}(\RP^{2M-1}; \mathbb{Z}_2) \rightarrow H^{\ast}(A_t;\mathbb{Z}_2)$
induced by the inclusion $A_t \hookrightarrow \RP^{2M-1}$.
The non-linear Maslov index of $\{\phi_t\}_{t \in [0,1]}$
is then defined to be the difference between the cohomological indices 
of $A_0$ and $A_1$.
The key difference in the construction for lens spaces 
is in the properties of the cohomological index.
In the case of projective space the cohomological index satisfies
\emph{subadditivity}
(the cohomological index of a union $A \cup B$
is not more than the sum of the cohomological indices of $A$ and $B$)
and \emph{join additivity}
(the cohomological index of an \emph{equivariant join}
is equal to the sum of the cohomological indices of the factors).
The proofs of both properties use the fact that
the cohomology ring of projective space is generated
by the class in degree one,
and thus they do not go through in the case of lens spaces.
However we show that weaker versions
of the subadditivity and join additivity properties
also hold in the case of general lens spaces,
and are enough to define the non-linear Maslov index
and prove the properties listed in Theorem \ref{theorem: main}.

In the case $k = 2$, Givental's proof of the join additivity property
uses an equivariant K\"unneth formula
\cite[Proposition A.1]{Givental - Nonlinear Maslov index}.
A crucial ingredient in the proof
is the fact that the K\"{u}nneth short exact sequence
(of modules over the equivariant cohomology of a point) splits,
but it is not clear to us why this fact should be true.
In any case, for $k > 2$ the K\"{u}nneth (or Eilenberg--Moore) spectral sequence
does not collapse in general
(due to the fact that for $k > 2$
the $\zk$-equivariant cohomology of a point  has zero divisors)
and so we do not even have a K\"{u}nneth short exact sequence
for the equivariant cohomology of a product.
We thus present a different proof (even for the case $k = 2$).
Our proof is based on the study of an equivariant join operation in homology,
which is defined in the same way for all $k$.
When $k = 2$ the properties of this operation
imply Givental's join additivity,
while for $k > 2$,
as the join of two even dimensional generators of the equivariant homology of a point is zero,
we only obtain a weaker \emph{join quasi-additivity property}
(Proposition \ref{proposition: cohomological index} (v)).
As mentioned above, this property is still strong enough
to imply the applications we are interested in.

Finally, the construction of generating functions
that we use in this article
is not the one from \cite{Givental - Nonlinear Maslov index}\footnote{
Because of this,
strictly speaking we do not know whether in the special case of projective space
our quasimorphism actually coincides
with Givental's.
However, all the features and properties are the same.},
but is an adaptation to our setting
of the work of Th\'eret \cite{Theret - Rotation numbers}
(cf.\ Remark \ref{remark: cf Givental}).

The article is organized as follows.
In Section \ref{sec:generating functions}
we explain how to construct
1-parameter families of generating functions
for contact isotopies of lens spaces,
and discuss some properties.
In Section \ref{section: cohomological index}
we describe the topological invariant
(the cohomological index)
that is used to analyze the changes in topology
of the sublevel sets of generating functions,
deferring to Appendix \ref{section: additivity under join}
the most technical part of the proof
of the join quasi-additivity property of this invariant.
In Section \ref{sec:Maslov index} we put these ingredients together
to define the non-linear Maslov index of a contact isotopy,
and prove the properties listed in Theorem \ref{theorem: main}.
In Section \ref{section: applications} we use the non-linear Maslov index
to prove Corollaries \ref{cor1}, \ref{cor 1bis} and \ref{cor2},
mostly following the corresponding arguments in the case of projective space.
In Appendix \ref{section appendix: on the construction}
we discuss several interpretations of the composition formula
that is used in the construction of generating functions.

Throughout the article, when we say that we follow
Givental \cite{Givental - Nonlinear Maslov index} or Th\'{e}ret \cite{Theret - Rotation numbers}
we mean that their arguments,
developed for $\mathbb{RP}^{2n-1}$ and $\mathbb{CP}^{n-1}$,
can be repeated for lens spaces
without any modification other than replacing the
action of  $\mathbb{Z}_2$ or $S^1$ by the $\mathbb{Z}_k$-action 
\eqref{equation: action}.

\subsection*{Acknowledgments} 

We thank Sue Tolman and Lisa Jeffrey for useful discussions,
and Dietmar Salamon for questions that prompted us to include Corollary \ref{cor 1bis}.
We also thank Instituto Superior T\'ecnico in Lisbon
for its hospitality and excellent work environment during several
of our meetings.
Part of this work was done
while the fourth author was visiting
the Unit\'e Mixte Internationale of CNRS
at the Centre de Recherche Math\'ematiques of Montr\'eal;
she wishes to thank CNRS and the CRM for this opportunity;
she also wishes to thank the members of the ANR grant CoSpIn
for many discussions
on the non-linear Maslov index and related topics.
We thank the referees for several useful comments and corrections.

The first author was partially supported
by the Funda\c{c}\~ao para a Ci\^encia e a Tecnologia of Portugal
through the project UID/MAT/04459/2013
and by the Gulbenkian Foundation.
The second author was partially supported by the Natural Sciences
and Engineering Research Council of Canada.
The third author was supported by
the Funda\c{c}\~ao para a Ci\^encia e a Tecnologia of Portugal
through the fellowship SFRH/BPD/87791/2012
and the projects PTDC/MAT/117762/2010 and EXCL/MAT-GEO/0222/2012,
and, during the last year of the project, partially supported
by the SFB-TRR 191 grant 
\textit{Symplectic Structures in Geometry, Algebra and Dynamics} 
funded by the Deutsche Forschungsgemeinschaft. 
The fourth author was partially supported
by the ANR grant CoSpIn.


\section{Generating functions}
\labell{sec:generating functions}

In this section we explain how,
given a contact isotopy $\{\phi_t\}_{t\in[0,1]}$ of a lens space $L_k^{2n-1}$
starting at the identity, 
one can define a 1-parameter family
of functions $f_t $ on a higher dimensional lens space $L_k^{2M-1}$
so that critical points of $f_t$ of critical value zero
correspond to discriminant points of $\phi_t$.
We also discuss some properties of such families of generating functions
that are important in the applications:
uniqueness, monotonicity and quasi-additivity.

\subsection*{Generating functions for symplectomorphisms of $\R^{2n}$.}

We start by recalling the definition of generating functions 
of Lagrangian submanifolds in cotangent bundles.
Observe first that any Lagrangian section 
of the cotangent bundle $T^{\ast}B$ of a manifold $B$
is the graph of a closed 1-form on $B$;
if this 1-form is exact,
given by the differential of a function $F$,
then we say that $F$
is a generating function for the Lagrangian. 
Generalizing this idea, one can associate a generating function 
to a larger class
of Lagrangian submanifolds of $T^{\ast}B$ by the following construction,
which goes back to H\"{o}rmander \cite{Hormander}.
Consider a function $F \colon E \rightarrow \mathbb{R}$
defined on the total space of a fibre bundle $p \colon E \rightarrow B$. 
Let $N_E^{\ast}$ be the fibre conormal bundle, i.e.\ 
the space of points $(e,\eta)$ of $T^{\ast}E$ 
such that $\eta$ vanishes on the kernel of $d p |_e$.
We say that $F$ is a {\it generating function}
if $dF \colon E \rightarrow T^{\ast}E$ is transverse to $N_E^{\ast}$.
If this condition is satisfied then
the set of fibre critical points $\Sigma_F = (dF)^{-1}(N_E^{\ast})$ 
is a smooth submanifold of $E$,
and the map
$$
i_F \colon \Sigma_F \rightarrow T^{\ast}B \; , \; 
e \mapsto \big(p(e),v^{\ast}(e)\big)
$$
defined by posing $v^{\ast}(e)(X) = dF\big(\widehat{X}\big)$ 
for $X \in T_{p(e)}B$,
where $\widehat{X}$ is any vector in $T_eE$ 
such that $p_{\ast} (\widehat{X}) = X$,
is a Lagrangian immersion.
If $i_F$ is an embedding then we say that $F$ is
a generating function for the Lagrangian submanifold $L_F := 
i_F(\Sigma_F)$ of $T^{\ast}B$.
In this case, critical points of $F$ are in 1--1 correspondence 
with intersections of $L_F$ with the zero section. 

In our applications, generating functions
are always defined on trivial bundles of the form
$E = \R^{2n} \times \R^{2nN} \to \R^{2n}$.
Denoting the coordinates by $(\zeta,\nu)$
with $\zeta \in \R^{2n}$ and $\nu \in \R^{2nN}$,
we then have that $dF$ is transverse to $N_E^{\ast}$ 
if and only if zero is a regular value of the vertical derivative
$\frac{\partial F}{\partial \nu} \colon E \rightarrow (\R^{2nN})^{\ast}$;
moreover, $\Sigma_F = (\frac{\partial F}{\partial \nu})^{-1}(0)$
and $i_F(\zeta,\nu)=(\zeta,\frac{\partial F}{\partial \zeta}(\zeta, \nu))$.

For a symplectomorphism 
$\Phi$ of $\R^{2n}$, as in \cite{Viterbo} we consider 
the Lagrangian submanifold $\Gamma_{\Phi}$ of $T^{\ast}\R^{2n}$ 
that is the image of the graph of $\Phi$ under the symplectomorphism 
$\tau \colon \overline{\R^{2n}} \times \R^{2n} \rightarrow T^{\ast}\R^{2n}$ 
given by
$$
\displaystyle{\tau (x,y,X,Y) 
= \Big(\frac{x+X}{2} \,,\, \frac{y+Y}{2} \,,\, Y-y \,,\, x-X \Big)}
$$
or, in complex notation, 
\begin{equation}\labell{e:identification}
\tau (z,Z) = \Big(\frac{z+Z}{2} \, , \, i (z-Z)\Big).
\end{equation}
We say that a function $F$ 
is a generating function for $\Phi$ if it is a generating function 
for $\Gamma_{\Phi}$. 
Since $\tau$ sends the diagonal onto the zero section, 
critical points of a generating function $F$ of $\Phi$ 
are in 1--1 correspondence with fixed points of $\Phi$.
Note also that if $F$ is a generating function for $\Phi$
then $-F$ is a generating function for $\Phi^{-1}$.

Any Hamiltonian symplectomorphism $\Phi$ of $\mathbb{R}^{2n}$
such that $\Gamma_{\Phi}$ is a section of $T^{\ast}\mathbb{R}^{2n}$
has a generating function
$F \colon \mathbb{R}^{2n} \rightarrow \mathbb{R}$.
In order to obtain a generating function 
for more general Hamiltonian symplectomorphisms
we use the following composition formula\footnote{
Th\'eret's composition formula in \cite{Theret - Rotation numbers} is
$$
F_1 \, \sharp \, F_2 \, (q;\zeta_1,\zeta_2,\nu_1,\nu_2) 
= F_1 (q + \zeta_2, \nu_1) + F_2 (\zeta_1 + \zeta_2, \nu_2) + 2 \left<q-\zeta_1,i\zeta_2\right> \,.
$$
Although it differs from ours just by a change of variables,
in \cite{Theret - Rotation numbers} it is proved to hold only
under the assumption that $F_1$ or $F_2$ has no fibre variables.
This is not sufficient for our purposes:
in the proof of the quasimorphism property 
of the non-linear Maslov index (Proposition \ref{p:quasimorphism property})
we need to allow fibre variables in both factors.}.

\begin{Proposition}[Composition formula]
\labell{p:composition formula}
If $F_1 \colon \R^{2n} \times \R^{2nN_1} \to \R$ 
and $F_2 \colon \R^{2n} \times \R^{2nN_2} \to \R$
are, respectively, generating functions for symplectomorphisms 
$\Phi^{(1)}$ and $\Phi^{(2)}$ of $\R^{2n}$,
then the function 
$F_1 \, \sharp \, F_2 \colon
\mathbb{R}^{2n} \times 
 (\mathbb{R}^{2n} \times \mathbb{R}^{2n} 
                  \times \mathbb{R}^{2nN_1} \times \mathbb{R}^{2nN_2})
\rightarrow \mathbb{R}$
defined by
$$
F_1 \, \sharp \, F_2 \, (q; \zeta_1, \zeta_2, \nu_1,\nu_2) 
 = F_1 (\zeta_1, \nu_1) + F_2 (\zeta_2, \nu_2) 
   - 2 \left< \zeta_2 - q, i (\zeta_1- q) \right>
$$
(where $\langle \,\cdot\,,\,\cdot\,\rangle$ denotes the standard inner product on $\mathbb{R}^{2n}$)
is a generating function for the composition 
$\Phi = \Phi^{(2)} \circ \Phi^{(1)}$.
\end{Proposition}

In  Appendix~\ref{section appendix: on the construction} we discuss two 
interpretations of the composition formula in terms of symplectic reduction, 
a generalization to any even number of factors 
and the relation to the method of broken trajectories
of Chaperon, Laudenbach and Sikorav \cite{Ch1,LS, Sikorav 1, Sikorav 2}
and to the construction in Givental \cite{Givental - Nonlinear Maslov index}.
Below we present a direct proof.

\begin{proof}

\textit{Step 1:  Criterion for fibre critical points}.

The vertical derivative of $F_1 \, \sharp \, F_2$ is
\begin{equation}\labell{d vert calF}
(q;\zeta_1, \zeta_2,\nu_1,\nu_2) \mapsto
\Big( \frac{\del F_1}{\del \zeta_1} + 2i(\zeta_2 - q), \frac{\del F_2}{\del\zeta_2} - 2i(\zeta_1 - q),
\frac{\del F_1}{\del \nu_1}, \frac{\del F_2}{\del \nu_2}\Big) \,,
\end{equation}
thus a point $(q; \zeta_1,\zeta_2,\nu_1,\nu_2)$ 
is a fibre critical point of $F_1 \, \sharp \, F_2$
if and only if $(\zeta_j,\nu_j)$ is a fibre critical point 
of $F_j$ ($j=1$, $2$) and
\begin{equation}\labell{equation: fibre critical points}
\begin{cases}
\frac{\partial F_1}{\partial \zeta_1} = - 2i \, (\zeta_2 - q) \\
\frac{\partial F_2}{\partial \zeta_2} = 2i \, (\zeta_1 - q) \,.
\end{cases}
\end{equation}
Since $F_j \colon \mathbb{R}^{2n} \times \mathbb{R}^{2nN_j} \rightarrow \mathbb{R}$ 
is a generating function for $\Phi^{(j)}$ ($j = 1$, $2$),
the map
\begin{equation}\labell{equation: zeta-z}
\Sigma_{F_j} \rightarrow \mathbb{R}^{2n} \, , \; (\zeta_j,\nu_j) \mapsto z_j
\end{equation}
given by the composition
$$
\xymatrix{ 
\Sigma_{F_j} \ar[r]^{i_{F_j}} & \Gamma_{\Phi^{(j)}} 
   \ar[rr]^-{\tau^{-1}|_{ \Gamma_{\Phi^{(j)}}}}
   && \textrm{gr}(\Phi^{(j)})  }
   \subset \ol{\R^{2n}} \times \R^{2n}
\xymatrix{
\ar[rr]^-{(z,Z)\mapsto z} && \R^{2n} }
$$
is a diffeomorphism.
Under the change of variables \eqref{equation: zeta-z}
the equations \eqref{equation: fibre critical points}
become
$$
\begin{cases}
i \, \big(z_1 - \Phi^{(1)}(z_1)\big) = - 2i \,  \big( \frac{z_2 + \Phi^{(2)}(z_2)}{2} - q \big) \\
i \, \big(z_2 - \Phi^{(2)}(z_2)\big) = 2i \,  \big( \frac{z_1 + \Phi^{(1)}(z_1)}{2} - q\big)
\end{cases}
$$
i.e.
\begin{equation}\labell{equation: fcp}
\begin{cases}
z_2 = \Phi^{(1)}(z_1) \\
q = \frac{z_1 + \Phi^{(2)}(z_2)}{2}\,.
\end{cases}
\end{equation}

\textit{Step 2:  $F_1 \, \sharp \, F_2$ is a generating function}.

In order to prove that $F_1 \, \sharp \, F_2$ is a generating function
we need to show that zero is a regular value of the vertical derivative 
of $F_1 \, \sharp \, F_2$.
This can be seen as follows.
The diffeomorphism \eqref{equation: zeta-z} is the restriction to 
$\Sigma_{F_j} \subset \R^{2n} \times \R^{2nN_j}$ of the map
\begin{equation} \labell{e: formula}
\R^{2n} \times \R^{2nN_j} \to \R^{2n} , \qquad (\zeta_j,\nu_j)
\mapsto \zeta_j + 
\frac{1}{2i} 
\left.\frac{\del F_j}{\del \zeta_j}\right|_{\left( \zeta_j,\nu_j \right)}.
\end{equation}
Thus, for every $(\zeta_j,\nu_j) \in \Sigma_{F_j} = (\frac{\partial F_j}{\partial \nu_j})^{-1}(0)$,
the restriction to $T_{(\zeta_j,\nu_j)} \Sigma_{F_j}
= \ker \big(\left. d \big(\frac{\partial F_j}{\partial \nu_j}\big)\right|_{(\zeta_j,\nu_j)} \big)$
of the differential at $(\zeta_j,\nu_j)$ of \eqref{e: formula} is bijective. 
As, moreover, $\left.d(\frac{\partial F_j}{\partial \nu_j})\right|_{(\zeta_j,\nu_j)}
 \colon \R^{2n} \times \R^{2nN_j} \to \R^{2nN_j}$ is surjective,
this implies that the matrix
$$
\def\arraystretch{1.5}
M_j = 
\left( 
\begin{array}{ll}
{ \frac{1}{2i}\frac{\del^2 F_j}{\del \zeta_j^2} + I } 
   & { \frac{1}{2i}\frac{\del^2 F_j}{\del \nu_j \del \zeta_j} } \\
{ \frac{\del^2 F_j}{\del \zeta_j \del \nu_j} }
   & { \frac{\del^2 F_j}{\del \nu_j^2}  }
\end{array}
\right)
$$
is invertible at every $(\zeta_j,\nu_j)$ that is fibre critical for $F_j$.
The differential of \eqref{d vert calF} at $(q;\zeta_1, \zeta_2,\nu_1,\nu_2)$
is given by the matrix
$$ 
\def\arraystretch{1.5}
\left(
\begin{array}{c|cc|cc}
 -2iI & \frac{\del^2F_1}{\del\zeta_1^2} & 2iI &
 \frac{\del^2F_1}{\del\nu_1\del\zeta_1} & 0  \\
  2iI & -2iI & \frac{\del^2F_2}{\del\zeta_2^2} &
 0 & \frac{\del^2F_2}{\del\nu_2\del\zeta_2} \\
\hline
 0 & \frac{\del^2F_1}{\del\zeta_1\del\nu_1} & 0 &
 \frac{\del^2F_1}{\del\nu_1^2} & 0 \\
 0 & 0 & \frac{\del^2F_2}{\del\zeta_2\del\nu_2}  &
 0 & \frac{\del^2F_2}{\del\nu_2^2}  \\
\end{array}
\right)
$$
\labell{page: dF matrix}
which, by elementary row and column operations, can be brought to the form
$$ 
\def\arraystretch{1.5}
 \left(
\begin{array}{c|cc}
 \star & M_1 & 0  \\
 \star & \star & M_2 \\
\end{array} \right) \,.
$$
Since each $M_j$ is invertible, the columns of this matrix
span all of $\R^{2n} \times \R^{2n} \times \R^{2nN_1} \times \R^{2nN_2}$,
proving that zero is a regular value of the vertical derivative 
of $F_1 \, \sharp \, F_2$.

\textit{Step 3: $F_1 \, \sharp \, F_2$ is a generating function for $\Phi$.}

We need to show that the Lagrangian immersion
$i_{F_1 \sharp F_2} \colon \Sigma_{F_1 \sharp F_2} \rightarrow T^{\ast}\mathbb{R}^{2n}$
induces a diffeomorphism between $\Sigma_{F_1 \sharp F_2}$ 
and $\Gamma_{\Phi}$.
The relation \eqref{equation: fcp} for fibre critical points
and the fact that the maps \eqref{equation: zeta-z} are diffeomorphisms
imply that the map
\begin{equation}\labell{equation: fibre critical point - z1}
\Sigma_{F_1 \sharp F_2} \rightarrow \mathbb{R}^{2n} \, , \;
(q; \zeta_1,\zeta_2,\nu_1,\nu_2) \mapsto z_1
\end{equation}
is a diffeomorphism.
For a fibre critical point $(q;\zeta,\nu)$ we have
$$
\frac{\partial ( F_1 \, \sharp \, F_2)}{\partial q} \, (q;\zeta,\nu) =
2i \, (\zeta_1-\zeta_2) =
i \, \big( z_1 + \Phi^{(1)}(z_1) - z_2 - \Phi^{(2)}(z_2) \big) =
i \, \big(z_1 - \Phi(z_1)\big)
$$
and so
$$
i_{F_1 \sharp F_2}(q; \zeta,\nu) =
\Big(q, \frac{\partial\left( F_1 \, \sharp \, F_2\right)}{\partial q} \, 
        (q,\zeta,\nu)\Big) 
  = \Big( \frac{z_1 + \Phi(z_1)}{2}, i \big(z_1 - \Phi(z_1)\big) \Big) 
  = \tau \big(z_1,\Phi(z_1)\big)\,.
$$
In other words,
$i_{F_1 \sharp F_2} \colon \Sigma_{F_1 \sharp F_2} \rightarrow T^{\ast} \R^{2n}$
is the composition of the diffeomorphism \eqref{equation: fibre critical point - z1}
with the embedding $\R^{2n} \rightarrow T^{\ast}\R^{2n}$,
$z_1 \mapsto \tau \big(z_1,\Phi(z_1)\big)$,
and so it induces a diffeomorphism
between $\Sigma_{F_1 \sharp F_2}$ and $\Gamma_{\Phi}$.
\end{proof}

Proposition \ref{p:composition formula}
(as well as the analogous constructions in
\cite{Ch1, LS, Sikorav 1, Sikorav 2, Theret - Rotation numbers})
can be used to show
that every compactly supported Hamiltonian diffeomorphism of $\mathbb{R}^{2n}$
has a generating function quadratic at infinity.
This class of generating functions is used for instance
in the work of Viterbo \cite{Viterbo} and Traynor \cite{Traynor}.
As we now explain, in our case
(similarly to \cite{Givental - Nonlinear Maslov index} and \cite{Theret - Rotation numbers})
we use Proposition \ref{p:composition formula}
to produce instead \emph{conical generating functions}
for Hamiltonian diffeomorphisms of $\R^{2n}$
that lift contactomorphisms of lens spaces.

\subsection*{Generating functions for contact isotopies of lens spaces}

Throughout our discussion,
we fix the vector of weights $\ul{w}$
that defines the action \eqref{equation: action}
of $\zk$ on $\R^{2n}$,
and denote the lens space $L_k^{2n-1}(\underline{w})$ by $L_k^{2n-1}$.
On products $(\R^{2n})^N$ that occur as the domains of definition
of various generating functions
we always take the diagonal action of $\zk$ that is given 
by this same $\ul{w}$ on each factor,
and denote the corresponding lens space by $L_k^{2nN-1}$.
The $\R_{>0}$-action on $\R^{2n}$, or on products of $\R^{2n}$,
is always the radial action.

We say that a contact isotopy is a \emph{based contact isotopy} 
if it starts at the identity.

Given a based contact isotopy $\{\phi_t\}$ of $L_k^{2n-1}$,
we obtain a Hamiltonian isotopy $\{\Phi_t\}$ of $\R^{2n} \ssminus \{ 0 \}$
by first lifting $\{\phi_t\}$ to a $\zk$-equivariant based contact isotopy of $S^{2n-1}$
and then lifting this 
to a Hamiltonian isotopy of the symplectization of the sphere,
which we identify with $\R^{2n} \ssminus \{ 0 \}$.
We now explain this procedure in more detail.
Recall that the symplectization of a 
co-oriented contact manifold $(V,\xi)$ is
the symplectic submanifold $SV$ of $T^{\ast}V$ 
that consists of those covectors that vanish on the contact distribution
and are positive with respect to the given co-orientation.
Given a contactomorphism of $V$, its lift to the cotangent bundle  
restricts to a symplectomorphism of $SV$;
the lift to $SV$ of a contact isotopy of $V$ is a Hamiltonian isotopy.
A choice of a contact form $\alpha$ for $\xi$ gives a diffeomorphism
$\R \times V \to SV$, defined by
$(\theta, u) \mapsto e^{\theta} \alpha|_u$.
In the special case of $V = S^{2n-1}$ and
$\alpha = \sum_{j=1}^n ( x_j dy_j - y_j dx_j)$
we further identify $\R \times V$ with $\R^{2n} \ssminus \{ 0 \}$
by the map $(\theta, u) \mapsto \frac{1}{2} \, e^{\theta} u$. 
The lift of a contactomorphism $\phi \colon S^{2n-1} \to S^{2n-1}$
is then the map $\Phi \colon \R^{2n} \ssminus \{ 0 \} 
\to \R^{2n} \ssminus \{ 0 \}$ given by the formula
$$
\Phi(z) = \frac{|z|}{ e^{g(\frac{z}{|z|}) }} \; 
 \phi \big(\frac{z}{|z|}\big) \, ,
$$
where $\phi^{\ast} \alpha = e^g \alpha$,
and we extend $\Phi$ continuously to $\mathbb{R}^{2n}$ by setting $\Phi(0) = 0$.
For any based contact isotopy $\{\phi_t\}$ of $L_k^{2n-1}$
the resulting homeomorphisms $\Phi_t$ of $\R^{2n}$
are $(\zk \times \R_{>0})$-equivariant,
and are smooth symplectomorphisms on $\mathbb{R}^{2n} \ssminus \{0\}$.
We call such maps \emph{conical symplectomorphisms} of $\R^{2n}$.

As in \cite{Theret - Rotation numbers},
in order to work with conical symplectomorphisms
we must relax the smoothness assumption 
on our generating functions.
Notice first that if $\Phi$ is a conical symplectomorphism of $\R^{2n}$
such that $\Gamma_{\Phi}$ is a section of $T^{\ast}\mathbb{R}^{2n}$
then its generating function
$F \colon \mathbb{R}^{2n} \rightarrow \mathbb{R}$ 
is $\Z_k$-invariant, homogeneous of degree $2$,
smooth on $\R^{2n} \ssminus \{ 0 \}$
and $\mathcal{C}^1$ with Lipschitz differential everywhere\footnote{
Indeed, the first partial derivatives of $\left. F \right\vert_{\mathbb{R}^{2n} \smallsetminus 0}$
are $\mathbb{R}_{>0}$-homogeneous of degree one,
hence extend continuously to $\mathbb{R}^{2n}$,
and the second partial derivatives of $\left. F \right\vert_{\mathbb{R}^{2n} \smallsetminus 0}$
are $\mathbb{R}_{>0}$-invariant, hence bounded.}.
More generally, we say that
$F \colon E \to \R$, where $E = \R^{2n} \times \R^{2nN}$, 
is a \emph{conical function}
if it is $\mathcal{C}^1$ with Lipschitz differential,
$\zk$-invariant and homogeneous of degree $2$.
Moreover we say that such an $F \colon E \to \R$
is a \emph{conical generating function}
if it is smooth near its fibre critical points
other than the origin $(0,0) \in \R^{2n} \times \R^{2nN}$
and $dF \colon E \to T^*E$ is transverse to the fibre conormal bundle $N_E^{\ast}$
except possibly at the origin.
If this condition is satisfied then
the set $\Sigma_F$ of fibre critical points
is a smooth $(\zk \times \R_{>0})$-invariant submanifold
except possibly at the origin, 
and the corresponding map $i_F \colon \Sigma_F \to T^* \R^{2n}$
is continuous, $(\zk \times \R_{>0})$-equivariant,
and is a smooth Lagrangian immersion outside the origin. 
If $i_F$ is a homeomorphism between $\Sigma_F$ and $\Gamma_{\Phi}$
for a conical symplectomorphism $\Phi$
then we say that $F$ is a conical generating function for $\Phi$
and for the induced contactomorphism $\phi$ of $L_k^{2n-1}$.
By a \emph{family of conical generating functions} we mean
a collection of conical generating functions $F_s \colon E \to \R$, 
parametrized by $s \in S$ for some manifold with corners $S$ 
(for example $[0,1]$ or $[0,1] \times [0,1]$),
such that the map $(s,x) \mapsto F_s(x)$
is $\calC^1$ with locally Lipschitz differential 
and is smooth near $(s,x)$ whenever $x$ is a fibre critical point of $F_s$
other than the origin.

\begin{prop}\labell{proposition: conical family}
If $F^{(1)}_t$ and $F^{(2)}_t$ are families of conical generating functions
for contact isotopies $\{\phi^{(1)}_t\}$ and $\{\phi^{(2)}_t\}$ of $L_k^{2n-1}$,
then the functions $F^{(1)}_t \, \sharp \, F^{(2)}_t$
defined as in Proposition~\ref{p:composition formula}
form a family of conical generating functions 
for the composition $\{\phi^{(2)}_t \circ \phi^{(1)}_t\}$.
\end{prop}

\begin{proof}
We first show that $F^{(1)}_t \, \sharp \, F^{(2)}_t$
is a family of conical generating functions.
It is immediate to see
that each $F_t^{(1)} \, \sharp \, F^{(2)}_t$ is $\zk$-invariant
and homogeneous of degree 2,
and that the family is $\calC^1$ with locally Lipschitz differential.
The property of being smooth at fibre critical points other than the origin
is also preserved by the composition formula,
as we now explain.
Let $(q; \zeta_1, \zeta_2, \nu_1,\nu_2)$ be a fibre critical point of $F_t^{(1)} \, \sharp \, F_t^{(2)}$
other than the origin.
From \eqref{d vert calF}
we see that $(\zeta_j,\nu_j)$ is a fibre critical point of $F_t^{(j)}$ for $j=1,2$.
Moreover, $(\zeta_j,\nu_j)$ is the origin if and only if 
the point $z_j \in \R^{2n}$ that corresponds to it
by the bijection \eqref{equation: zeta-z} is also the origin.
If this happens for one of $j=1$ or $j=2$
then \eqref{equation: fcp} implies that
$q$ and both $z_1$ and $z_2$ are the origin
and thus, using the bijection \eqref{equation: zeta-z} once more, 
that $(q; \zeta_1, \zeta_2, \nu_1,\nu_2)$ is the origin,
contrary to our assumptions.
Therefore both $(\zeta_1,\nu_1)$ and $(\zeta_2,\nu_2)$
must be different from the origin,
and so $F_t^{(1)} \, \sharp \, F_t^{(2)}$ is smooth at $(q; \zeta_1, \zeta_2, \nu_1,\nu_2)$.
The proof now continues by repeating 
the proof of Proposition~\ref{p:composition formula};
notice that derivatives of order higher than one
are taken only at fibre critical points. 
\end{proof}

Using Proposition \ref{proposition: conical family} we now show that
any based contact isotopy of $L_k^{2n-1}$
has what we call a based family of conical generating functions.
Later in this section we show that for a given based contact isotopy
such a based family of conical generating functions
is essentially unique.

We say that a homeomorphism of $\R^{2n} \times \R^{2nN}$
is a \emph{fibre preserving conical homeomorphism} 
if it takes each fibre $\{ z \} \times \R^{2nN}$ to itself
and is $(\zk \times \R_{>0})$-equivariant.
The \emph{stabilization} of a conical generating function
$F \colon \R^{2n} \times \R^{2nN} \rightarrow \R$
by a non-degenerate $\zk$-invariant quadratic form
$Q \colon \R^{2nN'} \rightarrow \R$
is the conical generating function
$$
F \, \oplus \, Q \colon  \mathbb{R}^{2n} \times 
\mathbb{R}^{2nN} \times \mathbb{R}^{2nN'}
\rightarrow \mathbb{R} \,.
$$
On the set of conical generating functions
we consider the smallest equivalence relation under which
two such functions are equivalent
if they differ by stabilization
or by a fibre preserving conical homeomorphism
that restricts to a diffeomorphism
between neighborhoods of fibre critical points
other than the origin.

\begin{ex}\label{example: 0 sharp 0}
The zero function $0: \mathbb{R}^{2n} \rightarrow \mathbb{R}$ is equivalent to
$0 \,\sharp \, 0: \mathbb{R}^{2n} \times (\mathbb{R}^{2n} \times \mathbb{R}^{2n}) \rightarrow \mathbb{R}$;
indeed, $0 \,\sharp\, 0$ differs from $0$
by a stabilization followed by the fibre preserving conical homeomorphism
$(q; \zeta_1,\zeta_2) \mapsto (q; \zeta_1 - q,\zeta_2 - q)$.
\eoe
\end{ex}

\begin{rmk}\labell{remark: sharp preserves equivalence 1}
If $F^{(1)}$ and $F^{(2)}$ are equivalent respectively to $G^{(1)}$ and $G^{(2)}$
then $F^{(1)} \, \sharp \, F^{(2)}$ is equivalent to $G^{(1)} \, \sharp \, G^{(2)}$.
Moreover\footnote{
This fact is needed in the proof of Lemma \ref{lemma: index based}.},
if $F^{(j)}$ and $G^{(j)}$ differ by a fibre preserving conical homeomorphism
for each $j = 1, 2$
then so do $F^{(1)} \, \sharp \, F^{(2)}$ and $G^{(1)} \, \sharp \, G^{(2)}$.
\eor
\end{rmk}

\begin{rmk}\labell{remark: new}
If conical generating functions $F$ and $G$ are equivalent
then there exist a stabilization of $F$ and a stabilization of $G$
that differ by a fibre preserving conical homeomorphism
that restricts to a diffeomorphism
between neighborhoods of the fibre critical points.
\eor
\end{rmk}

We say that a family
$F_t: \mathbb{R}^{2n} \times \mathbb{R}^{2nN} \rightarrow \mathbb{R}$
of conical generating functions
is a \emph{based family of conical generating functions}
if $F_0$ is equivalent to the zero function $\mathbb{R}^{2n} \rightarrow \mathbb{R}$.

\begin{rmk}\label{remark: based}
By Example \ref{example: 0 sharp 0}
and Remark \ref{remark: sharp preserves equivalence 1},
if $F^{(1)}_t$ and $F^{(2)}_t$ are based families of conical generating functions
then so is $F^{(1)}_t \, \sharp \, F^{(2)}_t$.
\eor
\end{rmk}

\begin{prop}[Existence of generating functions] 
\labell{existence of generating functions}
Any based contact isotopy $\{\phi_t\}_{t\in[0,1]}$ of $L_k^{2n-1}$
has a based family $F_t \colon \R^{2n} \times \R^{2nN} \rightarrow \R$
of conical generating functions.
\end{prop}

\begin{proof}
For $N$ big enough we can write
$\phi_t = \phi_t^{(N)} \circ \cdots \circ \phi_t^{(1)}$
for based contact isotopies $\{\phi_t^{(j)}\}$
having (based) families of conical generating functions
$F_t^{(j)} \colon \R^{2n} \rightarrow \R$.
(For instance, we can take
$\phi_t^{(j)} = \phi_{\frac{j}{N}t} \circ (\phi_{\frac{j-1}{N}t})^{-1}$.)
By Proposition \ref{proposition: conical family} and Remark \ref{remark: based},
a family of the form
$F_t := F_t^{(1)} \, \sharp \, \cdots \, \sharp \, F_t^{(N)}$
(for any choice of parenthetization of the factors)
is then a based family of conical generating functions 
for $\{\phi_t\}$.
\end{proof}

In Section \ref{sec:Maslov index} we use generating functions
to define the non-linear Maslov index of a
contact isotopy of $L_k^{2n-1}$.
In order to show that the index is well defined
we use the fact that generating functions are in some sense unique.
The following discussion leads to this result,
which we state and prove in Proposition \ref{uniqueness gf} below.

By a \emph{family of fibre preserving conical homeomorphisms} 
we mean a collection of fibre preserving conical homeomorphism $\theta_s$, 
parametrized by $s \in S$ for some manifold with corners $S$
(for example $[0,1]$ or $[0,1] \times [0,1]$),
such that $(s,x) \mapsto \theta_s(x)$ is continuous.
The \emph{stabilization} of a 1-parameter family
$F_t \colon \R^{2n} \times \R^{2nN} \rightarrow \R$
of conical generating functions by a non-degenerate $\zk$-invariant quadratic form
$Q \colon \R^{2nN'} \rightarrow \R$
is the 1-parameter family of conical generating functions
$$
F_t \, \oplus \, Q \colon  \mathbb{R}^{2n} \times 
\mathbb{R}^{2nN} \times \mathbb{R}^{2nN'}
\rightarrow \mathbb{R} \,.
$$
On the set of 1-parameter families of conical generating functions
we consider the smallest equivalence relation under which
two such families are equivalent
if they differ by stabilization
or by a 1-parameter family of fibre preserving conical homeomorphisms 
that restrict to diffeomorphisms
on neighborhoods of the fibre critical points other than the origin.

\begin{rmk}\labell{remark: sharp preserves equivalence}
If two families $F^{(1)}_t$ and $F^{(2)}_t$
are equivalent respectively to $G^{(1)}_t$ and $G^{(2)}_t$
then $F^{(1)}_t \, \sharp \, F^{(2)}_t$ is equivalent to $G^{(1)}_t \, \sharp \, G^{(2)}_t$.
\eor
\end{rmk}

Equivalent 1-parameter families of conical generating functions
generate the same contact isotopy.
In Proposition \ref{uniqueness gf} we prove
(mostly following Th\'eret's proof 
of \cite[Proposition 4.7]{Theret - Rotation numbers})
a partial converse:
any two \emph{based} families of conical generating functions
for a given contact isotopy of $L_k^{2n-1}$
are equivalent.
The main ingredient is \cite[Lemma 4.8]{Theret - Rotation numbers},
whose proof holds also in our situation
and gives the following result.

\begin{lemma} \labell{lemma 4.8 in Theret} 
Let $\{\phi_t\}_{t \in [0,1]}$ be a based contact isotopy of $L_k^{2n-1}$.
Suppose that $F_{s,t} \colon \mathbb{R}^{2n} \times \R^{2nN} \rightarrow \R$,
for $(s,t) \in [0,1] \times [0,1]$, 
is a family of conical generating functions
such that for every $(s,t)$ 
the function $F_{s,t}$ is a conical generating function for $\phi_t$.
Then there is a family $\theta_{s,t}$
of fibre preserving conical homeomorphisms
of $\mathbb{R}^{2n} \times \R^{2nN}$
such that $F_{s,t} \circ \, \theta_{s,t} = F_{0,t}$,
and such that each homeomorphism $\theta_{s,t}$
restricts to a diffeomorphism
from a neighborhood of the set of fibre critical points of $F_{0,t}$ other than the origin
to a neighborhood of the set of fibre critical points of $F_{s,t}$ other than the origin.
In particular, the 1-parameter families $F_{0,t}$ and $F_{1,t}$ are equivalent.
\end{lemma}

\begin{proof}
Write $E := \R^{2n} \times \R^{2nN}$
and $\calE = [0,1] \times [0,1] \times E$.
We seek a locally Lipschitz two-parameter family
of homogeneous of degree $1$, $\zk$-invariant,
vertical vector fields $X_{s,t}$ on $E \smallsetminus \{ (0,0) \}$
so that
\begin{equation}\labell{eqn to solve for uniqueness}
\frac{\del F_{s,t}}{\del s} \, (q,\zeta) + \frac{\del F_{s,t}}{\del \zeta} \,  (q,\zeta) \, X_{s,t} (q,\zeta) = 0 \,.
\end{equation}
Assuming we have such a family of vector fields,
for each  $t$  the local flow
of the time-dependent vector field $s \mapsto X_{s, t} (q, \zeta)$ on $E \smallsetminus \{ (0,0) \}$ exists
thanks to the locally Lipschitz assumption.
By homogeneity, the local flows fit into a family $s \mapsto \theta_{s, t} (q, \zeta)$
of global flows,
which we extend to a family of fibre preserving conical homeomorphisms
by demanding that $\theta_{s, t} (0,0) = (0,0)$.
The relation \eqref{eqn to solve for uniqueness} then implies that
$\frac{\del}{\del s} \, F_{s,t} \big(\theta_{s,t} (q,\zeta)\big) = 0$,
and so $F_{s,t} \big(\theta_{s,t} (q,\zeta)\big) = F_{0,t} (q,\zeta)$.

Writing $F: \calE \to \R$, $F (s,t,q,\zeta) = F_{s,t} (q,\zeta)$,
let
$$
\Sigma = \Big\{ \frac{\del F}{\del \zeta} = 0 \Big\}
= \bigsqcup_{s,t} \; \{ s \} \times \{ t \} \times \Sigma_{F_{s,t}} \,.
$$
Moreover denote
$$
\calS = \big\{ (s,t,q,\zeta) \in \calE \,\big\lvert\,  ||q||^2 + ||\zeta||^2 = 1 \big\}
$$
and
$$
\calE_{\neq 0} = [0,1] \times [0,1] \times \left( E \ssminus \{ (0,0) \} \right) \,.
$$
Notice that the subset $\Sigma$ of $\calE$ is $\R_{>0}$-invariant,
and its intersection with $\calE_{\neq 0}$ is a smooth submanifold
(as it is the preimage of a regular value
of $\frac{\del F}{\del \zeta} \colon \calE_{\neq 0} \to \R^{2nN}$).

The vector field $X$ is constructed following the steps below.
 
\textit{Step $1$: Solving \eqref{eqn to solve for uniqueness}
in a neighborhood of $\calS \smallsetminus \calS \cap \Sigma$.}
In a neighborhood of $\calS \smallsetminus \calS \cap \Sigma$
the family of vertical vector fields
$$
X^1_{s,t} (q,\zeta) =
\frac{- \frac{\del F_{s,t}}{\del s} (q,\zeta)}{|\frac{\del F_{s,t} } {\del \zeta}|^2}
\; \sum_k \frac{\del F_{s,t}}{\del \zeta_k}  (q,\zeta) \, \frac{\del }{\del \zeta_k}
$$
is well defined
and solves \eqref{eqn to solve for uniqueness}.
It is locally Lipschitz,
because so are $\frac{\del F_{s,t}}{\del s}$ and $\frac{\del F_{s,t}}{\del \zeta}$.

\textit{Step $2$: Solving \eqref{eqn to solve for uniqueness}
in a neighborhood of $\calS \cap \Sigma$ in $\calS$.}
Let $U$ be an open neighborhood of $\calS \cap \Sigma $ in $\calE$
on which $F$ is smooth.
Let $G_1$ denote the restriction of $\frac{\del F}{\del \zeta}$ to $U$,
so $U \cap \Sigma = G_1^{-1} (0)$.
The set $G_1^{-1} (0)$ consists only of regular points of $G_1$. 
Let $G_2 \colon U \rightarrow \R$ denote
the restriction of  $-\frac{\del F}{\del s}$ to $U$.
Then $G_2$ is smooth on $U$ and we claim that it vanishes on $\Sigma$.

Indeed, since the image of the family of Lagrangian embeddings $\Sigma_{F_{s,t}} \to T^*\R^{2n}$ is independent
of $s$, around each point $(s,t,q,\zeta) \in \Sigma \cap \calE_{\neq 0}$, we may pick coordinates
$(s,t,x,y)$ such that $\Sigma=\{(s,t,x,y) \colon y=0\}$ and $x=x(t,q,\zeta)$ is independent of $s$
(we take $y= \frac{\partial F}{\partial \zeta}$). 

As the Lagrangian embedding of $\Sigma_{F_{s,t}}$ is given by the expression $(q,\zeta) \mapsto (q, \frac{\partial F_{s,t}}{\partial q}) 
\in T^*(\R^{2n})$ it follows that $\frac{\partial}{\partial s} \left( \frac{\partial F}{\partial q}(s,t,x,0) \right) = 0$.
Clearly we also have
$\frac{\partial}{\partial s} \left( \frac{\partial F}{\partial \zeta}(s,t,x,0)\right) = 0$.
It is easy to check that the operator $\frac{\partial}{\partial s}$ is the same in both systems of coordinates,
hence $\frac{\partial F}{\partial s}$ has vanishing derivatives with respect to $q$ and $\zeta$ on $\Sigma \cap \calE_{\neq 0}$. 
As $\frac{\partial F}{\partial s}$ is homogeneous of degree $2$ in $(q,\zeta)$ it must vanish on $\Sigma$, as required.

Hadamard's Lemma (see \cite[Th\'eor\`{e}me 92]{Theret - Thesis}) now provides a smooth $B \colon U \rightarrow \R^{2nN}$
such that in some neighborhood of $U \cap \Sigma$ we have $G_2 = \langle B , G_1 \rangle$.
We use the coordinates of the map $B$ to define a vertical vector field 
$$
X^2 = \sum_{k=1}^{2nN} B_k \, \frac{\del}{\del \zeta_k}
$$
on the above neighborhood of $U \cap \Sigma$.
Then
$$
- \frac{\del F}{\del s} = G_2 =\langle B, G_1 \rangle
= \sum_{k=1}^{2nN} X^2_k \, \frac{\del F}{\del \zeta_k} \,,
$$
and so \eqref{eqn to solve for uniqueness} holds.
Note that this vector field is smooth,
thus also locally Lipschitz.
 
\textit{Step $3$: Patch these vector fields together}.
We use a smooth bump function
to obtain a locally Lipschitz, vertical vector field
that is defined in a neighborhood of $\calS$
and solves \eqref{eqn to solve for uniqueness}.

\textit{Step $4$: Average and extend.}
Average by the $\Z_k$-action
to obtain a $\zk$-invariant, locally Lipschitz, vertical vector field $\widetilde{X}$.
Then restrict $\widetilde{X}$ to $\calS$
and extend to the whole $\calE_{\neq 0}$ by homogeneity,
i.e.\ define
$$
X (s,t,q,\zeta) = |(q,\zeta)| \, \widetilde{X} \Big(s,t, \frac{q}{|(q,\zeta)|}, \frac{\zeta}{|(q,\zeta)|} \Big)
$$
for $(q,\zeta) \neq 0$ and $X(s,t,0,0) = 0$.
Then $X$ is homogeneous of degree 1, $\zk$-invariant,
vertical and locally Lipschitz.
Moreover, its flow restricts to a diffeomorphism
from a neighborhood of the set of fibre critical points of $F_{0,t}$ other than the origin
to a neighborhood of the set of fibre critical points of $F_{s,t}$ other than the origin.
\end{proof}

Using Lemma \ref{lemma 4.8 in Theret} we now prove the following statement.

\begin{lemma}[$0$ is neutral]\labell{lemma: G sharp 0 equivalent to G}
Let $F_t \colon \mathbb{R}^{2n} \times \R^{2nN} \rightarrow \R$
be a family of conical generating functions for a contact isotopy $\{\phi_t\}$,
and let $0: \mathbb{R}^{2n} \rightarrow \mathbb{R}$ be the zero function.
Then $F_t \, \sharp \, 0$ and $0 \, \sharp \, F_t$ are equivalent to $F_t$.
\end{lemma}

\begin{proof}
We prove that $F_t \, \sharp \, 0$ is equivalent to $F_t$
(the case of $0 \, \sharp \, F_t$ is similar). 
For each $s \in [0,1]$ we define a 1-parameter family
$F_{s,t} \colon \R^{2n} \times \big(\R^{2n} \times \R^{2n} \times \R^{2nN} \big) \rightarrow \R$ by 
$$
F_{s,t} (q;\zeta_1,\zeta_2,\nu) = F_t \big(s \zeta_1 + (1-s) q, \nu\big) - 2 \left< \zeta_2 - q, i(\zeta_1 - q)\right>\,.
$$
Then $F_{1,t} = F_t \, \sharp \, 0$, 
and $F_{0,t} (q; \zeta_1,\zeta_2,\nu) = F_t(q,\nu) - 2 \left< \zeta_2 - q, i(\zeta_1 - q) \right >$
is equivalent to $F_t$ since it differs from it
by a stabilization followed by the fibre preserving conical homeomorphism
$(q; \zeta_1,\zeta_2,\nu) \mapsto (q; \zeta_1 - q,\zeta_2 - q,\nu)$.
By arguments as in the proof of Proposition \ref{p:composition formula}
one sees that for each $s\in [0,1]$
the 1-parameter family $F_{s,t}$ generates $\{\phi_t\}$.
By Lemma \ref{lemma 4.8 in Theret} we conclude that $F_{0,t}$ and $F_{1,t}$ are equivalent,
and thus so are $F_t$ and $F_t \, \sharp \, 0$.
\end{proof}

We now put these ingredients together
to prove the following result.

\begin{prop}[Uniqueness of generating functions]\labell{uniqueness gf}
Any two based families of conical generating functions
for the same based contact isotopy of $L_k^{2n-1}$
are equivalent.
\end{prop}

\begin{proof}
Let $F_t$ and $G_t$ be two based families
of conical generating functions for a based contact isotopy $\{\phi_t\}$.
By Proposition \ref{proposition: conical family},
for every $s \in [0,1]$ the family $F_{st} \, \sharp \, \big( (-G_{st}) \, \sharp \, G_t\big)$
generates $\{\phi_t\} = \{ (\phi_t \circ \phi_{st}^{-1}) \circ \phi_{st}\}$,
and the family $(-G_{st}) \, \sharp \, G_{st}$
generates $\{ \text{id} \} = \{ \phi_{st} \circ \phi_{st}^{-1} \}$.
By Lemma \ref{lemma 4.8 in Theret}
we thus have that
$F_{0} \, \sharp \, \big( (-G_{0}) \, \sharp \, G_t\big)$ is equivalent to $F_{t} \, \sharp \, \big( (-G_{t}) \, \sharp \, G_t\big)$
and that $(-G_{0}) \, \sharp \, G_{0}$ is equivalent to $(-G_{t}) \, \sharp \, G_{t}$.
Denoting equivalence of families of generating functions by $\sim$,
Remark \ref{remark: sharp preserves equivalence},
Lemma \ref{lemma: G sharp 0 equivalent to G},
Example \ref{example: 0 sharp 0}
and the fact that the families $F_t$ and $G_t$ are based
then imply
$$
G_t \, \sim \, 0 \, \sharp \, (0\, \sharp \, G_t) \, \sim \, F_0 \, \sharp \, \big((-G_0)\, \sharp \, G_t\big)
\, \sim \, F_t \, \sharp \, \big((-G_t)\, \sharp \, G_t \big) \, \sim \, F_t \, \sharp \, \big((-G_0)\, \sharp \, G_0 \big)
 \, \sim \,  F_t \, \sharp \, 0 \, \sim \,  F_t\,.
$$
\end{proof}

The next result is used in Section \ref{sec:Maslov index}
to obtain that the non-linear Maslov index descends to a map 
on the universal cover of the identity component of the contactomorphism group.

\begin{Proposition}[Uniqueness for a homotopy class] \labell{uniqueness of gf for homotopy class}
Suppose that $\{\phi_{0,t}\}$ and $\{\phi_{1,t}\}$ are two contact isotopies
from the identity to the same contactomorphism $\phi$
that are smoothly homotopic with fixed endpoints.
Then there are based families $F_{0,t}$ and $F_{1,t}$
of conical generating functions for $\phi_{0,t}$ and $\phi_{1,t}$ respectively
such that $F_{0,0} = F_{1,0}$
and $F_{0,1}$ and $F_{1,1}$ differ by a fibre preserving conical homeomorphism.
\end{Proposition}

\begin{proof}
Let $\{\phi_{s,t}\}$ be a smooth homotopy with fixed endpoints 
from $\{\phi_{0,t}\}$ to $\{\phi_{1,t}\}$,
and for $N$ big enough write
$\phi_{s,t} = \phi_{s,t}^{(N)} \circ \cdots \circ \phi_{s,t}^{(1)}$,
with each family $\{\phi_{s,t}^{(j)}\}_{s \in [0,1]}$ depending smoothly on $s$ and $t$,
for based contact isotopies $\{\phi_{s,t}^{(j)}\}_{[0,1]}$
having families of conical generating function $F_{s,t}^{(j)} \colon \R^{2n} \to \R$.
(For instance, we can take
$\phi_{s,t}^{(j)} = \phi_{s,\frac{j}{N}t} \circ (\phi_{s,\frac{j-1}{N}t})^{-1}$.)
For every $s$ consider the based family
$F_{s,t} = F_{s,t}^{(1)} \, \sharp \, \cdots \, \sharp \, F_{s,t}^{(N)}$.
Then $F_{s,0} = 0 \, \sharp \, \cdots \, \sharp \, 0$
and $F_{s,1}$ generates $\phi$.
In particular $F_{0,0} = F_{1,0}$ and,
by Lemma \ref{lemma 4.8 in Theret},
$F_{0,1}$ and $F_{1,1}$ differ by 
a fibre preserving conical homeomorphism.
\end{proof}

Let $\{\phi_t\}_{t\in[0,1]}$ be a contact isotopy of $L_k^{2n-1}$, and
consider a (based) family 
$F_t \colon \R^{2n} \times \R^{2nN} \rightarrow \R$
of conical generating functions. 
Since all $F_t$ are homogeneous 
of degree 2, they are determined by their restrictions 
$\overline{f_t} \colon S^{2n (N+1) - 1} \rightarrow \R$. 
Moreover, as the $F_t$ are invariant by the $\Z_k$-action,
the $\overline{f_t}$ are also $\Z_k$-invariant
and so they descend to a family of functions 
$$
f_t \colon L_k^{2n(N+1)-1} \rightarrow \R
$$
that are $\mathcal{C}^1$ with Lipschitz differential.
We also say that $f_t$ (as well as the corresponding $F_t$)
is a (based) family of generating functions
for the contact isotopy $\{\phi_t\}$.

\begin{prop}
\labell{correspondence discriminant pts - critical pts}
Let $f_t \colon L_k^{2n(N+1)-1} \rightarrow \R$
be a family of generating functions
for a based contact isotopy $\{\phi_t\}$ of $L_k^{2n-1}$.
Then for every $t$ there is a 1--1 correspondence between 
discriminant points of $\phi_t$
whose preimage in $S^{2n-1}$ is a $\mathbb{Z}_k$-orbit
of discriminant points for the lift
and critical points of $f_t$ of critical value zero,
that takes non-degenerate discriminant points to non-degenerate critical points,
and that is given by the restriction of a map
$L_k^{2n-1} \to L_k^{2n(N+1)-1}$
that is isotopic to the standard inclusion $[z] \mapsto [z,0]$.
\end{prop}

\begin{proof}
For every $t$, denote by $F_t$
the conical function on $\R^{2n(N+1)}$ inducing $f_t$.
Since $F_t$ is homogeneous of degree 2 and $\zk$-invariant,
its critical points come in ($\zk \times \R_{>0}$)-orbits
and all have critical value zero.
Such ($\zk \times\R_{>0}$)-orbits
are in 1--1 correspondence with 
critical points  of $f_t$ of critical value zero.
On the other hand, discriminant points of $\phi_t$
that are also discriminant points of the lift to $S^{2n-1}$
correspond to ($\zk \times \R_{>0}$)-orbits of fixed points of $\Phi_t$. 
We then use the fact that
(non-degenerate) discriminant points of $\phi_t$
correspond to (non-degenerate) critical points of $F_t$
(see for example \cite[Lemma 3.5]{S - Morse estimate for translated points}).
We now show that the 1--1 correspondence is given by the restriction
of a map $L_k^{2n-1} \to L_k^{2n(N+1)-1}$
that is isotopic to the standard inclusion.
Recall that, for every $t \in [0,1]$, the conical map
$$
\R^{2n} \times \R^{2nN} \rightarrow T^{\ast}\R^{2n} \, , \;
(\zeta,\nu) \mapsto \big( \zeta, \frac{\partial F_t}{\partial \zeta} (\zeta,\nu)\big)
$$
restricts to a homeomorphism $i_{F_t}$
between the set $\Sigma_{F_{t}}$
of fibre critical points and $\Gamma_{\Phi_{t}}$.
Fix now $t \in [0,1]$.
The required isotopy is the map induced
on the quotient by the ($\zk \times \R_{>0}$)-action
by the composition
$$
[0,t] \times \R^{2n} \to \bigcup_{s \in [0,t]} \{s\} \times \Gamma_{\Phi_{s}}
\to \bigcup_{s \in [0,t]} \{s\} \times \Sigma_{F_{s}}
\subset [0,t] \times \R^{2n} \times \R^{2nN}
$$
$$
(s,\zeta) \mapsto \left(s, \tau \left(\zeta,\Phi_s(\zeta) \right) \right)
\mapsto \left( s, (i_{F_{s}})^{-1} \left(\tau \left(\zeta,\Phi_s(\zeta)\right)\right) \right) \,.
$$
\end{proof}

\subsection*{Monotonicity and quasi-additivity of generating functions.}

We start with the following monotonicity property for generating function,
which is used to show that lens spaces are orderable.

\begin{prop}[Monotonicity of generating functions]
\labell{proposition: monotonicity of generating functions}
Let $\{\phi_t\}_{t \in [0,1]}$ be a based contact isotopy of $L_k^{2n-1}$.
\begin{enumerate}
\item[(i)] Assume that $\{\phi_t\}_{t \in [0,1]}$ has a family $f_t: L_k^{2n-1} \rightarrow \mathbb{R}$
of generating functions without fibre variable,
and that $\{\phi_t\}_{t \in [0,1]}$ is non-negative
(respectively non-positive, positive, negative).
Then $\frac{\partial f_t}{\partial t} \geq 0$
(respectively $\frac{\partial f_t}{\partial t} \leq 0$,
$\frac{\partial f_t}{\partial t} > 0$, $\frac{\partial f_t}{\partial t} < 0$)
for all $t \in [0,1]$.
\item[(ii)] For a general $\{\phi_t\}_{t \in [0,1]}$,
assume that, for a subinterval $[a,b]$ of $[0,1]$,
$\{\phi_t\}_{t \in [a,b]}$ is non-negative
(respectively non-positive).
Then there exist
\begin{itemize}
\item[-] a based contact isotopy $\{\psi_t\}_{t \in [0,1]}$
representing the same element of $\widetilde{\Cont}_0 (L_k^{2n-1})$ as $\{\phi_t\}_{t \in [0,1]}$
and such that the restriction of $\{\phi_t\}$ and $\{\psi_t\}$
to $[0,a]$, $[a,b]$ and $[b,1]$
are homotopic with fixed endpoints
and $\{\psi_t\}_{t \in [a,b]}$ is non-negative
(respectively non-positive);
\item[-] a based family $f_t \colon L_k^{2M - 1} \rightarrow \R$
of generating functions for $\{\psi_t\}_{ t \in [0,1]}$
such that $\frac{\partial f_t}{\partial t} \geq 0$
(respectively $\frac{\partial f_t}{\partial t} \leq 0$)
for all $t \in [a,b]$.
\end{itemize}
\end{enumerate}
\end{prop}

As in \cite[Lemma 3.6]{S - Morse estimate for translated points},
the main ingredient for the proof is the classical Hamilton--Jacobi equation
for generating functions.

\begin{lemma}[Hamilton--Jacobi equation]\label{lemma: HJ}
Let $F_t: B \rightarrow \mathbb{R}$ be a family of functions
with $F_0 \equiv 0$.
For each $t$, denote by $L_t \subset T^{\ast}B$
the graph of the differential of $F_t$.
Consider a Hamiltonian isotopy $\{\varphi_t\}$ of $T^{\ast}B$,
generated by a Hamiltonian function $H_t: T^{\ast}B \rightarrow \mathbb{R}$,
such that, for every $t$, $L_t$ is the image of the zero section by $\varphi_t$.
For $x \in B$ and $t \in [0,1]$ let $\big( q_t (x), p_t (x) \big) = \varphi_t (x,0)$.
Then
$$
H_t \big( q_t (x), p_t (x) \big) = \frac{\partial F_t}{\partial t} \, \big(q_t (x)\big) + c(t)
$$
for some function of time $c(t)$.
\end{lemma}

\begin{proof}
We have
\begin{align*}
\frac{\partial}{\partial x} \; H_t \big(q_t(x),p_t(x)\big) &=
\frac{\partial H_t}{\partial q} \, \frac{\partial q_t}{\partial x} + \frac{\partial H_t}{\partial p} \, \frac{\partial p_t}{\partial x}
= \dot{p}_t \, \frac{\partial q_t}{\partial x} - \dot{q}_t \, \frac{\partial p_t}{\partial x} \\
&= \Big( \frac{\partial^2 F_t}{\partial t \partial q} + \frac{\partial^2 F_t}{\partial q^2} \, \dot{q}_t \Big) \, \frac{\partial q_t}{\partial x}
- \dot{q}_t \; \Big( \frac{\partial^2 F_t}{\partial q^2} \frac{\partial q_t}{\partial x} \Big)\\
&= \frac{\partial^2 F_t}{\partial q \partial t} \, \frac{\partial q_t}{\partial x}
= \frac{\partial}{\partial x} \; \Big(\frac{\partial F_t}{\partial t} \big(q_t(x)\big)\Big)
\end{align*}
where the second equality follows from Hamilton's equations for $H_t$,
and the third from differentiating the relation
$\frac{\partial F_t}{\partial q} \big(q_t(x)\big) = p_t(x)$
with respect to $t$ and with respect to $x$.
\end{proof}

The proof of Proposition \ref{proposition: monotonicity of generating functions}
also uses concatenation of contact isotopies.
The following remark allows us to represent elements of $\widetilde{\Cont}_0 (L_k^{2n-1})$
by concatenations.

\begin{rmk}[Concatenation]
Let $\{\phi_t\}_{t \in [0,1]}$ and $\{\psi_t\}_{t \in [0,1]}$ be based contact isotopies.
For any smooth map $\rho: [0,1] \rightarrow [0,1]$
such that $\rho (0) = 0$ and $\rho (1) = 1$
and whose derivatives of orders $\geq 1$ all vanish at the endpoints $0$ and $1$,
the concatenation 
$\{\phi_{\rho(t)}\}_{t \in [0,1]} \sqcup \{\psi_{\rho(t)} \circ \phi_{\rho(1)}\}_{t \in [0,1]}$
is smooth
and it represents the same element of $\widetilde{\Cont}_0 (L_k^{2n-1})$
as the composition $\{ \psi_t \circ \phi_t \}_{t \in [0,1]}$.
We can choose $\rho$ to be strictly monotone.
Then,
for any closed subinterval $I$ of $[0,1]$,
if $\{\phi_t\}_{t \in I}$ is respectively embedded, non-negative or non-positive
then so is $\{\phi_{\rho(t)}\}_{t \in \rho^{-1}(I)}$.
\eor
\end{rmk}

\begin{proof}[Proof of Proposition \ref{proposition: monotonicity of generating functions}]
Suppose first that $\{\phi_t\}$ has a family $f_t: L_k^{2n-1} \rightarrow \mathbb{R}$
of generating functions without fibre variables.
Let $\{\Phi_t\}$ be the corresponding conical Hamiltonian isotopy of $\mathbb{R}^{2n}$,
and consider the Hamiltonian isotopy
$\{\tau \circ (\text{id} \times \Phi_t) \circ \tau^{-1}\}$ of $T^{\ast}\mathbb{R}^{2n}$,
where as usual $\tau$ denotes the identification \eqref{e:identification}.
For each $t$,
the image of the zero section by $\tau \circ (\text{id} \times \Phi_t) \circ \tau^{-1}$
is the Lagrangian $\Gamma_{\Phi_t} = \tau \big(\text{gr} (\Phi_t) \big)$.
If the lift of $\{\phi_t\}$ to $S^{2n-1}$ is generated
by the contact Hamiltonian function $h_t: S^{2n-1} \rightarrow \mathbb{R}$
then on $\R^{2n} \smallsetminus \{0\}$
the isotopy $\{\Phi_t\}$
is generated by the Hamiltonian function $H_t: \mathbb{R}^{2n} \rightarrow \mathbb{R}$
defined by $H_t (z) = \lvert z \rvert^2 \, h_t \big(\frac{z}{\lvert z \rvert}\big)$,
and on $\tau \, \big(\overline{\R^{2n}} \times (\mathbb{R}^{2n} \smallsetminus \{0\})\big)$
the isotopy $\{\tau \circ (\text{id} \times \Phi_t) \circ \tau^{-1}\}$
is generated by the Hamiltonian function $\overline{H}_t \circ \tau^{-1}$,
where $\overline{H}_t: \overline{\mathbb{R}^{2n}} \times \mathbb{R}^{2n} \rightarrow \mathbb{R}$
is defined by $\overline{H}_t (z,z') = H_t(z')$.
Thus, if the contact isotopy $\{\phi_t\}$ is non-negative
(respectively non-positive, positive, negative)
then
$$
\overline{H}_t \big( z, \Phi_t(z) \big)= H_t \big( \Phi_t (z) \big) \geq 0
$$
(respectively $\leq 0$, $> 0$, $< 0$) for $z \neq 0$.
Since all our functions are conical
the Hamilton--Jacobi equation of Lemma \ref{lemma: HJ} holds with $c(t) \equiv 0$.
We thus obtain
$$
\frac{\partial F_t}{\partial t} \Big( \frac{z + \Phi_t (z)}{2} \Big)
= \overline{H}_t \big(z, \Phi_t(z) \big) \geq 0
$$
(respectively $\leq 0$, $> 0$, $< 0$) for $z \neq 0$,
where $F_t$ is the conical function induced by $f_t$.
Since $z \mapsto \frac{z + \Phi_t(z)}{2}$ is onto $\mathbb{R}^{2n}$,
we conclude that $\frac{\partial f_t}{\partial t} \geq 0$
(respectively $\leq 0$, $> 0$, $< 0$).
This proves (i).

In order to prove (ii),
take a sufficiently fine partition $0 = t_0 < t_1 < \cdots < t_N = 1$
such that $\{ t_0,\cdots,t_N\}$ contains the endpoints of the interval $I$.
For $j = 1$, $\cdots$, $N$ let $\psi_t^{(j)}$ be a smooth reparametrization of
$$
 \widehat{\psi_t}^{(j)} = \begin{cases}
 \text{id} & \text{ if } t \in [0,t_{j-1}] \\
 \phi_t \circ (\phi_{t_{j-1}})^{-1} & \text{ if } t \in [t_{j-1},t_j] \\ 
 \phi_{t_j} \circ (\phi_{t_{j-1}})^{-1} & \text{ if } t \in [t_j,1]
\end{cases} 
$$
so that the $\psi_t^{(j)}$ have a family of generating functions without fibre variable
and are non-negative (respectively non-positive) for $t \in I$,
and so that $\{\psi_t := \psi_t^{(N)} \circ \cdots \circ \psi_t^{(1)}\}$
is a reparametrization of $\{\phi_t\}$.
We then conclude using (i)
and the fact that if we apply the composition formula
(Proposition \ref{p:composition formula})
to two non-decreasing (respectively non-increasing) families of functions
then the resulting family is also non-decreasing
(respectively non-increasing).
\end{proof}

We now note that the generating function for a composition
that is given by Proposition \ref{p:composition formula}
agrees in codimension $2n$
with the direct sum of the generating functions of the factors.

\begin{prop}[Quasiadditivity of generating functions]
\labell{proposition: quasiadditivity of gf}
Suppose that $F_1 \colon \R^{2n} \times \R^{2nN_1} \rightarrow \R$
and $F_2 \colon \R^{2n} \times \R^{2nN_2} \rightarrow \R$
are (conical) generating functions for the (conical) symplectomorphisms 
$\Phi^{(1)}$ and $\Phi^{(2)}$ respectively.
Then there is a linear ($(\zk \times \R_{>0})$-equivariant) injection
$$
\iota \colon \R^{2n} \times \R^{2n} \times (\R^{2nN_1} \times \R^{2nN_2}) 
\rightarrow \R^{2n} \times (\R^{2n} \times \R^{2n} \times \R^{2nN_1} \times \R^{2nN_2})
$$
such that $(F_1 \, \sharp \, F_2) \circ \, \iota = F_1 \oplus F_2$.
\end{prop}

\begin{proof}
We have 
$$
F_1 \oplus F_2 (\zeta_1,\zeta_2;\nu_1,\nu_2) = F_1(\zeta_1,\nu_1) + F_2(\zeta_2,\nu_2)
$$
and 
$$
F_1 \, \sharp \, F_2 (q;\zeta_1,\zeta_2,\nu_1,\nu_2) = 
F_1(\zeta_1,\nu_1) + F_2(\zeta_2,\nu_2) - 2 \left< \zeta_2 - q, i(\zeta_1 - q) \right>\,.
$$
Thus for the injection
$\iota (\zeta_1,\zeta_2;\nu_1,\nu_2) = (\zeta_1;\zeta_1,\zeta_2,\nu_1,\nu_2)$
we have $(F_1 \, \sharp \, F_2) \circ \iota = F_1 \oplus F_2$.
\end{proof}

The above quasiadditivity property is crucial in proving
that the non-linear Maslov index is a quasimorphism
(Proposition \ref{p:quasimorphism property}).
If generating functions had been additive, 
not just quasiadditive, then at least for real projective spaces
the non-linear Maslov index would have been a homomorphism.
Since no non-trivial homomorphisms exist on contactomorphism groups
\cite{Tsuboi, Rybicki},
this shows that the lack of additivity of the composition formula
is not a technical failure but something essential.


\section{The cohomological index}
\labell{section: cohomological index}

As already outlined in the introduction,
the value of the non-linear Maslov index 
of a contact isotopy $\{\phi_t\}$ of a lens space $L_k^{2n-1}$ 
depends on the changes in topology of the sublevel sets
of a based family $f_t \colon L_k^{2M-1} \rightarrow \R$
of conical generating functions.
As in \cite{Givental - Nonlinear Maslov index} and \cite{Theret - Rotation numbers},
the topological invariant that we use to analyze these changes is the cohomological index.
In this section we review the definition of this invariant
and describe some of its properties in the case of lens spaces.
Cohomological indices have also been studied in a more general context
by Fadell and Rabinowitz \cite{FR - Generalized cohomological index theories}
(see also Remark \ref{remark: FR}).

Continuity of the cohomological index
(Proposition \ref{proposition: cohomological index} (ii) below)
is important for our applications.
Since the sets that we need to consider (sublevel sets of generating functions)
might not be locally contractible,
in order to guarantee continuity we work with \v{C}ech cohomology
(as for instance in \cite{FR - Generalized cohomological index theories}).
Note that \v{C}ech cohomology agrees with singular cohomology on 
spaces that are paracompact and locally contractible 
(see \cite[Corollary 6.8.8 and Theorem 6.9.1]{Spanier}),
in particular on manifolds or, more generally, on CW-complexes. 
Recall also that the \v{C}ech cohomology
of a compact subset $A$ of a manifold
can be computed in terms of singular cohomology as 
\begin{equation} \labell{direct limit}
\check{H}^*(A) = \varinjlim_j H^*(U_j)
\end{equation}
where $U_j$ is any decreasing sequence of open sets 
having $A$ as their intersection. 
Indeed, $\check{H}^{\ast}(A) = \varinjlim \check{H}^{\ast}(U_j)$
(see \cite[Theorem 6.6.2]{Spanier}) 
and, since the $U_j$ are open, their \v{C}ech cohomology 
agrees with their singular cohomology.

We now assume that $k$ is prime (cf.\ Remark \ref{remark: k prime}).

\begin{defi}\labell{defi cohomological index}
Let $A$ be a paracompact Hausdorff topological space
and $\pi \colon \widetilde A \to A$ a principal $\zk$-bundle
with classifying map $g \colon A \to B\zk = L_k^{\infty}$.
The \emph{cohomological index} of $\pi \colon \widetilde A \to A$
is the dimension over $\zk$ of the image of the induced map
$g^{\ast}\colon \check{H}^{\ast}(L^\infty_k;\zk) \to \check{H}^\ast(A;\zk)$.
If $A$ is a subset of a lens space\footnote{
In our applications $M$ is a multiple of $n$
and the $M$-tuple of weights on $L_k^{2M-1}$
has the form $\ul{w} = (\underline{w}', \cdots,\underline{w}')$
for an $n$-tuple of weights $\underline{w}'$ on $L_k^{2n-1}$.
However in this section $\ul{w}$ can be any tuple of weights.}
$L_k^{2M-1} (\underline{w})$ then
its cohomological index, denoted by $\ind(A)$,
is defined to be the cohomological index
of the restriction $\pi \colon \widetilde{A} \rightarrow A$
of the principal $\zk$-bundle\footnote{
We write
$S^{2M-1}(\ul{w})$ or $\C^{M}(\ul{w})$
when we wish to specify the $\zk$-action.}
$\pi \colon S^{2M-1}(\ul{w}) \rightarrow L_k^{2M-1}(\underline{w})$.
\eod
\end{defi}

We now specialize to the case when the prime $k$ is different from $2$,
leaving to the reader the task
of adapting the discussion to the (easier) case of $k=2$
(cf.\ Remarks \ref{RPM} and \ref{remark: FR}).

A principal $\zk$-bundle $\widetilde{A} \to A$
is determined by the \v{C}ech cohomology class $\alpha \in \check{H}^1(A; \zk)$
that is represented by the transition functions
for a choice of local trivializations. 
The Bockstein homomorphism
$\calB \colon \check{H}^q(A) \to \check{H}^{q+1}(A)$
(see \cite[Section 3.E]{Hatcher})
is a derivation whose square is zero, so, setting\footnote{
If $\{ \psi_{ij} \colon U_i \cap U_j \to \Z \}$ 
are lifts of the transition functions 
$\{ \alpha_{ij} \colon U_i \cap U_j \to \zk \}$,
then $\beta$ is represented by the \v{C}ech 2-cocycle
$\{ \frac{1}{k} \left( \psi_{ij} + \psi_{j\ell} + \psi_{\ell i} \right) 
                        \mod k \}$.
}
$\beta = \calB(\alpha) \in \check{H}^2(A; \zk)$,
we have
$$
\mathcal{B}(\alpha \beta^j) = \beta^{j+1} \quad \text{and} \quad 
\mathcal{B}(\beta^j) = 0 \quad \quad \text{ for all } j\geq 0 .
$$
A map of principal $\zk$-bundles $\widetilde{A} \to \widetilde{B}$ 
pulls back the classes $\alpha,\beta$ on the base~$B$ 
to the classes $\alpha,\beta$ on the base~$A$. 

\begin{lemma} \labell{no holes in index}
For any $M$-tuple of weights $\underline w$ and $0 \leq j \leq 2M-1$,
\begin{equation} \labell{cohomology lens space}
H^j\big(L_k^{2M-1}(\underline w) ; \bb{Z}_k\big) \text{ is generated by } 
\begin{cases}
\beta^i & \text{ for } j=2i \\
\alpha \beta^i & \text{ for } j=2i+1 \,.
\end{cases}
\end{equation}
For a subset $A$ of $L^{2M-1}_k (\underline{w})$,
$$
\text{\emph{ind}} (A) = \dim_{\zk} (\im \iota^*)
$$
where
$\iota^* \colon \check{H}^*\big(L_k^{2M-1}(\underline{w}); \zk\big) \to \check{H}^*(A; \zk)$
is the map on \v{C}ech cohomology
that is induced by the inclusion
$\iota \colon A \hookrightarrow L_k^{2M-1}(\underline{w})$.
Moreover,
$$ 
\im \iota^* \cap \check{H}^{j} (A;\zk) \cong
\begin{cases}
 \zk & \textrm{ if } 0 \leq j < \ind(A) \\
 0 & \textrm{ if } j \geq \ind(A).
\end{cases}
$$
\end{lemma}

\begin{proof}
For \eqref{cohomology lens space}, see \cite[Example 3E.2]{Hatcher}.
The equality $\ind (A) = \dim_{\zk} (\im \iota^*)$
follows from the facts that 
the classifying map $g(\ul{w}) \colon L^{2M-1}_k(\ul{w}) \to L^\infty_k$
induces a surjection in cohomology (by \eqref{cohomology lens space})
and that $g(\ul{w}) \circ \iota$ is a classifying map for $A$.
The ring structure
and the action of the Bockstein homomorphism
imply that if $x \in H^*\big(L^{2M-1}_{k}(\underline{w})\big)$ is
non-zero and $\iota^*(x) = 0$ then $\iota^*(y) = 0$ for all $y$
with $\deg(y)\geq \deg(x)$; this implies the last statement.
\end{proof}

It follows from Lemma \ref{no holes in index}
that if $A$ is a subset of $L_k^{2M-1} (\underline{w})$
then $0 \leq \ind(A) \leq 2M$.
Also note that $\text{ind}(\emptyset) = 0$,
and that $\ind(A) = 1$ if $A$ is finite and non-empty.

A \emph{lens subspace} of $L^{2M-1}_k (\ul{w})$ 
is the $\zk$-quotient of the intersection 
of $S^{2M-1}(\ul{w})$ with a $\zk$-invariant
complex linear subspace of $\C^{M}(\ul{w})$.
Lemma~\ref{no holes in index} implies that the cohomological index 
of a $2r-1$ dimensional lens subspace is $2r$.

\begin{Remark} \labell{RPM}
For a subset $A$ of a real projective space $\mathbb{RP}^M$,
\begin{equation}\labell{equation: RP}
\ind(A) = \min \{\, j \in \mathbb{N} \; \lvert \; \iota^{\ast} (x^{j}) = 0\,\}
\end{equation}
where $x$ is the generator of $\check{H}^1 (\mathbb{RP}^M; \zt)$.
Similarly one defines
the cohomological index
for subsets of complex projective spaces
and for principal $S^1$-bundles;
in this case the analogue of \eqref{equation: RP} holds
for $x$ a generator of degree two (cf.\ \cite{Theret - Rotation numbers}).
\end{Remark}

\begin{rmk}\labell{remark: FR}
For a compact Lie group $G$,
let $\pi \colon \widetilde{A} \rightarrow A$ be a principal $G$-bundle
over a paracompact Hausdorff topological space
with classifying map $g \colon A \rightarrow BG$.
For any non-zero class $\eta \in \check{H}^*(BG)$ Fadell and Rabinowitz
\cite{FR -   Generalized cohomological index theories}
define the $\mathit{\eta}$\emph{-index} as
the maximal $j \in \mathbb{N}$ such that $g^{\ast}(\eta^{j})\neq 0$. 
If $G = \zt$ and $\eta$ is the generator of $\check{H}^1(\RP^\infty;\zt)$ 
then the $\eta$-index just differs by $1$
from the index of Definition \ref{defi cohomological index}.
If however $G=\zk$ with $k \neq 2$
and $\beta$ is a generator of $\check{H}^2(L^\infty_k;\zk)$
then the $\beta$-index is equal to $\lfloor \frac{\ind(A)-1}{2}\rfloor$.
We use the index from Definition \ref{defi cohomological index}
rather than the $\beta$-index
in order to obtain a better bound
on the number of translated 
points (see Section \ref{section: applications}):
using the $\beta$-index we would only prove
existence of $n$ translated points on $L_k^{2n-1}$,
even in the non-degenerate case.
\eor
\end{rmk}

Given subsets $A$ of $L_k^{2M-1}(\underline{w})$ 
and $B$ of $L_k^{2M'-1}(\underline{w'})$
with preimages $\wt{A} \subset S^{2M-1}(\ul{w})$ 
and $\wt{B} \subset S^{2M'-1}(\ul{w}')$,
their $\Z_k$-\emph{join} is the subset
$$
A \ast_{\Z_k} B \subset L_k^{2(M+M')-1}(\underline{w},\underline{w'})
$$
defined by
\begin{equation} \labell{join1}
A \ast_{\zk} B = \left\{
 [ \, \sqrt{t} \, a , \sqrt{1-t} \, b \, ] \ \ | \ 
a \in \wt{A},\ b \in \wt{B},\ 0 \leq t \leq 1 \right\}
\end{equation}
if $A$ and $B$ are non-empty.
If $B$ is empty, we define $A \ast_{\zk} \emptyset$
to be the image of $A$ under the natural embedding $a \mapsto [a,0]$
of $L_k^{2M-1}(\ul{w})$ into $L_k^{2(M+M')-1}(\ul{w},\ul{w}')$.
We define $\emptyset \ast_{\zk} B$ similarly.
Finally, the $\zk$-join of the empty sets is empty.

We now describe the properties of the cohomological index
that we need for our applications.
The proofs of properties (i)-(iv)
are easy adaptations of the corresponding proofs 
in \cite{Givental - Nonlinear Maslov index, Theret - Rotation numbers,
FR - Generalized cohomological index theories}
and are included for the convenience of the reader.
In (v), the lower bound on $\ind(A \ast_\zk B)$ 
requires a more involved proof, which we postpone 
to Appendix~\ref{section: additivity under join}.

\begin{prop}\labell{proposition: cohomological index}
The cohomological index of subsets of lens spaces
has the following properties:

\begin{enumerate}
\renewcommand{\labelenumi}{(\roman{enumi})}

\item (\emph{Monotonicity}) 
If $A \subset B \subset L_k^{2M-1}(\underline{w})$ 
then $\text{\emph{ind}}(A) \leq \text{\emph{ind}}(B)$.

\item (\emph{Continuity}) 
Every closed subset $A$ of $L^{2M-1}_k(\ul{w})$
has a neighborhood $U$ such that 
if $A \subset V \subset U$ then $\ind(V) = \ind(A)$.

\item (\emph{Lefschetz property}) 
Let $A$ be a closed subset of $L_k^{2M-1}(\underline{w})$,
and let $A' = A \cap H$ where $H \subset L_k^{2M-1}(\underline{w})$ 
is a lens subspace of codimension 2.
Then $\text{\emph{ind}}(A') \geq \text{\emph{ind}}(A) - 2$.

\item (\emph{Subadditivity}) 
For closed subsets $A$ and $B$ of $L_k^{2M-1}(\underline{w})$ we have
$$
\text{\emph{ind}} (A \cup B) \leq \text{\emph{ind}} (A) + \text{\emph{ind}} (B) + 1
$$
and
$$
\text{\emph{ind}} (A \cup B) \leq \text{\emph{ind}} (A) + \text{\emph{ind}} (B)
\quad  \text{ if $\ind(A)$ is even or $\ind(B)$ is even.} 
$$

\item (\emph{Join quasi-additivity}) 
For closed subsets $A$ of $L_k^{2M-1}(\underline{w})$ 
and $B$ of $L_k^{2M'-1}(\underline{w'})$ we have
$$
\left| \, \text{\emph{ind}} (A \ast_{\mathbb{Z}_k} B) - \text{\emph{ind}} (A) - \text{\emph{ind}}(B) \, \right| \leq 1
$$
and
$$
\ind (A \ast_{\zk} B) = \ind (A) + \ind (B) 
\quad  \text{ if $\ind(A)$ is even or $\ind(B)$ is even}. 
$$
In particular (\emph{Join stability}),
$$
\ind \big(A \ast_{\mathbb{Z}_k} L_k^{2K-1}(\ul{w}')\big) 
 = \ind (A) + 2K \,.
$$
\end{enumerate}
\end{prop}

\begin{Remark}
The above properties (ii), (iii), (iv), and (v) 
are stated for closed subspaces of lens spaces
but they hold (and are proved) also for open subsets of lens spaces.
\end{Remark}

\begin{proof}[Proof Proposition \ref{proposition: cohomological index}]
\begin{enumerate}[(i)]

\item 
Let $\wt{A}$ and $\wt{B}$ be the preimages in $S^{2n-1}_k(\ul{w})$
of $A$ and $B$.
The result follows from the fact that the restriction to $A$ 
of the classifying map of $\wt{B} \to B$ is a classifying map
for $\wt{A} \to A$.

\item 
Let $x$ be a generator of
$H^{\ind(A)} \big(L^{2M-1}_k(\underline w); \zk\big)$.
Then $i_A^*(x) = 0$.
By~\eqref{direct limit}, there exists an open neighborhood $U$ of $A$
such that $i^*_U(x) = 0$, where $i_U \colon U \to M$ is the inclusion map.
By Lemma~\ref{no holes in index}, $\ind(U) \leq \ind(A)$.
By monotonicity, $\ind(U) \geq \ind(A)$.

\item 
Assume that $\ind(A) \geq 3$, otherwise the inequality is trivial.
By continuity of the cohomological index, 
there exist open neighborhoods $U$ of $A$ and $V$ of $H$ 
such that $\ind(U) = \ind(A)$, \ $\ind(V) = \ind(H)$,
and $\ind (U \cap V) = \ind (A \cap H)$.
We have the following commuting diagram, 
where $D$
denotes the Poincaré duality isomorphism
and $\bullet$ the homology intersection product (see \cite[VIII.13.5]{Dold}): 
$$
\xymatrix{
H_*(U) \otimes H_*(V) \ar[d]^{i_{U*}\otimes i_{V*}} \ar[r]^{\bullet} & H_*(U\cap V) \ar[d]^{i_{U\cap V*}} \\
H_*(L_k^{2M-1}(\underline w)) \otimes H_*(L_k^{2M-1}(\underline w)) \ar[d]^{D\otimes D} &
H_*(L_k^{2M-1}(\underline w)) \ar[d]^D \\
H^*(L_k^{2M-1}(\underline w)) \otimes H^*(L_k^{2M-1}(\underline w))
\ar[r]^{\quad \quad \quad \cup} & H^*(L_k^{2M-1}(\underline w))\,.
}
$$
Let $x$ be a class in $H_{\ind(A)-1}(U)$ such that $i_{U*}(x)\neq 0$
(this exists, as homology and cohomology with field coefficients
are dually paired)
and similarly let $y$ be a class in $H_{2M-3}(V)$ with $i_{V*}(y) \neq 0$.
Since $D \big(i_{U*}(x)\big)$ is a non-zero class in 
$H^{\leq 2M - 3}\big( L_k^{2M-1}(\ul{w}) \big)$ 
and $D \big(i_{V*}(y)\big)$ is a non-zero multiple of $\beta$,
we have
$D \big(i_{U*}(x)\big) \cup D \big(i_{V*}(y)\big) \neq 0$.
It follows that $i_{U\cap V*}(x \bullet y) \neq 0$, 
which shows that $ \ind(A\cap H) = \ind(U\cap V) \geq \ind(A)-2$.

\item 
Assume that $\ind(A) + \ind(B) < 2M$, otherwise the inequality is trivial.
By continuity, there exist open neighborhoods $U$ of $A$
and $V$ of $B$ such that $\ind(U) = \ind(A)$, $\ind(V) = \ind(B)$,
and $\ind(U \cup V) = \ind (A \cup B)$.
By the exact cohomology sequence of the pair
$$
\cdots \rightarrow H^{*}(L_k^{2M-1}(\underline w),U)
\stackrel{j_U^{\phantom{U}\ast}}{\longrightarrow} H^{*}(L_k^{2M-1}(\underline w))
\stackrel{i_U^{\phantom{U}\ast}}{\longrightarrow}H^{*}(U)
\rightarrow \cdots \, .
$$
and Lemma \ref{no holes in index},
the index of $U$ is the lowest degree of
a non-zero class in the image of $j_U^{\phantom{U}\ast}$.
A similar statement holds for $V$.
Consider the commutative diagram
\begin{displaymath}
\xymatrix{
H^{*} \big(L_k^{2M-1}(\underline w),U\big) \otimes H^{*} \big(L_k^{2M-1}(\underline w),V\big)
\ar[d]^{j_{U}^{\phantom{U}\ast} \otimes j_{V}^{\phantom{U}\ast}}
\ar[r]^{\quad \quad \quad \quad \cup}
& H^{*} \big(L_k^{2M-1}(\underline w),U \cup V\big)
\ar[d]^{j_{U \cup V}^{\phantom{UV}\ast}} \\
H^{*} \big(L_k^{2M-1}(\underline w)\big) \otimes H^{\ast} \big(L_k^{2M-1}(\underline w)\big)
\ar[r]^{\quad \quad \quad  \cup}
& H^{*} \big(L_k^{2M-1}(\underline w)\big)\,. }
\end{displaymath}
Assume first that one of the indices is even,
for instance that of $A$. 
Let $x$ be a class in $H^{\ind(A)} \big(L_k^{2M-1}(\underline w),U\big)$
such that $j_U^{\phantom{U}\ast}(x)\neq 0$,
and $y$ a class in  $H^{\ind(B)} \big(L_k^{2M-1}(\underline w),V\big)$ 
such that $j_V^{\phantom{V}\ast}(y)\neq 0$.
Since $j_U^{\phantom{U}\ast}(x)$ is a non-zero multiple of $\beta^{\ind(A)/2}$  
and $\ind(A)+\ind(B) < 2M$,
it follows that 
$ j_{U\cup V}^{\phantom{U \cup V}\ast}(x \cup y) = j_U^{\phantom{U}\ast}(x) \cup j_V^{\phantom{V}\ast}(y)$
is non-zero and so
$\ind (A \cup B) \leq \ind(A) + \ind (B)$.
If both $\ind(A)$ and $\ind(B)$ are odd, 
replace $x$ in the above argument
with a class $x'$ in $H^{\ind(A)+1}(L_k^{2M-1}(\underline w),U)$ 
such that $j_U^{\phantom{U}\ast}(x')\neq 0$,
to obtain $\ind (A \cup B) \leq \ind(A) + \ind (B) + 1$. 

\item 
The subset $A' = (A\ast_{\zk} B) \ssminus B$ deformation retracts to $A$,
and the subset $B' = (A \ast_{\zk} B) \ssminus A$ deformation retracts to $B$.
Since $A\ast_{\zk} B = A' \cup B'$, 
the subadditivity property~(iv) implies that
$\ind (A \ast_{\mathbb{Z}_k} B) \leq \text{ind} (A) + \text{ind}(B) + 1$
and $\text{ind} (A \ast_{\mathbb{Z}_k} B) \leq \text{ind} (A) + \text{ind}(B)$
if at least one of the indices is even.
The reverse inequalities are proved
in Appendix~\ref{section: additivity under join}.
\end{enumerate}
\end{proof}

\begin{rmk}\labell{remark: stronger properties for RP}
In the case of real projective spaces
the cohomology ring is generated by the generator in degree one,
and the above arguments can be adapted to show that properties (iii) and (iv) 
of Proposition \ref{proposition: cohomological index}
hold in the following stronger form:
\begin{itemize}
\item[(iii')] 
If $A$ is a closed subset of $\RP^{M}$,
and $A' = A \cap H$ where $H \subset \RP^M$ is a real projective subspace
of codimension one, then $\text{ind}(A') \geq \text{ind}(A) - 1$.
\item[(iv')] 
For closed subsets $A$ and $B$ of $\RP^{M}$ we have
$$
\text{ind}(A \cup B) \leq \text{ind}(A) + \text{ind}(B)\,.
$$
\end{itemize}
Moreover, we also have the following result
(see Remark \ref{remark: RP} or \cite{Givental - Nonlinear Maslov index}):
\begin{itemize}
\item[(v')] (\emph{Join additivity}) 
For closed subsets $A$ of $\RP^{M}$ and $B$ of $\RP^{M'}$ we have
$$
\text{ind} (A \ast_{\zt} B) = \text{ind} (A) + \text{ind} (B) \,.
$$
\end{itemize}
Analogous properties hold for the cohomological index
of subsets of complex projective spaces (see \cite{Theret - Rotation numbers}).
As we will see in Sections \ref{sec:Maslov index} 
and \ref{section: applications},
the weaker properties that we have in the case of lens spaces
still suffice to define a non-linear Maslov index 
and recover the applications we are interested in.
\eor
\end{rmk}

\begin{rmk}
Rafael Gomes and the first author \cite{GG}
have shown that for $k$ an odd prime
there exist subsets $A$ and $B$ of $L_k^{2M-1}$
such that $\ind (A \ast_{Z_k} B) = \ind (A) + \ind(B) + 1$,
so join additivity does not hold in general.
\eor
\end{rmk}

We define the \emph{index}
of a conical function $F \colon \R^{2M} \rightarrow \R$ by 
$$
\ind (F) = \ind \big(\{ f \leq 0 \}\big)
$$
where $f$ is the function on $L_k^{2M-1}(\underline{w})$ induced by $F$.

\begin{rmk}\labell{remark: ind i}
If $Q$ is a $\zk$-invariant quadratic form on $\R^{2M}$
then $\ind(Q)$ coincides with $i(Q)$,
the maximal dimension of a subspace on which $Q$
is negative semi-definite.
In particular, $\zk$-invariance implies (if $k > 2$)
that in this case $\ind(Q)$ is even.
\eor
\end{rmk}

Given functions $f$ and $g$
on $L_k^{2M-1}(\underline{w})$ and $L_k^{2M'-1}(\underline{w}')$ respectively,
we write 
$$
f \oplus g \colon L_k^{2(M+M')-1}(\underline{w},\underline{w}') \rightarrow \R
$$  
for the function induced by the sum 
$F \oplus G \colon \R^{2(M+M')}\rightarrow \R$,
where $F$ and $G$ are the conical functions
on $\R^{2M}$ and $\R^{2M'}$ associated to $f$ and $g$.

\begin{prop}
\labell{direct sum and join}
Let  $f \colon L_k^{2M-1}(\ul{w}) \to \R$ 
and $g \colon L_k^{2M'-1}(\ul{w'}) \to \R$ 
be continuous functions.  Then 
$$
\text{\emph{ind}} \big(\{ f \oplus g \leq 0\}\big) 
= \text{\emph{ind}} \big(\{f\leq 0\} \ast_{\zk} \{g\leq 0\}\big) \,.
$$
\end{prop}

\begin{proof}
By continuity (Proposition \ref{proposition: cohomological index}(ii)) 
there is a neighborhood $U$ of $\{f\leq 0\} \ast_{\zk} \{g\leq 0\}$ 
with $ \ind(U) = \ind(\{f\leq 0\} \ast_{\zk} \{g\leq 0\}) $. 
Consider the diagram
$$
\xymatrix{
\{f \leq 0\} \ast_{\zk} \{g \leq 0\}  \ar@{^(->}[r] \ar@{^(->}[d]& \{f \oplus g \leq 0 \} \ar@{-->}[dl]_r \ar@^{(->}[ddl]^j \\
U \ar@{^(->}[d] & \\
L^{2M+2M'-1}(\ul{w},\ul{w'}) . &
}
$$
By monotonicity, it suffices to show that the inclusion $j$ can be deformed 
into a map $r$ with image contained in $U$. 
This can be done as follows.
We work $\zk$-equivariantly on the preimages in $S^{M+M'-1}(\ul{w},\ul{w'})$.
To simplify the formulas,
given $x \in S^{2M-1}(\ul{w})$ and $y \in S^{2M'-1}(\ul{w'})$ 
we write $tx+(1-t)y$ for
$\big( \sqrt{t} x, \sqrt{1-t} y \big) \in S^{2M+2M'-1}(\ul{w},\ul{w'})$.
Moreover we still write $f \colon S^{2M-1}(\ul{w}) \to \R$,
\ $g \colon S^{2M'-1}(\ul{w'}) \to \R$,
and $f \oplus g \colon S^{2(M+M')-1}(\ul{w},\ul{w'}) \to \R$
for the composition of the original functions $f$, $g$, and $f \oplus g$
with the projections from spheres to lens spaces. 
With this notation we have
$$ 
(f\oplus g) \big(tx+(1-t)y\big) = tf(x) + (1-t)g(y) \,.
$$
The rough idea for constructing the map $r$ is as follows.
If a point $tx+(1-t)y$ is in $\{f \oplus g \leq 0 \}$
then at least one of $f(x)$ and $g(y)$ is non-positive.
The map $r$ will act as the identity on points $tx+(1-t)y$ 
with $f(x)$ and $g(y)$ both non-positive.
If $f(x)$ is positive, and thus $g(y)$ is negative, $r$ will move 
the point $tx+(1-t)y$ to a point $(1-s)(tx+(1-t)y) +sy$, 
with $s \in [0,1]$ big enough so that this point is 
in the chosen neighborhood $U$ of
$\{f \leq 0\} \ast_{\zk} \{g \leq 0\}$. Similarly if $g(y)$ is positive.
However, one needs to
interpolate between these deformations
in order to ensure that the resulting map is continuous.
Here are the details.
For each $\delta>0$ consider the map
$$
R_{\delta} \colon \{ f \oplus g \leq 0 \} \times [0,1] \to \{ f\oplus g \leq 0 \}
$$
defined by the expression
$$
R_{\delta}( tx+(1-t)y, s) = 
\begin{cases}
(1-s)(tx+(1-t)y)+sy & \text{ if } f(x)\geq \delta \\
(1-s\frac{f(x)}{\delta})(tx+(1-t)y) + s\frac{f(x)}{\delta} y & \text{ if } 0 \leq f(x) \leq \delta \\
tx+(1-t)y & \text{ if } f(x) \leq 0 \text{ and } g(y) \leq 0 \\
(1-s\frac{g(y)}{\delta})(tx+(1-t)y) + s\frac{g(y)}{\delta} x & \text{ if } 0 \leq g(y) \leq \delta \\
(1-s)(tx+(1-t)y)+sx & \text{ if } g(y)\geq \delta \, .
\end{cases}
$$
Thus $R_{\delta}$ moves a point $tx+(1-t)y$ such that $f(x)>0$ 
along the segment  
$$
s\mapsto(1-s) \big(tx+(1-t)y\big) +sy
$$
a portion of the way towards $y$ (note that $g(y)$ must be negative). 
To ensure continuity, the portion depends on the value of $f$ at $x$.
The pasting lemma guarantees the continuity of $R_{\delta}$. 
By continuity of $(t,x,y) \mapsto ( tx + (1-t)y )$
and compactness of its domain,
for $\delta$ small enough
the set $\{ f \oplus g \leq 0 \} \cap
     \{ tx + (1-t)y \ | \ f(x) \leq \delta \text{ or } g(y) \leq \delta \}$
is contained in $U$.
For such $\delta$,
the image of $R_{\delta}(\cdot,1)$ is contained
in the preimage of $U$ in $S^{2(M+M')-1}(\ul{w},\ul{w'})$.
We take $r$ to be the map induced by $R_\delta(\cdot,1)$ on the $\zk$-orbits.
\end{proof}

Proposition \ref{direct sum and join} and
Proposition \ref{proposition: cohomological index} (v)
imply the following result.

\begin{cor}\labell{proposition: stab additivity}
For conical functions $F \colon \R^{2M} \to \R$ and $G \colon \R^{2M'} \to \R$
we have
$$
\left\vert \, \text{\emph{ind}} (F \oplus G) - \text{\emph{ind}} (F) - \text{\emph{ind}} (G) \, \right\vert \leq 1 \,. 
$$
Moreover, if either $F$ or $G$ has even index,
in particular if either $F$ or $G$ is a $\zk$-invariant  quadratic form
(for $k > 2$),
then 
$$
\text{\emph{ind}} (F \oplus G) = \text{\emph{ind}} (F) + \text{\emph{ind}} (G) \,.
$$
\end{cor}


\section{The non-linear Maslov index}
\labell{sec:Maslov index}

Using the construction of generating functions 
given in Section \ref{sec:generating functions}
and the definition and properties of the cohomological index 
discussed in Section \ref{section: cohomological index},
we now define the non-linear Maslov index
$\mu \colon \widetilde{\text{Cont}}_0(L_k^{2n-1}) \rightarrow \mathbb{Z}$
on the universal cover of the identity component of the contactomorphism group of $L_k^{2n-1}$,
and describe the properties
that are used in the applications.

\subsection*{Definition and quasimorphism property}

As before, $L_k^{2n-1}$ denotes a lens space with any vector of weights.
The \emph{non-linear Maslov index}
of a based contact isotopy $\{\phi_t\}_{t\in[0,1]}$ of $L_k^{2n-1}$ 
is defined by
$$
\mu(\{\phi_t\}) = \text{ind}(F_0) - \text{ind}(F_1)
$$ 
where $F_t \colon \R^{2n} \times \R^{2nN} \rightarrow \R$ is 
a based family of conical generating functions for $\{\phi_t\}$.
Existence of such a family is given
by Proposition \ref{existence of generating functions};
by Proposition \ref{uniqueness gf}
and Corollary \ref{proposition: stab additivity},
$\mu$ does not depend on the choice.
Moreover, Proposition \ref{uniqueness of gf for homotopy class} implies that
$\mu(\{\phi_t\})$ only depends on the smooth homotopy class of $\{\phi_t\}$
with fixed endpoints,
and thus $\mu$ descends to a map
$$
\mu \colon \widetilde{\text{Cont}_0}(L_k^{2n-1}) \rightarrow \mathbb{Z} \, .
$$

\begin{ex}\label{example: def in small case}
We have $\mu \big( \{ \text{id}\}_{t \in [0,1]} \big) = 0$.
Note also that if a based contact isotopy $\{\phi_t\}_{t \in [0,1]}$
has a family of generating functions with no fibre variable
then $0 \leq  \mu(\{\phi_t\}) \leq 2n$;
by Proposition \ref{proposition: monotonicity of generating functions}(i),
if moreover $\{\phi_t\}$ is positive then $\mu \big(\{\phi_t\}\big) = 2n$,
and if it is non-positive then $\mu \big(\{\phi_t\}\big) = 0$.
\eoe
\end{ex}

We now prove that the non-linear Maslov index is a quasimorphism
(in the case of projective spaces,
see also \cite{Ben Simon} and \cite[Theorem 9.1]{Givental - Nonlinear Maslov index}).
We start with the following lemma.

\begin{lemma}\label{lemma: index based}
If two conical generating functions $F$ and $G$
are equivalent to the zero function then 
$$
\ind (F \,\sharp \, G) = \ind (F) + \ind (G) \,.
$$
\end{lemma}

\begin{proof}
By Remark \ref{remark: new},
there are stabilizations $F \oplus P_F$ and $G \oplus P_G$ of $F$ and $G$
that are composition by fibre preserving conical homeomorphisms
of stabilizations $0 \oplus Q_F$ and $0 \oplus Q_G$ of the zero function.
Using Corollary \ref{proposition: stab additivity} we thus have
\begin{equation}\label{equation: lemma 4}
\ind (F) + \ind (P_F) = \ind (F \oplus P_F) = \ind (0 \oplus Q_F) = 2n + \ind (Q_F)
\end{equation}
and similarly for $G$;
moreover, since
$(F \oplus P_F) \, \sharp \, (G \oplus P_G)$ and $(F \, \sharp \, G) \oplus P_F \oplus P_G$
just differ by a permutation of the homogeneous coordinates
we have
\begin{equation}\label{equation: lemma 1}
\ind \big( (F \oplus P_F) \, \sharp \, (G \oplus P_G) \big) = \ind (F \, \sharp \, G) + \ind (P_F) + \ind (P_G) \,.
\end{equation}
By Remark \ref{remark: sharp preserves equivalence 1},
$(F \oplus P_F) \, \sharp \, (G \oplus P_G)$ differs from $( 0 \oplus Q_F) \, \sharp \, (0 \oplus Q_G)$
by a fibre preserving conical homeomorphism,
and so
\begin{equation}\label{equation: lemma 2}
\ind \big( (F \oplus P_F) \, \sharp \, (G \oplus P_G) \big) = \ind \big( (0 \oplus Q_F) \, \sharp \, (0 \oplus Q_G) \big) \,.
\end{equation}
Since
\begin{equation}\label{equation: lemma 2 bis}
(0 \oplus Q_F) \, \sharp \, (0 \oplus Q_G) = ( 0 \, \sharp \, 0 ) \oplus Q_F \oplus Q_G \,,
\end{equation}
we have
\begin{equation}\label{equation: lemma 3}
\ind \big( (0 \oplus Q_F) \, \sharp \, (0 \oplus Q_G) \big) = 4n + \ind (Q_F) + \ind (Q_G) \,.
\end{equation}
By \eqref{equation: lemma 1}, \eqref{equation: lemma 2},
\eqref{equation: lemma 3} and \eqref{equation: lemma 4}
we thus have $\ind (F \,\sharp \, G) = \ind (F) + \ind (G)$.
\end{proof}

\begin{prop}[Quasimorphism property]
\labell{p:quasimorphism property}
For elements 
$[\{\phi_t\}]$ and $[\{\psi_t\}]$ of $\widetilde{\Cont}_0(L_k^{2n-1})$ 
we have
$$
\left| \, \mu\big([\{\phi_t\}]\cdot[\{\psi_t\}]\big) - \mu\big([\{\phi_t\}]\big) - \mu\big([\{\psi_t\}]\big) \, \right| \leq 2n + 1\,.
$$
\end{prop}

\begin{proof}
By Proposition \ref{proposition: conical family} and Remark \ref{remark: based},
if $F_t \colon \R^{2n} \times \R^{2nN_1} \rightarrow \R$
and $G_t \colon \R^{2n} \times \R^{2nN_2} \rightarrow \R$
are based families of conical generating functions for $\{\phi_t\}$ and $\{\psi_t\}$ respectively
then
$$
G_t \, \sharp \, F_t \colon \R^{2n} \times ( \R^{2n} \times \R^{2n} \times \R^{2nN_1} \times \R^{2nN_2}) \rightarrow \R
$$
is a based family of conical generating functions for $\{\phi_t \circ \psi_t\}$.
Since $[\{\phi_t\}] \cdot [\{\psi_t\}] = [\{\phi_t \circ \psi_t\}]$
we have
$ \mu \big([\{\phi_t\}] \cdot [\{\psi_t\}] \big) = \ind(G_0 \, \sharp \, F_0) - \ind(G_1 \, \sharp \, F_1)$
and thus
\begin{align*}
\big\vert \, \mu \big([\{\phi_t\}] \cdot [\{\psi_t\}]\big) &- \mu ([\{\phi_t\}]) - \mu ([\{\psi_t\}]) \, \big\vert \\
&= \big\vert  \, \text{ind} (G_0 \, \sharp \, F_0) - \text{ind}(G_1 \, \sharp \, F_1) 
- \text{ind} (F_0) + \text{ind}(F_1) - \text{ind} (G_0) + \text{ind}(G_1) \, \big\vert \\
&\leq \big\vert \, \text{ind} (G_0 \, \sharp \, F_0) - \text{ind} (G_0) - \text{ind}(F_0) \, \big\rvert 
+ \big\vert \, \text{ind}(G_1 \, \sharp \, F_1) - \text{ind} (G_1) - \text{ind}(F_1) \, \big\rvert \\
&= \big\vert \, \text{ind}(G_1 \, \sharp \, F_1) - \text{ind} (G_1) - \text{ind}(F_1) \, \big\rvert 
\end{align*}
where the last equality follows from Lemma \ref{lemma: index based}.
By Proposition \ref{proposition: quasiadditivity of gf},
$G_1 \, \sharp \, F_1$ coincides with $G_1 \oplus F_1$ in codimension $2n$.
Therefore, using the Lefschetz property
from Proposition \ref{proposition: cohomological index}
we get
$$
\big\vert \, \text{ind}(G_1 \, \sharp \, F_1) - \text{ind} (G_1 \oplus F_1) \big\vert \leq 2n \,.
$$
On the other hand,
by Corollary \ref{proposition: stab additivity} we have
$$
\big\vert \, \text{ind}(G_1 \oplus F_1) - \text{ind} (G_1) - \text{ind} (F_1) \, \big\vert \leq 1 \,.
$$
Thus we obtain
$$
\big\vert \, \text{ind}(G_1 \, \sharp \, F_1) - \text{ind} (G_1) - \text{ind}(F_1) \, \big\vert \leq 2n + 1
$$
and so 
$$
\big| \, \mu\big([\{\phi_t\}]\cdot[\{\psi_t\}]\big) - \mu\big([\{\phi_t\}]\big) - \mu\big([\{\psi_t\}]\big) \, \big| \leq 2n + 1 \,.
$$
\end{proof}

Recall that a quasimorphism $\nu \colon G \rightarrow \mathbb{R}$
is said to be \textit{homogeneous} if
$\nu(x^m) = m \, \nu(x)$ for all $x \in G$ and $m \in \mathbb{Z}$.
Any quasimorphism $\nu \colon G \rightarrow \mathbb{R}$ 
has an associated homogeneous quasimorphism, 
defined by
$$
\ol{\nu}(x) = \lim_{m \to \infty} \frac{\nu(x^m)}{m} 
$$
(see for instance \cite[Section 2.2.2]{scl}).
The homogeneous quasimorphism
$\ol{\mu} \colon \widetilde{\text{Cont}}_0(L_k^{2n-1}) \rightarrow \mathbb{R}$
associated to the non-linear Maslov index
is called the \emph{asymptotic non-linear Maslov index}\footnote{
In \cite[Section 9]{Givental - Nonlinear Maslov index}
the asymptotic non-linear Maslov index
of a contact isotopy $\{\phi_t\}_{t \in [0,\infty)}$ starting at the identity 
is defined as
$\overline{\mu}_{\Giv} (\{\phi_t\}_{t \in [0,\infty)}) = 
\lim_{T \to \infty}  \frac{\mu (\{\phi_t\}_{t \in [0,T]})}{T}$.
Given a contact isotopy $\{\phi_t\}_{t \in [0,1]}$ we can extend it 
to a contact isotopy defined for $t \in [0,\infty)$
by posing, for $t = l+s$ with $l \in \mathbb{N}$ and $s \in (0,1)$,
$\phi_t = \phi_s \circ (\phi_1)^l$.
Then
$\overline{\mu}_{\Giv} \big(\{\phi_t\}_{t \in [0,\infty)}\big) = \ol{\mu} \big(\{\phi_t\}_{t\in [0,1]}\big)$.}.

\subsection*{The linear case}

We now show that
if the lift to $\R^{2n}$ of a based contact isotopy $\{\phi_t\}$ of $L_k^{2n-1}$
is a loop $\{\Phi_t\}$ in $\Sp(2n;\R)$
then $\mu(\{\phi_t\})$ is equal to the linear Maslov index of $\{\Phi_t\}$.
If $\{\Phi_t\}$ is a path in $\Sp(2n;\R)$ starting at the identity
then the construction of Section \ref{sec:generating functions}
gives a family
$Q_t \colon \R^{2n} \times \R^{2nN} \rightarrow \R$
of generating functions so that
each $Q_t$ is a $\zk$-invariant quadratic form.
As in \cite{Theret - Thesis, Theret - camel}, we define
(cf.\ Remark \ref{remark: ind i})
$$
\nu(\{\Phi_t\}) = \ind (Q_0) - \ind (Q_1) = i(Q_0) - i(Q_1)
$$
where $i$ denotes the maximal dimension of a subspace
on which a quadratic form is negative semi-definite.
By the arguments in Section \ref{sec:generating functions},
the integer $\nu(\{\Phi_t\})$ is well defined
and depends only on the homotopy class 
(with fixed endpoints) of the path $\{\Phi_t\}$.
Moreover, if $\{\Phi_t\}$ is the lift of a contact isotopy $\{\phi_t\}$ of $L^{2n-1}_k$ 
then $\nu( \{\Phi_t\}) =\mu(\{\phi_t\})$.

\begin{prop}\labell{proposition: linear}
The induced map
$\nu \colon \pi_1 \big(\Sp(2n;\R)\big) \to \Z$
is a group homomorphism,
and agrees with the linear Maslov index.
\end{prop}

The proof of this result is based on the following lemma,
which is taken from \cite[Proposition 35]{Theret - Thesis}
and whose equivariant version
is also used in Section \ref{section: applications}
to prove the contact Arnold conjecture.

\begin{lemma}\labell{prop 35 in Theret}
If $Q \colon \R^{2n} \times \R^{2nN} \rightarrow \R$
is a quadratic form generating the identity
then there is an isotopy $\{\Psi_s\}_{s\in [0,1]}$
of fibre preserving linear diffeomorphisms
of $\R^{2n} \times \R^{2nN}$
such that $\Psi_0$ is the identity
and $Q \circ \Psi_1$ is a quadratic form
that only depends on the fibre variable.
Moreover, if $Q$ is $\zk$-invariant
then $\{\Psi_s\}_{s\in [0,1]}$ can be chosen to be $\zk$-equivariant.
\end{lemma}

\begin{proof}
Write $Q(z) = \frac{1}{2} \, \langle z, Bz \rangle$
for a symmetric matrix
$B = \left[ \begin{array}{cc} 
a & b \\
b^T & c \\
\end{array}
\right]$.
Since $Q$ generates the zero section of $T^{\ast}\mathbb{R}^{2n}$
we have that $c$ is invertible and $a - b\, c^{-1} \, b^T = 0$.
Then
$$
\Psi_s (\zeta,\nu) =  (\zeta \, , \,  \nu - s \, c^{-1} \, b^T \, \zeta )
$$
is an isotopy of fibre preserving linear diffeomorphisms
of $\R^{2n} \times \R^{2nN}$ such that
$Q \circ \Psi_1$ only depends on the fibre variable,
as $Q \circ \Psi_1 (\zeta,\nu) = \frac{1}{2} \, \nu^T c \nu$.
If $Q$ is $\zk$-invariant
then $\{\Psi_s\}$ is $\zk$-equivariant.
\end{proof}

\begin{proof}[Proof of Proposition \ref{proposition: linear}]
Let $\{\Phi_t^{(1)}\}$ and $\{\Phi_t^{(2)}\}$
be loops in $\Sp(2n;\R)$, based at the identity.
If $Q_t^{(1)}$ and $Q_t^{(2)}$
are based families of generating quadratic forms
for $\{\Phi_t^{(1)}\}$ and $\{\Phi_t^{(2)}\}$ respectively
then $Q_t^{(1)} \, \sharp \, Q_t^{(2)}$ is a based family
of generating quadratic forms for $\{\Phi_t^{(2)} \circ \Phi_t^{(1)}\}$,
and so
\begin{equation}\labell{e: nu}
\nu \big( [\{\Phi_t^{(1)}\}] \cdot [\{\Phi_t^{(2)}\}] \big)
= \ind (Q_0^{(1)} \, \sharp \, Q_0^{(2)}) - \ind (Q_1^{(1)} \, \sharp \, Q_1^{(2)}) \,.
\end{equation}
By Lemma \ref{lemma: index based},
\begin{equation}\labell{e: nu2}
\ind (Q_0^{(1)} \, \sharp \, Q_0^{(2)}) = \ind (Q_0^{(1)}) + \ind (Q_0^{(2)}) \,.
\end{equation}
By Lemma \ref{prop 35 in Theret},
for $j = 1, 2$ there is an isotopy $\{\Psi_{s}^{(j)}\}_{s \in [0,1]}$
of fibre preserving linear diffeomorphisms
such that $\Psi_{0}^{(j)}$ is the identity
and $Q_1^{(j)} \circ \Psi_{1}^{(j)}$ is a quadratic form
that does not depend on the base variable,
and so is equal to a quadratic form $\underline{Q}_1^{(j)}$ on the fibre.
Using \eqref{equation: lemma 2 bis} we therefore obtain
$$
(Q_1^{(1)} \circ \Psi_{1}^{(1)}) \, \sharp \, (Q_1^{(2)} \circ \Psi_{1}^{(2)}) =
(0 \, \sharp \, 0) \oplus \underline{Q}_1^{(1)} \oplus \underline{Q}_1^{(2)} \,,
$$
and so
$$
\ind (Q_1^{(1)} \, \sharp \, Q_1^{(2)}) =  4n + \ind (\underline{Q}_1^{(1)}) + \ind (\underline{Q}_1^{(2)}) = \ind (Q_1^{(1)}) + \ind (Q_1^{(2)}) \,.
$$
Together with \eqref{e: nu} and \eqref{e: nu2} this gives
$$
\nu \big( [\{\Phi_t^{(1)}\}] \cdot [\{\Phi_t^{(2)}\}] \big)
= \nu \big( [\{\Phi_t^{(1)}\}] \big) + \nu \big( [\{\Phi_t^{(2)}\}] \big) \,,
$$
and thus $\nu \colon \pi_1 \big(\Sp(2n;\R)\big) \to \Z$ is a homomorphism.

In order to show that $\nu$ agrees with the linear Maslov index,
it now suffices to check that it takes the value $2$
on the standard loop
$t \mapsto \Phi_t = e^{2\pi i t} \in \Sp(2;\R)$.
For $0\leq t \leq \frac{1}{3}$ we have
the family of generating quadratic forms
$Q_t(z) = \frac{\sin (2\pi t)}{1+\cos(2\pi t)} \, \langle z, z \rangle$.
Applying the composition formula twice
we obtain a family of generating quadratic forms
on $\R^{10}$ determined by the family of
$10 \times 10$ symmetric matrices
$$ 
A_t = \left[ \begin{array}{ccccc}
0 & J & -J & 0 & 0 \\
-J & 0 & J & J & -J \\
J & -J & \lambda_t &  0 & 0 \\
0 & -J & 0 & \lambda_t & J \\
0 & J & 0 & -J & \lambda_t
\end{array}
\right],
\quad \quad \text{ with } \lambda_t =
\frac{ \sin\left(\frac{2\pi t}{3} \right) }{ 1+ \cos\left(\frac{2\pi t}{3} \right)}
\; \Id_{2 \times 2}
$$
where we denote by $J$ the matrix
$\begin{bmatrix} 
0 & -1 \\
1 & 0 \\
\end{bmatrix}$.
It follows that
$$
\nu(\{\Phi_t\}) = i(A_0) - i(A_1) = 6 - 4 = 2
$$
as required. 
\end{proof}

We can now prove the formula
at the beginning of Theorem \ref{theorem: main}.

\begin{ex}\labell{Reeb flow}
Recall that we denote by $\{r_t\}$ the Reeb flow on $L_k^{2n-1}$
with respect to the contact form whose pullback to $S^{2n-1}$
is equal to the pullback from $\mathbb{R}^{2n}$
of the 1-form $\sum_{j=1}^n ( x_j dy_j - y_j dx_j)$.
Since the lift to $\mathbb{R}^{2n}$ of $\{ r_{2 \pi t} \}_{t \in [0,1]}$
is a loop in $\Sp (2n; \mathbb{R})$ of linear Maslov index $2n$,
it follows from Proposition \ref{proposition: linear} that
$$
\mu \big(\{r_{2 \pi l t}\}_{t \in [0,1]}\big) = \ol{\mu} \big(\{r_{2 \pi l t}\}_{t \in [0,1]}\big)  = 2n l
$$
for every integer $l$.
\eoe
\end{ex}

\subsection*{Relation with discriminant points}

We now show that the way the non-linear Maslov index of a contact isotopy 
changes for $t$ varying in a subinterval of $[0,1]$ 
is related to the changes in the topology of the set of discriminant points. 
Before stating the results we recall the following fact. 

\begin{lemma} \labell{topological lemma}
Let $V$ be a compact manifold and 
$f_t \colon V \rightarrow \R$, for $t\in [0,1]$, 
a family of functions
such that the total map $f \colon V \times [0,1] \rightarrow \mathbb{R}$ is
$\mathcal{C}^1$ with Lipschitz differential.
Suppose that $a \in \R$ is a regular value of $f_t$
for every $t \in [0,1]$.
Then there is an isotopy $\theta_t$ of $V$ 
such that $\theta_t(\{f_0\leq a\}) = \{f_t\leq a\}$.
\end{lemma}

\begin{proof}
Consider the open subset
$$
\mathcal{U}:=\{ (x,t) \ | \ df_t|_x \neq 0 \}
$$
of $V \times [0,1]$.
By assumption, the set $\{ (x,t) \ | \ f_t (x) = a \}$
is contained in $\mathcal{U}$.
Fix a Riemannian metric on $V$.
Since the functions $f_t$ are $\mathcal{C}^1$ with Lipschitz differential,
their gradient flow is well defined and enjoys the usual properties.
Note also that the gradient  $\nabla f_t$
is non-zero at a point $x$ exactly when $(x,t) \in \mathcal{U}$.
For every $t \in [0,1]$ define a vector field $u_t$ on 
$\mathcal{U}_t := \{ x \in V \ | \ df_t|_x \neq 0\}$ by
$u_t = \nabla f_t / \| \nabla f_t \|^2 $.
Then $df_t(u_t) \equiv 1$ on $\mathcal{U}$.
Take $\varepsilon >0$ small enough so that the closed neighborhood
$$
\mathcal{W}:=\{(x,t) \in V \times [0,1]\ |\ |f(x,t)-a| \leq \varepsilon \}
$$
of $\{ (x,t) \ | \ f_t (x) = a \}$ is contained in $\mathcal{U}$.
Let $\rho \colon V \times [0,1] \to \R$
be a smooth function that is supported in $\mathcal{U}$
and is equal to~$1$ on $\mathcal{W}$,
and consider the time-dependent vector field $\{X_t\}_{t \in [0,1]}$ on $V$ 
that is given by 
$$
X_t (x) = - \, \rho(x,t) \, \dot{f}_t(x) \, u_t (x)
$$
for $(x,t)$ in $\mathcal{U}$
and that vanishes for $(x,t)$ outside of $\mathcal{U}$.
Its flow is an isotopy $\theta_t$ of $V$
with the required properties.
Indeed
$$
\frac{d}{dt} \, f_t \big(\theta_t(x)\big) 
= \dot{f}_t \big(\theta_t(x)\big) + df_t(X_t) \big(\theta_t(x)\big)
= \big(1-\rho(\theta_t(x),t)\big) \, \dot{f}_t \big(\theta_t(x)\big)
$$
thus $\frac{d}{dt} f_t \big(\theta_t(x)\big) = 0$ if  $\big(\theta_t(x),t\big) \in \mathcal{W}$,
and so each $\theta_t$ sends $\{ f_0 = a \}$ onto $\{ f_t = a \}$.
Since $\theta_0$ is the identity,
by continuity the isotopy $\theta_t$
sends $\{ f_0 \leq a \}$ onto $\{ f_t \leq a \}$ for all $t$.
\end{proof}

We can now prove that the non-linear Maslov index
detects discriminant points,
as described in Theorem \ref{theorem: main}(iii).
More precisely we show the following result.

\begin{prop}[Relation with discriminant points]\labell{discriminant points}
Let $\{\phi_t\}_{t \in [0,1]}$ be a based contact isotopy of $L_k^{2n-1}$,
and let $[t_0,t_1]$ be a subinterval of $[0,1]$.
\begin{enumerate}
\renewcommand{\labelenumi}{(\roman{enumi})}
\item 
If there are no values of $t \in [t_0,t_1]$
for which $\phi_t$ belongs to the discriminant then
$$
\mu\big([\{\phi_t\}_{t\in[0,t_0]}]\big) = \mu\big([\{\phi_t\}_{t\in[0,t_1]}]\big) \,.
$$
\item If $\underline{t}$ is the only value of $t \in [t_0,t_1]$ 
for which $\phi_t$ belongs to the discriminant then
$$
\big| \, \mu\big([\{\phi_t\}_{t\in[0,t_1]}]\big) - \mu\big([\{\phi_t\}_{t\in[0,t_0]}]\big) \, \big|
\leq \text{\emph{ind}} \big(\Delta(\phi_{\underline{t}})\big) + 1
$$
where $\Delta(\phi_{\underline{t}}) \subset L_k^{2n-1}$
is the set of discriminant points of $\phi_{\underline{t}}$.
Consequently,
$$
\big| \, \mu\big([\{\phi_t\}_{t\in[0,t_1]}]\big) - \mu\big([\{\phi_t\}_{t\in[0,t_0]}]\big) \, \big|
\leq 2n + 1
$$
and if $\phi_{\underline{t}}$ has only finitely many discriminant points then
$$
\left| \, \mu \big( [\{\phi_t\}_{t \in [0,t_1]}] \big) - \mu \big( [\{\phi_t\}_{t \in [0,t_0]}] \big) \, \right| \leq 2 \,.
$$

\item If $\underline{t}$ is the only value of $t \in [t_0,t_1]$ 
for which $\phi_{t}$ belongs to the discriminant,
and moreover all discriminant points of $\phi_{\underline{t}}$ are non-degenerate,
then
$$
\big| \, \mu\big([\{\phi_t\}_{t\in[0,t_1]}]\big) - \mu\big([\{\phi_t\}_{t\in[0,t_0]}]\big) \, \big| \leq 1 \,.
$$
\end{enumerate}
\end{prop}

\begin{proof}
Let $f_t \colon L_k^{2M-1} \rightarrow \R$
be a based family of generating functions for $\phi_t$.
If there are no values of $t \in [t_0,t_1]$
for which $\phi_t$ belongs to the discriminant then,
by Proposition \ref{correspondence discriminant pts - critical pts},
zero is a regular value of $f_t$  for all $t \in [t_0,t_1]$.
Hence, (i) follows from Lemma \ref{topological lemma}.

Suppose now that $\underline{t}$ is the unique value of $ t \in [t_0,t_1]$
for which $\phi_t$ belongs to the discriminant.
For any $\epsilon > 0$ we have
\begin{equation}
\labell{eq dp1}
\big| \, \mu\big([\{\phi_t\}_{t\in[0,t_1]}] \big) - \mu\big([\{\phi_t\}_{t\in[0,t_0]}]\big) \, \big| \leq 
\ind( \{ f_{\underline{t}} \leq \epsilon \}) - \ind( \{f_{\underline{t}} \leq -\epsilon\}) \,.
\end{equation}
This is a consequence of (i), monotonicity of the index,
and the fact that,
since $f \colon L_k^{2M-1} \times [0,1] \to \R$ is continuous,
for every $\epsilon > 0$ and $a \in \R$ 
there exists $\delta > 0$ such that for all $t, t'$ with $|t-t'| < \delta$
we have $\{f_t(x) \leq a \} \subset \{f_{t'}(x) \leq a + \epsilon\}$.

Let $C$ be the set of critical points of $f_{\underline{t}}$ with critical value $0$.
By continuity of the index, there is an open subset $W$ in $L_k^{2M-1}$
that contains $C$ and has the same index.
For sufficiently small $\epsilon > 0$ we have 
\begin{equation}
\labell{incsubsets}
\ind( \{ f_{\underline{t}} \leq \epsilon \} ) \leq \ind \big( \{f_{\underline{t}} \leq -\epsilon\} \cup W \big).
\end{equation}
This follows from monotonicity of the index
and the fact that, as we now explain, 
$\{ f_{\underline{t}} \leq \epsilon \}$ deformation retracts into
$\{f_{\underline{t}} \leq -\epsilon\} \cup W$
(cf.\ \cite[p.\ 548]{Viterbo - Massey products}).
Pick $\delta > 0$ such that if $df_{\underline{t}}(x) = 0$ and $|f_{\underline{t}}(x)|\leq \delta$ then $x\in W$.
Consider the disjoint closed sets
$$
V_0 = \big\{ \, x \in L_k^{2M-1} \;\lvert\; df_{\underline{t}}(x) = 0 \text{ or } | f_{\underline{t}} (x) | \geq 2 \delta \, \big\}
$$
and
$$
V_1 = f_{\underline{t}}^{-1}([-\delta,\delta]) \ssminus W \,,
$$
and let $\rho \colon L_k^{2M-1} \to [0,1]$
be a smooth function that vanishes in a neighborhood of $V_0$
and is constant equal to $1$ in a neighborhood of $V_1$.

\begin{figure}[h]
\begin{tikzpicture}[scale=0.7]
    \clip (0.1,0.1) rectangle (9.9,9.9);
    
    \filldraw[fill=gray!70]  (4,0) .. controls (3.9,4) and (4,3.9) ..  (0,4) -- (0,0);
    \filldraw[fill=gray!70]  (0,6) .. controls (4,6.1) and (3.9,6) ..  (4,10) -- (0,10);
    \filldraw[fill=gray!70]  (6,10) .. controls (6.1,6) and (6,6.1) ..  (10,6) -- (10,10);
    \filldraw[fill=gray!70]  (6,0) .. controls (6.1,4) and (6,3.9) ..  (10,4) -- (10,0);
    
    \filldraw[fill=gray!40] (0,5.2) .. controls (5,5.2) and (4.8,5)  ..  (4.8,10) -- (5.2,10) .. controls (5.2,5) and (5,5.2) .. (10,5.2) -- (10,4.8)
               .. controls (5,4.8) and (5.2,5) .. (5.2,0) -- (4.8,0) .. controls (4.8,5) and (5,4.8) .. (0,4.8);
    
     \filldraw[fill=gray!20, dashed, very thin] (4,4) rectangle (6,6);
      
     \filldraw [gray!70]  (5,5) circle (2pt);
    
    \draw (2.5,2.5) node {$V_0$};
    \draw (3,5) node {$V_1$};
    \draw (5.5,5.5) node {$W$};
    
\end{tikzpicture}
\caption{The sets $V_0$, $V_1$ and $W$ near a non-degenerate critical point of index one on a surface.}
\end{figure}

Fix a metric on $L_k^{2M-1}$ and consider the vector field 
$X = - \, \rho \, \frac{ \nabla f_{\underline{t}}}{\| \nabla f_{\underline{t}}\|^2}$.
Writing $\theta_t$ for the flow of $X$ we have
\begin{equation}\labell{e: 1}
\frac{d}{dt} \, f_{\underline{t}} \big(\theta_t(x)\big) = -\rho \big(\theta_t(x)\big) \,.
\end{equation}
Let 
$m = \max \, \{\, \|X(x)\| \;\lvert\; x \in L_k^{2M-1} \, \}$
and
$d = \dist \, \Big( \overline{\{ \rho(x) < 1\}} \cap f_{\underline{t}}^{-1}([-\delta,\delta]) \, , \, W^c \Big)$.
Note that $d > 0$.
For $\epsilon < \min \{ \delta, \frac{d}{2m} \}$ we now prove that
\begin{equation}\labell{equation: retracts}
\theta_{2\epsilon}( \{f_{\underline{t}} \leq \epsilon \}) \subset \{f_{\underline{t}} \leq -\epsilon\} \cup W \,.
\end{equation}
Given $x \in \{ \lvert f_{\underline{t}} \rvert \leq \epsilon\}$ set 
$$
s(x) = \inf \big\{\, t \in [0,2\epsilon] \;\lvert\; \rho( \theta_t(x) ) < 1 \text{ or } t = 2\epsilon \,\big\} \,.
$$
If $s(x) = 2\epsilon$ then $\rho(\theta_t(x)) = 1$ for all $t \in [0,2\epsilon]$
and, by \eqref{e: 1},
$f_{\underline{t}} \big(\theta_{2\epsilon}(x)\big) = f_{\underline{t}}(x) - 2\epsilon \leq - \epsilon$. 
If $s(x) < 2 \epsilon$ then $\theta_{s(x)}(x)\in \{\rho<1\} \cap f_{\underline{t}}^{-1}([-\delta, \delta])$ 
(as $\epsilon < \delta$).
Then we must have $\dist(\theta_{s(x)}, W^c)\geq d$
and our bound on $\epsilon$ ensures that
the path $\{\theta_t(x) \;\lvert \; t \in [s(x),2\epsilon]\}$
is entirely contained in $W$.
In particular, $\theta_{2\epsilon}(x) \in W$.
This completes the proof of \eqref{equation: retracts} and hence of \eqref{incsubsets}.

By \eqref{incsubsets},
subadditivity of the cohomological index
and Proposition \ref{correspondence discriminant pts - critical pts}
we have
\begin{align*}
\ind (\{f_{\underline{t}} \leq \epsilon\}) &\leq
\ind \big(\{f_{\underline{t}} \leq - \epsilon\} \cup W\big)
\leq \ind \big( \{f_{\underline{t}} \leq - \epsilon \}\big) + \ind(W) + 1 \\
& = \ind \big( \{f_{\underline{t}} \leq - \epsilon \} \big) + \ind(C) + 1 \\
& = \ind \big( \{f_{\underline{t}} \leq - \epsilon \} \big) + \text{ind} \big(\Delta(\phi_{\underline{t}})\big) + 1 \,.
\end{align*}
Together with \eqref{eq dp1}, this implies (ii).

As for (iii), if all discriminant points of $\phi_{\underline{t}}$ are non-degenerate
then (by Proposition \ref{correspondence discriminant pts - critical pts})
all critical points of $f_{\underline{t}}$ of critical value zero are non-degenerate.
Thus $f_{\underline{t}}$ has only finitely many critical points with critical value zero,
and zero is an isolated critical value of $f_{\underline{t}}$.
We can choose $W$ so that
$\{f_{\underline{t}} \leq -\epsilon\} \cup W$
can be obtained from $\{f_{\underline{t}} \leq - \epsilon\}$
by attaching a finite number of disjoint handles.
If $H$ is a handle and $A \subset L_k^{2M-1}$
then $\ind(A \cup H) \leq \ind(A)+1$
(as the sum of the Betti numbers of $A\cup H$ 
is at most one more than the sum of Betti numbers of $A$)
and, unless the index of the handle
$H$ is equal to $\ind(A)+1$, we have $\ind(A \cup H) = \ind(A)$.
Attaching the handles in $W$ sequentially,
starting with those of highest index,
we therefore obtain
$$
\ind \big(\{f_{\underline{t}} \leq -\epsilon\} \cup W\big)
\leq \ind (\{ f_{\underline{t}} \leq -\epsilon\}) + 1 \,,
$$
which, together with \eqref{eq dp1} and \eqref{incsubsets}, concludes the proof of (iii).
\end{proof}

\begin{rmk}\labell{remark: Massey}
In order to prove the version of the contact Arnold conjecture
with the bound given by the Lusternik--Schnirelmann category
(cf.\ Section \ref{section: applications})
we would need to know that
the conclusion of Proposition \ref{discriminant points}(iii)
holds also in the degenerate case.
Using Massey products similarly to \cite{Viterbo - Massey products}
it is possible to prove (at least if $k = 3$) that this is the case
if, in the notation of the proof above,
$\ind \big(\{ f_{\underline{t}} \leq \epsilon \} \big) = 1$.
It is not clear to us whether such arguments
can be pushed further to improve this result.
\eor
\end{rmk}

\subsection*{Further properties}

We now prove the positivity property from Theorem \ref{theorem: main},
and the fact that the asymptotic non-linear Maslov index
is monotone and has the vanishing property.

\begin{prop}[Positivity]
\labell{proposition: positivity}
If $\{\phi_t\}$ is a non-negative (respectively non-positive) contact isotopy
then $\mu \big( [\{\phi_t\}] \big) \geq 0$ (respectively $\mu \big([\{\phi_t\}]\big) \leq 0$).
Moreover, if $\{\phi_t\}$ is positive then $\mu \big( [\{\phi_t\}] \big) > 0$.
\end{prop}

\begin{proof}
It follows from monotonicity of generating functions 
(Proposition \ref{proposition: monotonicity of generating functions}) 
and monotonicity of the cohomological index 
(Proposition \ref{proposition: cohomological index}(i)) 
that if $\{\phi_t\}$ is non-negative (respectively non-positive)
then $\mu \big([\{\phi_t\}] \big) \geq 0$ (respectively $\mu \big([\{\phi_t\}]\big) \leq 0$).
By Example \ref{example: def in small case},
if $\{\phi_t\}$ is a small positive contact isotopy
then $\mu \big( [\{\phi_t\}] \big) = 2n > 0$.
Since (by Proposition \ref{proposition: monotonicity of generating functions}
and Proposition \ref{proposition: cohomological index}(i))
the cohomological index does not decrease along a positive contact isotopy,
we conclude that $\mu \big( [\{\phi_t\}] \big) > 0$ for any positive contact isotopy $\{\phi_t\}$.
\end{proof}

We now show that the asymptotic non-linear Maslov index
satisfies the following stronger property.
Recall from \cite{EP - Partially ordered groups}
that for a contact manifold $(V,\xi)$
the relation $\leq$ on $\widetilde{\Cont}_0(V,\xi)$
is defined by posing $[\{\phi_t\}] \leq [\{\psi_t\}]$
if $[\{\psi_t\}] \cdot [\{\phi_t\}]^{-1}$
can be represented by a non-negative contact isotopy.
As in \cite{BZ} we say that a quasimorphism $\nu$ on $\widetilde{\Cont}_0(V,\xi)$
is \emph{monotone} if $\nu([\{\phi_t\}] )\leq \nu([\{\psi_t\}])$
whenever  $[\{\phi_t\}] \leq [\{\psi_t\}]$.
The proof of the following result
is a direct imitation of
the proof of the similar statement for real projective spaces
that is given in \cite{BZ}.

\begin{prop}\labell{proposition: monotonicity - asymptotic}
The asymptotic non-linear Maslov index
$\ol{\mu}$ on $\widetilde{\Cont}_0(L_k^{2n-1})$
is monotone.
\end{prop}

\begin{proof}
Suppose that $[\{\phi_t\}] \leq [\{\psi_t\}]$.
Since the set of non-negative elements of  $\widetilde{\text{Cont}}_0(L_k^{2n-1})$
is conjugation invariant and closed under multiplication,
$[\{\phi_t\}]^m \leq [\{\psi_t\}]^m$  for any $m \in \Z_{>0}$.
By Proposition \ref{proposition: positivity}
we thus have $\mu \, ([\{\psi_t\}]^m \cdot [\{\phi_t\}]^{-m})\geq 0$,
and so
\begin{align*}
\mu ([\{\psi_t\}]^m) - \mu ([\{\phi_t\}]^m) 
&= \mu ([\{\psi_t\}]^m) + \mu ([\{\phi_t\}]^{-m}) \\
&\geq \mu \big([\{\psi_t\}]^m \cdot [\{\phi_t\}]^{-m}\big) - D \geq - D \, ,
\end{align*}
where $D$ is the error of the quasimorphism.
Dividing by $m$ and taking the limit when $m \rightarrow \infty$
we obtain that 
$\ol{\mu} ([\{\phi_t\}]) \leq \ol{\mu} ([\{\psi_t\}])$.
\end{proof}

For the next result we also follow \cite{BZ}.
Recall that a subset $\mathcal{U}$ of a contact manifold $(V,\xi)$ 
is said to be {\it displaceable} if there exists a contactomorphism $\psi$
contact isotopic to the identity such that 
$\overline{\mathcal{U}} \cap \psi(\mathcal{U})= \emptyset$.
A quasimorphism $\nu$ on $\widetilde{\Cont_0}(V,\xi)$ 
is said to have the \emph{vanishing property} if 
$\nu \big([\{\phi_t\}]\big) = 0$ for any contact isotopy $\{\phi_t\}_{t \in [0,1]}$
that is supported in $[0,1] \times \mathcal{U}$ 
for a displaceable set $\mathcal{U}$.

\begin{prop}\labell{proposition: vanishing property}
The asymptotic non-linear Maslov index has the vanishing property.
\end{prop}

\begin{proof}
Suppose that a subset $\mathcal{U}$ of $L_k^{2n-1}$ is displaceable
by a contactomorphism $\psi$ contact isotopic to the identity.
After taking a $\mathcal{C}^0$-perturbation,
we can assume that $\psi$ has no discriminant points.
Let $\{\psi_t\}_{t\in [0,1]}$ be a contact isotopy from the identity to $\psi_1 = \psi$,
and $\{\phi_t\}_{t\in[0,1]}$ a contact isotopy
supported in $[0,1] \times \mathcal{U}$.
We need to show that $\ol{\mu} \big([\{\phi_t\}]\big) = 0$.
Observe first that,
for every $m \in \mathbb{N}$ and $t \in [0,1]$,
the contactomorphism $\psi \circ \phi_t^m$ has no discriminant points.
Indeed, assume by contradiction 
that $p$ is a discriminant point of $\psi \circ \phi_t^m$.
Since $\phi_t^m$ is supported in $\mathcal{U}$
and $\psi$ has no discriminant points
we must have $p \in \mathcal{U}$.
But then $\phi_t^m(p) \in \mathcal{U}$ and so 
$\psi \circ \phi_t^m(p) \in \mathcal{U} \cap \psi(\mathcal{U})$
contradicting the hypothesis that $\psi$ displaces $\mathcal{U}$.
Since $\psi \circ \phi_t^m$ has no discriminant points for all $t \in [0,1]$,
it follows from Proposition \ref{discriminant points}(i) that for 
the concatenation $\{\psi_t\} \sqcup \{\psi \circ \phi_t^m\}$ we have
$$
\mu \big( \{\psi_t\} \sqcup \{\psi \circ \phi_t^m\} \big) = \mu (\{\psi_t\})\,.
$$
Since $\{\psi_t\} \sqcup \{\psi \circ \phi_t^m\}$ and $\{\psi_t \circ \phi_t^m\}$
are homotopic, the quasimorphism property
(Proposition \ref{p:quasimorphism property})
implies that $\lvert \mu (\{\phi_t^m\})\rvert \leq 2n+1$.
As this holds for every $m \in \mathbb{N}$,
we conclude that $\ol{\mu} \big([\{\phi_t\}]\big) = 0$.
\end{proof}


\section{Applications}
\labell{section: applications}

In this section we use the properties of the non-linear Maslov index
to prove the applications listed in Corollaries \ref{cor1}, \ref{cor 1bis} and \ref{cor2}. 
Most of the arguments are taken from
\cite{EP - Partially ordered groups, CS - Discriminant metric, S - Morse estimate for translated points, BZ}
with only minor changes to adapt them to the case of lens spaces,
and are included here for the sake of completeness.

\subsection*{Orderability}\labell{section: orderability}
As in the case of projective spaces
discussed in \cite{EP - Partially ordered groups},
orderability of lens spaces
follows from 
positivity of the non-linear Maslov index
and the fact that the non-linear Maslov index is well defined on the universal cover 
of the contactomorphism group.
Indeed, suppose by contradiction that
a lens space $L_k^{2n-1}$ is  not orderable,
and thus admits a positive contractible loop $\{\phi_t\}_{t\in[0,1]}$.
Since $\{\phi_t\}_{t\in[0,1]}$ is contractible we have that 
$\mu( [\{\phi_t\}_{t\in[0,1]} ]) = 0$. 
On the other hand, 
since $\{\phi_t\}$ is positive, Proposition \ref{proposition: positivity}
implies that $\mu([\{\phi_t\}]) > 0$,
giving a contradiction.

\subsection*{Unboundedness of the discriminant and oscillation metrics}\labell{subsection: metrics}

Recall that any bi-invariant (pseudo)metric $d: G \times G \rightarrow \mathbb{R}$ on a group $G$
defines a conjugation-invariant (pseudo)norm $\lVert \, \cdot \, \rVert: G \rightarrow \mathbb{R}$
by posing $\lVert g \rVert = d (g, \text{id})$;
conversely, any conjugation-invariant (pseudo)norm $\lVert \, \cdot \, \rVert: G \rightarrow \mathbb{R}$
defines a bi-invariant (pseudo)metric by posing $d(g_1,g_2) = \lVert g_1 g_2^{-1}\rVert$.
The discriminant and oscillation metrics
on $\widetilde{\Cont}_0(V,\xi)$ for a compact co-oriented contact manifold $(V,\xi)$
are defined as follows \cite{CS - Discriminant metric}.
As proved in \cite{CS - Discriminant metric},
any non-trivial element in $\widetilde{\Cont}_0(V,\xi)$
has a representative $\{\phi_t\}_{t \in [0,1]}$ that can be written
as the concatenation of a finite number of pieces $\{\phi_t\}_{t \in [t_{j-1},t_{j}]}$, $j = 1, \ldots, L$, 
such that each piece is \emph{embedded}, 
i.e.\ for every two distinct $t$ and $t'$ in $[t_{j-1},t_{j}]$
the composition $\phi_t \circ \phi_{t'}^{-1}$ does not have any discriminant point.
The discriminant norm of a non-trivial element $[\{\phi_t\}]$
is then defined to be the minimal number of pieces in such decompositions
(in other words, the discriminant norm is the word norm
with respect to the generating set
formed by elements of $\widetilde{\Cont}_0(V,\xi)$
that can be represented by an embedded contact isotopy).
Moreover, any non-trivial element in $\widetilde{\Cont}_0(V,\xi)$ has a representative $\{\phi_t\}_{t\in [0,1]}$
that can be written as the concatenation of a finite number of embedded pieces
such that each piece is either non-negative or non-positive.
Let $L_+$ and $L_-$ be respectively
the minimal number of non-negative and of non-positive pieces
in such decompositions;
the oscillation pseudonorm of $[\{\phi_t\}]$ is then defined to be $L_+ + L_-$.
The oscillation pseudometric is non-degenerate, hence a metric,
if and only if $(V,\xi) $ is orderable.

In \cite{CS - Discriminant metric} the non-linear Maslov index
has been used to show that
the discriminant and oscillation metrics for real projective space are unbounded, 
hence not equivalent to the trivial metric. 
The argument, applied to lens spaces, is as follows.

Consider the Reeb flow $\{r_t\}$ on $L_k^{2n-1}$
with respect to the contact form whose pullback to $S^{2n-1}$
is equal to the pullback from $\mathbb{R}^{2n}$
of the 1-form $\sum_{j=1}^n ( x_j dy_j - y_j dx_j)$.
We first show that the discriminant norm
of the $3l$-th iteration $\{ r_{6 \pi l t} \}_{t \in [0,1]}$ of the loop $\{r_{2 \pi t} \}_{t\in [0,1]}$
is at least $l + 1$. 
By Example \ref{Reeb flow} we know that $\mu([\{ r_{6 \pi l t} \}_{t \in [0,1]}]) = 6nl$. 
Let $\{\phi_t\}_{t \in [0,1]}$ be a contact isotopy that represents $[\{ r_{6 \pi l t} \}_{t \in [0,1]}]$, 
is a concatenation of embedded pieces
and minimizes the number of such pieces. 
Then $\mu ([\{\phi_t\}_{t\in[0,1]}]) = 6nl$.
If $l > 0$ then, by Example \ref{example: def in small case},
for a sufficiently small $\epsilon > 0$
we have that $\mu ([\{\phi_t\}_{t\in[0,\epsilon]}])$ is ($\leq 2n$, hence) different from $6nl$.
By Proposition \ref{discriminant points}(i)
this implies that $\{\phi_t\}_{t \in [0,1]}$ must intersect the discriminant,
and so $\{\phi_t\}_{t \in [0,1]}$ has at least two embedded pieces.
Suppose now that $l > 1$,
and write $\{\phi_t\}_{t \in [0,1]}$ as a concatenation
of $L$ embedded pieces $\{ \phi_t \}_{t \in [t_{j-1},t_j]}$, $j = 1, \ldots, L$.
For each $j$, since $\{ \phi_t \}_{t \in [t_{j-1},t_j]}$ is embedded
we have in particular that $\phi_t \circ \phi_{t_{j-1}}^{\phantom{j-}-1}$
does not have any discriminant point for every $t \in (t_{j-1},t_j]$.
Consider a value of time $\underline{t} \in (t_{j-1},t_j]$ such that
$\{ \phi_t \circ \phi_{t_{j-1}}^{\phantom{j-}-1} \}_{t \in [t_{j-1},\underline{t}]}$ is $\mathcal{C}^1$-small.
By Proposition \ref{discriminant points}(i) and Example \ref{example: def in small case} we have
$$
\mu \big( \big[\{ \phi_t \circ \phi_{t_{j-1}}^{\phantom{j-}-1} \}_{t \in [t_{j-1},t_j]} \big] \big) = 
\mu \big( \big[\{ \phi_t \circ \phi_{t_{j-1}}^{\phantom{j-}-1} \}_{t \in [t_{j-1},\underline{t}]} \big] \big) \leq 2n\,.
$$
Suppose now by contradiction that $L < l+1$.
Then
\begin{equation}\label{equation: discr}
\sum_{j = 1}^{L} \; \mu \big( \big[\{ \phi_t \circ \phi_{t_{j-1}}^{\phantom{j-}-1} \}_{t \in [t_{j-1},t_j]} \big] \big) \leq 2nL < 2n (l+1) \,.
\end{equation}
On the other hand, by the quasimorphism property
(Proposition \ref{p:quasimorphism property})
we have
$$
\Big\lvert \; \mu([\{\phi_t\}_{t\in[0,1]}]) - \sum_{j = 1}^{L} \; \mu \big( \big[\{ \phi_t \circ \phi_{t_{j-1}}^{\phantom{j-}-1} \}_{t \in [t_{j-1},t_j]} \big] \big) \; \Big\rvert
\leq (L-1) (2n+1) < l (2n+1) \,,
$$
and thus
$$
\sum_{j = 1}^{L} \; \mu \big( \big[\{ \phi_t \circ \phi_{t_{j-1}}^{\phantom{j-}-1} \}_{t \in [t_{j-1},t_j]} \big] \big) > 4nl - l \,.
$$
This contradicts \eqref{equation: discr},
and thus concludes the proof that
the discriminant norm of $\{ r_{6 \pi l t} \}_{t \in [0,1]}$ is at least $l + 1$. 

The fact that the oscillation metric is unbounded
can be seen by combining the above argument
with positivity of the non-linear Maslov index, as follows.
We show that the oscillation norm
of the class $[\{ r_{20 \pi l t} \}_{t \in [0,1]}]$ of $10l$-th iteration
of the loop $\{r_{2 \pi t} \}_{t\in [0,1]}$ is at least $l + 1$.
For $[\{ r_{20 \pi l t} \}_{t \in [0,1]}]$,
the minimal number $L_-$ of non-positive embedded pieces is zero,
because any subdivision of the representative $\{ r_{20 \pi l t} \}_{t \in [0,1]}$
has no non-positive pieces.
So we need to show that
the minimal number $L_+$ of embedded non-negative pieces is at least $l + 1$.
Let $\{\phi_t\}_{t \in [0,1]}$ be a contact isotopy that represents $[\{ r_{20 \pi l t} \}_{t \in [0,1]}]$, 
is a concatenation of non-negative or non-positive embedded pieces
and minimizes the number of non-negative ones.
Then, by Example \ref{Reeb flow}, $\mu([\{\phi_t\}_{t \in [0,1]}]) = 20 nl$.
Regarding adjacent embedded pieces of the same sign
as a single non-negative or non-positive isotopy,
let $K_+$ and $K_-$ respectively
be the number of non-negative and of non-positive isotopies in the decomposition.
If $K_+ \geq l+1$ then the number $L_+$ of non-negative embedded pieces
is also at least $l+1$ and so we are done.
Assume thus that $K_+ < l+1$,
and so (since $K_- \leq K_+ + 1$) $K := K_+ + K_- < 2l + 2$.
By the quasimorphism property we then have
$$
\Big\lvert \; \mu([\{\phi_t\}_{t\in[0,1]}]) - \sum_{j = 1}^{K} \; \mu \big( \big[\{ \phi_t \circ \phi_{t_{j-1}}^{\phantom{j-}-1} \}_{t \in [t_{j-1},t_j]} \big] \big) \; \Big\rvert
\leq (K - 1) (2n+1) < (2l+1) (2n+1) \,.
$$
Write
$\sum_{j = 1}^{K} \; \mu \big( \big[\{ \phi_t \circ \phi_{t_{j-1}}^{\phantom{j-}-1} \}_{t \in [t_{j-1},t_j]} \big] \big) = \mu_+ + \mu_-$,
where $\mu_+$ and $\mu_-$ are respectively the sum
of the non-linear Maslov indices of the non-negative and of the non-positive pieces.
By Proposition \ref{proposition: positivity} we have $\mu_- \leq 0$,
and thus we conclude that
\begin{equation}\label{equation: osc}
\mu_+ > 20 nl - (2l + 1)(2n + 1) \,.
\end{equation}
On the other hand,
if the number $L_+$ of non-negative embedded pieces was less than $l+1$
then we would have $\mu_+ < (4n+1) (l+1)$,
contradicting \eqref{equation: osc}.
This concludes the proof.

\subsection*{Contact Arnold conjecture}
\labell{subsection: Contact Arnold conjecture}

Adapting to our case the argument given for $\RP^{2n-1}$
in \cite{S - Morse estimate for translated points}
(which in turn is an adaptation of the proof
of the Hamiltonian Arnold conjecture for $\mathbb{CP}^n$
given in \cite{Theret - Rotation numbers} and \cite{Givental - Nonlinear Maslov index})
we show that for any contactomorphism of $L_k^{2n-1}$
that is contact isotopic to the identity 
the number of translated points
(with respect to the standard contact form) is at least $n$,
and at least $2n$ if all translated points are assumed to be non-degenerate.

Recall that we denote by $\{r_t\}$ the Reeb flow on $L_k^{2n-1}$.
The set of translated points of a contactomorphism $\phi$ of $L_k^{2n-1}$
contact isotopic to the identity
is equal to the union for $t \in (0,1]$ of the sets of those 
discriminant points of $r_{2 \pi t} \circ \phi$
that correspond to $\mathbb{Z}_k$-orbits of discriminant points
of the lift to the sphere.
Without loss of generality we can assume that $\phi$ has no discriminant points.
Indeed, either $\phi$ has infinitely many translated points
or, for some value $\underline{t}$ of $t \in [0,1]$,
the composition $r_{2\pi \underline{t}} \circ \phi$ has no discriminant points;
we can then replace $\phi$ with $r_{2\pi \underline{t}} \circ \phi$,
since the translated points of $r_{2\pi \underline{t}} \circ \phi$
are in 1--1 correspondence with those of $\phi$.
Assume thus that $\phi$ has no discriminant points,
and let $\{\phi_t\}_{t \in [0,1]}$ be a contact isotopy from the identity to $\phi$.
We first prove that for the concatenation
$\{\phi_t\}_{t \in [0,1]} \sqcup \{r_{2\pi t} \circ \phi \}_{t \in [0,1]}$
we have
\begin{equation}\labell{difference 2n}
\mu \big(\{\phi_t\}_{t \in [0,1]} \sqcup \{r_{2\pi t} \circ \phi \}_{t \in [0,1]} \big)
- \mu \big(\{\phi_t\}_{t \in [0,1]} \big) = 2n\,.
\end{equation}
Let $F_t$ be a based family of generating functions for $\{\phi_t\}_{t \in [0,1]}$,
and $Q_t$ a based family of $\zk$-invariant generating quadratic forms for $\{r_{2\pi t}\}_{t \in [0,1]}$.
Then $F_t \, \sharp \, Q_t$ is a based family
of generating functions for $\{ r_{2 \pi t} \circ \phi_t\}$,
which is homotopic to $\{\phi_t\} \sqcup \{r_{2\pi t} \circ \phi \}$,
and $F_t \, \sharp \, Q_0$ is a based family of generating functions for $\{\phi_t\}$.
Thus
$$
\mu \big(\{\phi_t\}_{t \in [0,1]} \sqcup \{r_{2\pi t} \circ \phi \}_{t \in [0,1]} \big)
- \mu \big(\{\phi_t\}_{t \in [0,1]} \big) 
= \text{ind}(F_1 \, \sharp \, Q_0) - \text{ind}(F_1 \, \sharp \, Q_1)\,.
$$
By Lemma \ref{prop 35 in Theret}
there are isotopies starting from the identity
$\{\Psi_s^0\}_{s \in [0,1]}$ and $\{\Psi_s^1\}_{s \in [0,1]}$
of $\zk$-equivariant fibre preserving linear diffeomorphisms
such that $Q_0 \circ \Psi^0_1$ and $Q_1 \circ \Psi^1_1$ are independent of the base variable,
and so are equal to quadratic forms $\overline{Q}_0$ and $\overline{Q}_1$ on the fibre.
For $j = 0, 1$ we have
$F_1 \, \sharp \, (Q_j \circ \Psi_1^{j}) = (F_1 \, \sharp \, 0) \oplus \overline{Q}_j$,
and so
$$
\ind (F_1 \, \sharp \, Q_j) = \ind \big( F_1 \, \sharp \, (Q_j \circ \Psi_1^{j}) \big)
= \ind \big( (F_1 \, \sharp \, 0) \oplus \overline{Q}_j \big)
= \ind (F_1 \, \sharp \, 0) + \ind (\overline{Q}_j)
$$
where the last equality
follows from Corollary \ref{proposition: stab additivity}.
By Example \ref{Reeb flow} we thus have
$$
\text{ind}(F_1 \, \sharp \, Q_0) - \text{ind}(F_1 \, \sharp \, Q_1)
= \ind (\overline{Q}_0) - \ind (\overline{Q}_1) = 2n \,,
$$
hence \eqref{difference 2n}.
Knowing this, Proposition \ref{discriminant points}(ii)
and the fact that $\phi$ has no discriminant points
imply that either there are at least $n$ distinct values of $t \in (0,1)$
at which $\mu$ jumps,
and so $\phi$ has at least $n$ translated points
(which are necessarily all distinct,
since their lifts to the sphere have different time-shifts),
or there is at least one value of $t$ at which $\mu$ jumps by more than 2, 
and thus $\phi$ has infinitely many translated points. 
Moreover, if all translated points of $\phi$ are non-degenerate 
then Proposition \ref{discriminant points}(iii) implies that
there must be at least $2n$ of them (all distinct).

\subsection*{Existence of translated points
with respect to an arbitrary contact form
that induces the standard contact structure}

Let $\alpha$ be a contact form on $L_k^{2n-1}$
that induces the standard contact structure,
and consider a contactomorphism $\phi$
that is contact isotopic to the identity.
We now prove that
there are infinitely many distinct real numbers
that are time-shifts of translated points of $\phi$ with respect to $\alpha$,
and thus in particular that $\phi$ has at least one translated point with respect to $\alpha$.
Denote by $\alpha_0$ the standard contact form,
and by $f$ the function such that $\alpha = e^f \alpha_0$.
Then $h = e^{-f}$ is the contact Hamiltonian function
(with respect to $\alpha_0$)
of the Reeb flow $\{\varphi_t^{\alpha}\}$ of $\alpha$.
Let
$$
\epsilon := \min h > 0 \,.
$$
For any based contact isotopy $\{\phi_t\}_{t \in [0,1]}$ with $\phi_1 = \phi$
we prove that
\begin{equation}\labell{lower bound}
\mu \Big( \big[ \{\phi_t\}_{t \in [0,1]} \sqcup \{\varphi^{\alpha}_{tT} \circ \phi \}_{t \in [0,1]} \big] \Big)
\geq \mu \big( \big[ \{\phi_t\}_{t \in [0,1]} \big] \big) + 2n \, m_T - 6n - 3 \,,
\end{equation}
where $m_T = \lfloor \frac{\epsilon T}{2 \pi} \rfloor$.
For this, note first that
$$
\big[ \{\phi_t\}_{t \in [0,1]} \sqcup \{\varphi^{\alpha}_{tT} \circ \phi \}_{t \in [0,1]} \big]
\geq \big[ \{\phi_t\}_{t \in [0,1]} \sqcup \{r_{t \epsilon T} \circ \phi \}_{t \in [0,1]} \big] \,,
$$
where, as above, $\{r_t\}_{t \in \mathbb{R}}$ denotes the Reeb flow of $\alpha_0$.
Indeed,
$$
\big[ \{\phi_t\}_{t \in [0,1]} \sqcup \{\varphi^{\alpha}_{tT} \circ \phi \}_{t \in [0,1]} \big] \cdot
\big[ \{\phi_t\}_{t \in [0,1]} \sqcup \{r_{t \epsilon T} \circ \phi \}_{t \in [0,1]} \big]^{-1}
$$
is represented by the contact isotopy
$\{\Id\}_{t \in [0,1]} \sqcup \{ \varphi_{tT}^{\alpha} \circ r_{t \epsilon T}^{-1}\}_{t \in [0,1]}$,
and the contact Hamiltonian function
of $\{\varphi_{tT}^{\alpha} \circ r_{t \epsilon T}^{-1}\}$
with respect to $\alpha_0$ is
$$
T \, h \left( 1 - \frac{\epsilon}{ \, h \circ (\varphi^{\alpha}_{tT})^{-1}} \right) \,
$$
which is non-negative by our choice of $\epsilon$.
By the quasimorphism property (Proposition \ref{p:quasimorphism property})
and positivity (Proposition \ref{proposition: positivity}),
it follows that
\begin{equation}\labell{two reebs}
\mu \big( \big[ \{\phi_t\}_{t \in [0,1]} \big] \cdot \big[ \{\varphi_{tT}^{\alpha}\}_{t \in [0,1]} \big] \big)
\geq \mu \big( \big[ \{\phi_t\}_{t \in [0,1]} \big] \cdot \big[ \{r_{t \epsilon T} \}_{t \in [0,1]} \big] \big) - 2n - 1 \,.
\end{equation}
By Proposition \ref{proposition: positivity},
$\mu \big(\big[ \{r_{t \tau} \}_{t \in [0,1]} \big]\big) \geq 0$
for all $\tau \in [0,2\pi]$;
by the quasimorphism property we thus have
\begin{equation}\labell{lower for reeb}
\mu \big(\big[ \{r_{t \epsilon T} \}_{t \in [0,1]} \big]\big) \geq \mu \big(\big[ \{r_{2 \pi m_T t} \}_{t \in [0,1]} \big]\big) - 2n -1 \,.
\end{equation}
Using \eqref{two reebs}, the quasimorphism property, \eqref{lower for reeb}
and Example \ref{Reeb flow} we obtain
\begin{align*}
\mu \Big( \big[ \{\phi_t\}_{t \in [0,1]} \sqcup \{\varphi^{\alpha}_{tT} \circ \phi \}_{t \in [0,1]} \big] \Big)
&= \mu \big( \big[ \{\phi_t\}_{t \in [0,1]} \big] \cdot \big[ \{\varphi_{tT}^{\alpha}\}_{t \in [0,1]} \big] \big) \\
& \geq \mu \big( \big[ \{\phi_t\}_{t \in [0,1]} \big] \cdot \big[ \{r_{t \epsilon T} \}_{t \in [0,1]} \big] \big) - 2n - 1 \\
& \geq \mu \big( \big[ \{\phi_t\}_{t \in [0,1]} \big] \big) + \mu \big( \big[ \{r_{t \epsilon T} \}_{t \in [0,1]} \big] \big) - 4n - 2 \\
& \geq \mu \big( \big[ \{\phi_t\}_{t \in [0,1]} \big] \big) + \mu \big(\big[ \{r_{2 \pi m_T t} \}_{t \in [0,1]} \big]\big) - 6n - 3 \\
&= \mu \big( \big[ \{\phi_t\}_{t \in [0,1]} \big] \big) + 2n \, m_T - 6n - 3 \,.
\end{align*}
Thus, inequality \eqref{lower bound} holds.
This inequality in turns implies that
$\big\lvert \, \mu \big( \big[ \{\phi_t\}_{t \in [0,1]} \sqcup \{\varphi^{\alpha}_{tT} \circ \phi \}_{t \in [0,1]} \big] \big) \, \big\rvert$
approaches $\infty$ as $T \rightarrow \infty$.
By Proposition \ref{discriminant points}(i)
we thus conclude that
there are infinitely many distinct real numbers
that are time-shifts of translated points of $\phi$ with respect to $\alpha$,
and so in particular $\phi$ has at least one translated point with respect to $\alpha$.

\subsection*{Weinstein conjecture}\labell{section: Weinstein}

Following Givental \cite{Givental - Nonlinear Maslov index},
we now prove that any contact form $\alpha$ on $L_k^{2n-1}$ 
defining the standard contact structure
has closed Reeb orbits.
As above, denote by $\alpha_0$ the standard contact form,
and by $f$ the function such that $\alpha = e^f \alpha_0$.
Then $h = e^{-f}$ is the contact Hamiltonian function
(with respect to $\alpha_0$)
of the Reeb flow $\{\varphi_t^{\alpha}\}$ of $\alpha$
and, for every $m \in \mathbb{N}$,
$mh$ is the contact Hamiltonian function of $\{\varphi_{mt}^{\alpha}\}$.
Discriminant points of $\varphi_{mt}^{\alpha}$ 
correspond to closed Reeb orbits of $\alpha$ of period $mt$.
Since $h > 0$, there is $m \in \mathbb{N}$ such that $mh \geq 2 \pi$.
The constant function $2 \pi$ is the contact Hamiltonian
of $\{r_{2\pi t}\}_{t \in [0,1]}$.
By Proposition \ref{proposition: monotonicity - asymptotic} 
and Example \ref{Reeb flow}
we then have 
$\ol{\mu} \, \big([\{\varphi_{mt}^{\alpha}\}_{t \in [0,1]}]\big) \geq 
\ol{\mu} \, \big( [ \{r_{2\pi t}\}_{t \in [0,1]} ] \big) = 2n > 0$.
It thus follows from Proposition \ref{discriminant points} 
that $\alpha$ has at least one closed Reeb orbit.

\subsection*{Constructing new quasimorphisms via contact reduction,
and more applications to orderability 
and non-displaceability}\labell{relation bz}

In \cite{BZ} Borman and Zapolsky explain
how in certain situations quasimorphisms descend under contact reduction,
and use this to show that
Givental's asymptotic non-linear Maslov index on projective spaces
induces quasimorphisms on certain prequantizations of symplectic toric manifolds.
Moreover they obtain applications to orderability
and existence of non-displaceable pre-Lagrangian toric fibres.
As already observed in \cite[Remark 1.5]{BZ},
our extension of Givental's non-linear Maslov index to lens spaces
allows us to enlarge the class of spaces
to which the results of \cite{BZ} apply.

Consider a closed contact manifold $(V, \xi)$,
and suppose that it is equipped with a non-trivial monotone quasimorphism
$\nu \colon \widetilde{\Cont}_0(V, \xi) \rightarrow \R$.
Following \cite{BZ} we say that a subset $Y$ of $V$
is \emph{subheavy} with respect to $\nu$
if $\nu$ vanishes on all elements that can be represented by
a contact isotopy generated by
an autonomous Hamiltonian that vanishes on $Y$.
Suppose now that $(V,\xi)$ is also equipped with a contact $\mathbb{T}^n$-action,
and denote by $f_{\alpha}: V \rightarrow \R^n$ the moment map
with respect to a $\mathbb{T}^n$-invariant contact form $\alpha$ for $\xi$.
Recall that if $\mathbb{T}^n$ acts freely
on the level set $f_{\alpha}^{-1}(0)$
then $\alpha$ induces a contact form $\alpha'$
on the quotient $V' = f_{\alpha}^{-1}(0) / \mathbb{T}^{n}$
(see for instance \cite[Theorem 7.7.5]{Geiges}).
The contact manifold $(V', \xi' = \text{ker}(\alpha'))$
is said to be the contact reduction of $(V,\xi)$
at the level $f_{\alpha}^{-1}(0)$.
By \cite[Theorem 1.8]{BZ},
if $f_{\alpha}^{-1}(0)$ is subheavy with respect to
the non-trivial monotone quasimorphism $\nu$
then $\nu$ naturally descend to a non-trivial monotone quasimorphism
$\nu' \colon  \widetilde{\Cont}_0(V', \xi') \rightarrow \R$.
Moreover, if $\nu$ has the vanishing property then so does $\nu'$.
By \cite[Theorem 1.3]{BZ},
if a monotone symplectic toric manifold $(W,\omega)$ is \emph{even},
i.e.\ the sum of the normals of the moment polytope $\Delta \subset t^{\ast}$
is in $2 \, \mathfrak{t}_{\mathbb{Z}}$,
then there is a rescaling $a \omega$ of the symplectic form
such that the prequantization $(V,\xi)$ of $(W, a\omega)$
can be written as contact reduction of a projective space $\mathbb{RP}^{2n-1}$
at a level $f_{\alpha}^{-1}(0)$ containing the torus
$$
T_{\RP^{2n-1}} = \{ \, [z] \in \RP^{2n-1} \; \lvert \; \ |z_1|^2 = \ldots = |z_n|^2  \, \} \,.
$$
By \cite[Lemma 1.22 and Theorem 1.11 (i)]{BZ},
$T_{\RP^{2n-1}}$ is subheavy with respect to
the asymptotic non-linear Maslov index $\overline{\mu}$
and so, by \cite[Proposition 1.10(ii)]{BZ},
$f_{\alpha}^{-1}(0)$ is also subheavy.
It follows that $\overline{\mu}$ descends to a non-trivial monotone quasimorphism
on $\widetilde{\Cont}_0(V, \xi)$ with the vanishing property.
By \cite[Theorem 1.28]{BZ},
if a contact manifold $(V,\xi)$ admits a non-trivial monotone quasimorphism
$\nu \colon \widetilde{\Cont}_0(V, \xi) \rightarrow \R$
then it is orderable;
by \cite[Theorem 1.17]{BZ},
if moreover $(V,\xi)$ is the prequantization of a symplectic toric manifold
and the quasimorphism $\nu$ also has the vanishing property
then $V$ has a non-displaceable pre-Lagrangian toric fibre.
The conclusion is thus that any monotone even symplectic toric manifold
has a prequantization that is orderable and has a
non-displaceable pre-Lagrangian toric fibre.

In the case of lens spaces,
repeating the proof of \cite[Lemma 1.22]{BZ}
one sees that
$$
T_{L_k^{2n-1}}=\{ \, [z] \in L_k^{2n-1} \; \lvert \; \ |z_1|^2 = \ldots = |z_n|^2  \, \} \subset L_k^{2n-1}(1,\ldots,1)
$$
is subheavy with respect to the asymptotic non-linear Maslov index on $L_k^{2n-1}(1,\ldots,1)$.
Consider now a compact monotone symplectic toric manifold $(W^{2n}, \omega)$.
Write the moment polytope as
$\Delta = \{\, x \in \mathfrak{t}^{\ast} \text{ , } \langle \nu_j,x \rangle + \lambda \geq 0 \text{ for } j=1,\cdots,d\,\}$,
where $d$ is the number of facets and $\nu_j \in \mathfrak{t}$
are vectors normal to the facets and primitive in the integer lattice
$\mathfrak{t}_{\mathbb{Z}} = \text{ker} \, (\text{exp} \colon \mathfrak{t} \rightarrow \mathbb{T}^n)$.
Suppose that, for some integer $k \geq 2$,
\begin{equation}\labell{e: condition toric}
\sum_{j=1}^d \nu_j \in k \, \mathfrak{t}_{\mathbb{Z}}\,.
\end{equation}
Then the same argument as in the proof of \cite[Theorem 1.3]{BZ}
shows that there is a rescaling\footnote{
If $\omega$ is rescaled
so that $[\omega] = c_1 (TW)$
then we can take $a = k$.
Indeed, the moment polytope for $(W,k\omega)$
can be written as
$\Delta = \{\, x \in \mathfrak{t}^{\ast} \text{ , } \langle \nu_j,x \rangle + k \geq 0 \text{ for } j=1,\cdots,d\,\}$
and the prequantization $(V,\xi)$ of $(W, k \,\omega)$
corresponds to a cone in $\R^{n+1}$
with primitive inward normals $(\nu_j,k)\in \mathfrak{t} \times \R$, $j=1,\cdots,d$.
The contact toric manifold $(V,\xi)$ is a contact reduction of $(S^{2d-1}, \ker \alpha_{std})$
by a subgroup $K$ of $T^d$ containing $[\frac{1}{k}, \ldots, \frac{1}{ k}] \in \R^d/\Z^d = T^d$
(because $\sum \frac 1 k (\nu_j,k)$ is in the lattice of $\mathfrak{t} \times \R$; see \cite{Lerman}).
In fact it would be enough to take $\omega = k \eta$
where $\eta$ denotes the primitive integral class in the direction of $c_1(TW)$.}
$a \omega$
of the symplectic form such that
the prequantization of $(W,a\omega)$
can be written as contact reduction of  $L_k^{2n-1}(1,\ldots,1)$
at a level $f_{\alpha}^{-1}(0)$ containing $T_{L_k^{2n-1}}$. 
Therefore, such a prequantization
admits a non-trivial monotone quasimorphism
with the vanishing property,
and so it is orderable and it contains
a non-displaceable pre-Lagrangian toric fibre.

\begin{ex}
A compact monotone symplectic toric manifold $(W^{2n}, \omega)$
satisfying condition \eqref{e: condition toric}
can be obtained by the following generalization
of \cite[Example (i) of page 385]{BZ}.
Consider the $\CP^1$-bundle over $\CP^n$
obtained, for $1\leq k \leq n$,
as the projectivization $\mathbb{P}(\mathbf{1} \oplus \mathcal{O}(k))$
of the direct sum of a trivial line bundle
with the bundle $\mathcal{O}(k)$ over $\CP^n$.
This manifold can be equipped with a monotone symplectic structure,
and the inward normals of the corresponding moment polytope
are $e_1,\ldots,e_n, e_{n+1},-e_{n+1},ke_{n+1}-e_1-\ldots-e_n$.
\eoe
\end{ex}

\begin{rmk}\labell{rmk BZ}
One can show that for any compact monotone symplectic toric manifold $(W, \omega)$
the prequantization of $W$ with appropriately rescaled $\omega$
can be written as a contact reduction of  $L_k^{2n-1}(\underline{w})$
at a level containing $T_{L_k^{2n-1}(\underline{w})}$.
However, if $\underline{w} \neq (1,\ldots,1)$
we do not know whether $T_{L_k^{2n-1}(\underline{w})}$ is subheavy
and so we cannot conclude that this prequantization has an induced quasimorphism.
In order to prove that $T_{L_k^{2n-1}} \subset L_k^{2n-1}(1,\ldots,1)$ is subheavy
one uses the fact that the Clifford torus in
$\CP^{n-1}=\CP^{n-1}(1,\ldots,1)$ 
is the unique non-displaceable orbit of the standard torus action.
A similar statement is not true in general for weighted projective spaces:
for example $\CP(1,3,5)$ contains a $2$-dimensional family
of non-displaceable Lagrangian toric fibres \cite{WilsonWoodward}.
\end{rmk}


\appendix

\section{On the construction of generating functions}
\labell{section appendix: on the construction}

In Proposition \ref{p:composition formula} we proved that
if $\Phi^{(1)}$ and $\Phi^{(2)}$ are Hamiltonian symplectomorphisms
of $\mathbb{R}^{2n}$ with generating functions 
$F_1 \colon \mathbb{R}^{2n} \times \R^{2nN_1}\rightarrow \mathbb{R}$ 
and $F_2 \colon \mathbb{R}^{2n}\times \R^{2nN_2} \rightarrow \mathbb{R}$ respectively
then the function 
$F_1 \, \sharp \, F_2 \colon \mathbb{R}^{2n} \times (\mathbb{R}^{2n} \times \mathbb{R}^{2n} \times  \R^{2nN_1}\times \R^{2nN_2}) 
\rightarrow \mathbb{R}$
defined by
$$
F_1 \, \sharp \, F_2 (q; \zeta_1, \zeta_2, \nu_1, \nu_2) 
= F_1 (\zeta_1,\nu_1) + F_2 (\zeta_2,\nu_2) - 2 \left<\zeta_2 - q,i(\zeta_1-q)\right>
$$
is a generating function for the composition $\Phi = \Phi^{(2)} \circ \Phi^{(1)}$. 
Here we present two alternative proofs 
of this fact in terms of symplectic reduction,
we generalize the composition formula
to the case of any even number of factors
and discuss its relation with  the method of broken trajectories 
by Chaperon, Laudenbach and Sikorav \cite{Ch1, LS, Sikorav 1, Sikorav 2}.

Recall that if $V$ is a coisotropic submanifold of a symplectic manifold $(W,\omega)$
then the kernel of the restriction of $\omega$ to $V$ is an integrable distribution.
If the space of leaves $V / \hspace{-0.1cm}\sim$ is a manifold
then it inherits a symplectic form $\omega$,
and is said to be the symplectic reduction of $(W,\omega)$ along $V$.
If $L$ is a Lagrangian submanifold of $W$ that is transverse to $V$
then the restriction to $L \cap V$ of the projection
$V \rightarrow V / \hspace{-0.1cm}\sim$ is a Lagrangian immersion.
For instance, consider a fibre bundle $p \colon E \rightarrow B$.
The fibre conormal bundle $N_E^{\ast}$
is a coisotropic submanifold of $T^{\ast}E$,
and the symplectic reduction can be identified with $T^{\ast}B$.
If $F \colon E \rightarrow \mathbb{R}$ is a generating function
then the Lagrangian immersion $i_F \colon \Sigma_F \rightarrow T^{\ast}B$
described in Section \ref{sec:generating functions}
is the reduction of $dF \subset T^{\ast}E$ with respect to $N_E^{\ast}$.

\subsection*{First interpretation}

The first interpretation of Proposition \ref{p:composition formula}
in terms of symplectic reduction that we present 
is an adaptation to our composition formula 
of the discussion in Th\'{e}ret \cite[Section I.3]{Theret - Thesis}.
We use three basic properties of generating functions
(Lemmas \ref{lemma: direct sum}, \ref{lemma: add dh} and \ref{lemma: gf and reduction})
whose verification is immediate and therefore left to the reader.

\begin{lemma}\labell{lemma: direct sum}
If $L_1 \subset T^{\ast}B_1$ has generating function
$F_1 \colon B_1 \times \mathbb{R}^{N_1} \rightarrow \mathbb{R}$
and $L_2 \subset T^{\ast}B_2$ has generating function
$F_2 \colon B_2 \times \mathbb{R}^{N_2} \rightarrow \mathbb{R}$
then the product $L_1 \times L_2 \subset T^{\ast}B_1 \times T^{\ast}B_2 \equiv T^{\ast}(B_1\times B_2)$
has generating function 
$F_1 \oplus F_2 \colon
(B_1\times B_2) \times (\mathbb{R}^{N_1} \times \mathbb{R}^{N_2}) \rightarrow \mathbb{R}$.
\end{lemma}

\begin{lemma}\labell{lemma: add dh}
Suppose that $L \subset T^{\ast}B$ has generating function 
$F \colon B \times \mathbb{R}^N \rightarrow \mathbb{R}$,
and consider a symplectomorphism $A_h$ of $T^{\ast}B$
of the form $A_h (q,p) = \big(q, p + dh(q)\big)$
for some function $h: B \rightarrow \mathbb{R}$.
Then $A_h(L) \subset T^{\ast}B$ has generating function $F + h$.
\end{lemma}

\begin{lemma}\labell{lemma: gf and reduction}
If a Lagrangian submanifold $L$ of $T^{\ast}(\mathbb{R}^{n} \times \mathbb{R}^{m})$
has a generating function
$F \colon (\mathbb{R}^{n} \times \mathbb{R}^{m}) \times \mathbb{R}^N \rightarrow \mathbb{R}$,
then the reduction $\overline{L} \subset T^{\ast}\mathbb{R}^{n}$ of $L$
with respect to the coisotropic submanifold 
$V = \mathbb{R}^{n} \times \mathbb{R}^{m} \times (\mathbb{R}^{n})^{\ast} \times 0$
of $T^{\ast}(\mathbb{R}^{n} \times \mathbb{R}^{m})$
has a generating function
$\overline{F} \colon \mathbb{R}^{n} \times (\mathbb{R}^{m} \times \mathbb{R}^N) \rightarrow \mathbb{R}$,
$\overline{F}(\zeta_1; \zeta_2,\nu) = F(\zeta_1,\zeta_2; \nu)$.
\end{lemma}

Since $\Gamma_{\text{id}}$
has generating function $\mathbb{R}^{2n} \rightarrow \mathbb{R}$, $q \mapsto 0$,
Lemma \ref{lemma: direct sum} implies that the function 
$\mathbb{R}^{2n} \times \mathbb{R}^{2n} \times \mathbb{R}^{2n} \times  (\R^{2nN_1}\times \R^{2nN_2})
\rightarrow \mathbb{R}$ defined by
$$
(q, \zeta_1, \zeta_2;\nu_1, \nu_2) \mapsto F_1 (\zeta_1, \nu_1) + F_2 (\zeta_2, \nu_2)
$$
is a generating function for 
$\Gamma_{\text{id}} \times \Gamma_{\Phi^{(1)}} \times \Gamma_{\Phi^{(2)}} 
\subset 
T^{\ast} (\mathbb{R}^{2n} \times \mathbb{R}^{2n} \times \mathbb{R}^{2n})$. 
By applying Lemma \ref{lemma: add dh} with
$$
h \colon \mathbb{R}^{2n} \times \mathbb{R}^{2n}  \times \mathbb{R}^{2n} \rightarrow \mathbb{R},
\ \ h(q,\zeta_1,\zeta_2) = - 2 \left< \zeta_2 - q, i(\zeta_1 - q) \right>
$$
we obtain that the function 
$\mathbb{R}^{2n} \times \mathbb{R}^{2n} \times \mathbb{R}^{2n} \times  (\R^{2nN_1}\times \R^{2nN_2})
\rightarrow \mathbb{R}$
defined by
\begin{equation}\labell{equation: interpretation as symplectic reduction}
(q,\zeta_1,\zeta_2;\nu_1, \nu_2) \mapsto 
F_1(\zeta_1, \nu_1) + F_2(\zeta_2, \nu_2) - 2 \left<\zeta_2 - q,i(\zeta_1-q)\right>
\end{equation}
is a generating function for
$A_h \big( \Gamma_{\text{id}} \times \Gamma_{\Phi^{(1)}} \times \Gamma_{\Phi^{(2)}}\big)
\subset  T^{\ast} (\mathbb{R}^{2n} \times \mathbb{R}^{2n}  \times \mathbb{R}^{2n})$.
The function \eqref{equation: interpretation as symplectic reduction}
is equal to $F_1 \, \sharp \, F_2$, 
except that in the latter $\zeta_1$ and $\zeta_2$ are fibre variables. 
Thus, it follows from Lemma \ref{lemma: gf and reduction}
that $F_1 \, \sharp \, F_2$ is a generating function for the reduction of 
$$
L := A_h \big( \Gamma_{\text{id}} \times \Gamma_{\Phi^{(1)}} \times \Gamma_{\Phi^{(2)}}\big)
\subset  T^{\ast} (\mathbb{R}^{2n} \times \mathbb{R}^{2n}  \times \mathbb{R}^{2n})
$$
along the coisotropic submanifold
$$
V = \mathbb{R}^{2n} \times \mathbb{R}^{2n}  \times \mathbb{R}^{2n} \times (\R^{2n})^* \times \, \{0\}  \times \, \{0\}
$$
of $T^*(\mathbb{R}^{2n} \times \mathbb{R}^{2n}  \times \mathbb{R}^{2n})$.
We are left to prove that such reduction 
is equal to $\Gamma_{\Phi}$.
Observe that the reduction $V \rightarrow V / \hspace{-0.1cm}\sim$ sends a point
$(q,\zeta_1,\zeta_2,\xi,0,0)$ to $(q,\xi)$.
We have
\begin{align*}
L &= A_h \big(\Gamma_{\text{id}} \times \Gamma_{\Phi^{(1)}} \times \Gamma_{\Phi^{(2)}}\big) \\
&= A_h \Big( \Big\{ \Big(q,\frac{z_1+\Phi^{(1)}(z_1)}{2},\frac{z_2+\Phi^{(2)}(z_2)}{2}, 
0, i\big(z_1-\Phi^{(1)}(z_1)\big), i \big(z_2-\Phi^{(2)}(z_2)\,\big)\Big)\Big\} \Big)\\
&= \Big\{ \Big(q,\frac{z_1+\Phi^{(1)}(z_1)}{2}, \frac{z_2 + \Phi^{(2)}(z_2)}{2}, 0, i(z_1-\Phi^{(1)}(z_1)),i(z_2-\Phi^{(2)}(z_2)\,\big)\Big) \\
&\quad \quad + dh \big(q,\frac{z_1+\Phi^{(1)}(z_1)}{2},\frac{z_2 + \Phi^{(2)}(z_2)}{2}\big)\Big\}\\
&= \Big\{ \Big(q,\frac{z_1+ \Phi^{(1)}(z_1)}{2},\frac{z_2 + \Phi^{(2)}(z_2)}{2}, 
i \big(z_1 + \Phi^{(1)}(z_1) - z_2 - \Phi^{(2)}(z_2)\big), \\
&\quad \quad i \big(z_1 - \Phi^{(1)}(z_1) + z_2 + \Phi^{(2)}(z_2) - 2q\big),
i(z_2 - \Phi^{(2)}(z_2) - z_1 - \Phi^{(1)}(z_1) + 2q\,\big)\Big)\Big\}\,.
\end{align*}

The intersection $L \cap V$ is given by the points  in the above set that satisfy
\begin{displaymath}
\begin{cases}
2 q = z_1 - \Phi^{(1)}(z_1) + z_2 + \Phi^{(2)}(z_2) \\
2 q = z_1 + \Phi^{(1)}(z_1) - z_2 + \Phi^{(2)}(z_2)
\end{cases}
\end{displaymath}
hence $z_2 = \Phi^{(1)}(z_1)$
and $q = \frac{z_1 + \Phi^{(2)} (z_2)}{2} = \frac{z_1+ \Phi(z_1)}{2}$.
The reduction is thus given by
\begin{align*}
(L\cap V) / \sim &= \Big\{ \Big(q, i \big(z_1 + \Phi^{(1)}(z_1) - z_2 - \Phi^{(2)}(z_2)\big)\Big)
\text{ with } z_2 = \Phi^{(1)}(z_1) \text{ and } q = \frac{z_1+ \Phi(z_1)}{2} \Big\}\\
&= \Big\{ \Big(\frac{z_1 + \Phi (z_1)}{2}, i \big(z_1 - \Phi(z_1)\big)\Big) \Big\}
= \Gamma_{\Phi}
\end{align*}
as we wanted.

\subsection*{Second interpretation}

The second alternative proof
of Proposition \ref{p:composition formula}
that we discuss
uses symplectic reduction at the level of graphs 
and is based on the fact, immediate to verify,
that the function 
$$
h \colon \mathbb{R}^{2n} \times \mathbb{R}^{2n}  \times \mathbb{R}^{2n} \rightarrow \mathbb{R} 
 \; , \quad 
h(q,\zeta_1,\zeta_2) = - 2 \left< \zeta_2 - q, i(\zeta_1 - q) \right>
$$
is a generating function for the symplectomorphism
$$
\sigma \colon \mathbb{R}^{2n} \times \mathbb{R}^{2n}  \times \mathbb{R}^{2n}
\rightarrow \mathbb{R}^{2n} \times \mathbb{R}^{2n}  \times \mathbb{R}^{2n} \; , \quad
\sigma(z_0,z_1,z_2) = (z_2,z_0,z_1) \,.
$$
Proposition \ref{p:composition formula} 
can be deduced from this fact as follows.
For simplicity of notation we assume that the generating functions 
of $\Phi^{(1)}$ and $\Phi^{(2)}$ have no fibre variables
(the general case does not present any additional difficulty).
Suppose thus that $\Phi^{(1)}$ and $\Phi^{(2)}$ are Hamiltonian symplectomorphisms
of $\mathbb{R}^{2n}$ with generating functions 
$F_1 \colon \mathbb{R}^{2n} \rightarrow \mathbb{R}$ 
and $F_2 \colon \mathbb{R}^{2n} \rightarrow \mathbb{R}$ respectively,
and consider the function 
$F = F_1 \, \sharp \, F_2 \colon \mathbb{R}^{2n} \times (\mathbb{R}^{2n} \times \mathbb{R}^{2n}) \rightarrow \mathbb{R}$,
$F (q; \zeta_1, \zeta_2) 
= F_1 (\zeta_1) + F_2 (\zeta_2) + h(q,\zeta_1,\zeta_2)$.
Denote the coordinates of $T^{\ast} (\R^{2n} \times \R^{2n} \times \R^{2n})$
by $(q,\zeta_1,\zeta_2;p_0,p_1,p_2)$,
and recall that, by the definition of generating function,
the Lagrangian submanifold $L_F$ of $T^{\ast} \R^{2n}$ generated by $F$
is the symplectic reduction of $dF \subset T^{\ast} (\R^{2n} \times \R^{2n} \times \R^{2n})$
with respect to the fibre conormal bundle 
$V = \{\, p_1 = p_2 = 0\,\}$.
The submanifold
$$
V_{\Phi} := \Big\{ \big(q,\zeta_1,\zeta_2,p_0,- \frac{\partial F_1}{\partial \zeta_1} (\zeta_1),- \frac{\partial F_2}{\partial \zeta_2} (\zeta_2)\big) \Big\}
$$
is also coisotropic
(it is $T^*(\R^{2n}) \times (-dF_1) \times (-dF_2)$
hence the product of one symplectic and two Lagrangian factors).
The projection $V_\Phi \to V_\Phi\, /  \hspace{-0.05 cm}\sim$ forgets the Lagrangian factors
and is therefore given by 
$(q,\zeta_1,\zeta_2,p_0,p_1,p_2) \mapsto (q,p_0)$.

The plane $dh$ intersects $V_\Phi$ at the points defined by the conditions
$$
p_0 = 2i (\zeta_1- \zeta_2); \quad \quad -\frac{\partial F_1}{\partial \zeta_1} = 2i(\zeta_2-q);
\quad \quad - \frac{\partial F_2}{\partial \zeta_2} = - 2i (\zeta_1 - q) \,.
$$
These are equivalent to the conditions that
$(q,\zeta_1,\zeta_2)$ is a fibre critical point and $p_0 = \frac{\partial F}{\partial q}$.
It follows that $L_F$ also equals the symplectic reduction of $dh$ with respect to $V_\Phi$.

Since $\tau^{-1}(dh) = \text{gr}(\sigma)$,
our problem is reduced to proving that the reduction of 
$\text{gr}(\sigma)$ along
$$
\tau^{-1} (V_{\Phi}) = 
\big\{\, (z_0,z_1,z_2; Z_0,Z_1,Z_2)  \; \lvert \; 
Z_1 = (\Phi^{(1)})^{-1}(z_1) \text{ and } Z_2 = (\Phi^{(2)})^{-1}(z_2) \, \big\}
$$
is equal to the graph of $\Phi$.
But, the projection
$\tau^{-1} (V_{\Phi}) \rightarrow \tau^{-1} (V_{\Phi}) \, /  \hspace{-0.05 cm}\sim$ is given by
$$
(z_0,z_1,z_2; Z_0,Z_1,Z_2) \mapsto (z_0,Z_0)
$$
and $\text{gr}(\sigma) \cap \tau^{-1} (V_{\Phi})$ is the set of points
$(z_0,z_1,z_2,z_2,z_0,z_1)$ such that
$z_0 = (\Phi^{(1)})^{-1}(z_1)$ and $z_1 = (\Phi^{(1)})^{-1}(z_2)$.
The projection sends such a point to
$(z_0,z_2) = \big(z_0, \Phi(z_0)\big)$.

\subsection*{Generalization to any even number of factors}

The second interpretation of Proposition \ref{p:composition formula}
in terms of symplectic reduction permits to easily generalize the composition formula 
to the case of any even number of factors
(obtaining an alternative proof 
of Proposition \ref{existence of generating functions}).
Assume that $N$ is even,
and consider the symplectomorphism $\sigma$ of  $\R^{2n} \times (\mathbb{R}^{2n})^N$
defined by
$$
\sigma (z_0,z_1,\cdots,z_N) = (z_N, z_0, z_1, \cdots, z_{N-1})\,.
$$
A straightforward calculation shows that
the function $h \colon \R^{2n} \times \R^{2nN} \to \R$ given by
$$ 
h (q,\zeta_1,\ldots,\zeta_N) = 
    2 \sum_{1 \leq j \leq N} (-1)^{j} \left< \zeta_{j} \ , \ iq \right>
 + 2 \sum_{1 \leq j < \ell \leq N} (-1)^{j+\ell-1} \left< \zeta_j , i \zeta_{\ell} \right> 
 $$
is a generating function for $\sigma$.

\begin{Proposition}
\labell{app proposition: construction of generating functions}
Suppose that, for each $j = 1 , \ldots , N$,
$\Phi^{(j)}$ is a Hamiltonian diffeomorphism of $\R^{2n}$
with generating function $F_{j} \colon \R^{2n} \to \R$. 
Then the function
$ F \colon \R^{2n} \times \R^{2nN} \to \R$
defined by 
$$
F (q ; \zeta_1,\ldots,\zeta_N) = F_1(\zeta_1) + \ldots + F_N(\zeta_N) + h (q, \zeta_1,\ldots,\zeta_N)
$$
is a generating function for the composition
$\Phi = \Phi^{(N)} \circ \cdots \circ \Phi^{(1)}$.
\end{Proposition}

\begin{proof}
Denote the coordinates on $T^{\ast} (\R^{2n} \times \R^{2nN})$
by $(q,\zeta_1, \cdots, \zeta_N ; p_0, p_1, \cdots, p_N)$.
The Lagrangian submanifold $L_F$ of $T^{\ast} \R^{2n}$ generated by $F$
is the symplectic reduction of 
$dh \subset T^{\ast} (\R^{2n} \times \R^{2nN})$
with respect to the coisotropic submanifold
$$
V_{\Phi} = \big\{\, p_1 = - \frac{\partial F_1}{\partial \zeta_1} (\zeta_1) \, , \, \cdots \, , \, 
p_N = - \frac{\partial F_N}{\partial \zeta_N} (\zeta_N)\, \big\} \,.
$$
Since
$\tau^{-1}(dh) = \text{gr}(\sigma)$,
our problem is reduced to proving that the reduction of 
$\text{gr}(\sigma)$ along
$$
\tau^{-1} (V_{\Phi}) = 
\{\, (z_0,z_1,\cdots,z_N; Z_0,Z_1,\cdots,Z_N) \; \lvert \; 
Z_1 = (\Phi^{(1)})^{-1}(z_1), \cdots, Z_N = (\Phi^{(N)})^{-1}(z_N) \,\}
$$
is equal to the graph of $\Phi$.
But, the projection
$\tau^{-1} (V_{\Phi}) \rightarrow \tau^{-1} (V_{\Phi}) \, / \hspace{-0.05cm}\sim$ is given by
$$
(z_0,z_1, \cdots,z_N; Z_0,Z_1,\cdots, Z_N) \mapsto (z_0,Z_0)
$$
and $\text{gr}(\sigma) \cap \tau^{-1} (V_{\Phi})$ is the set of points
$(z_0,z_1, \cdots,z_N, z_N,z_0,z_1,\cdots,z_{N-1})$ such that
$z_{j-1} = (\Phi^{(j)})^{-1}(z_j)$ for $j=1, \cdots, N$.
The projection sends such a point to 
$\big(z_0, \Phi(z_0)\big)$.
\end{proof}

\begin{rmk}\labell{remark: cf Givental}
In the case of $\RP^{2n-1}$ Givental does not use directly
the generating function given by Proposition 
\ref{app proposition: construction of generating functions}. 
Instead he studies a path $- \Phi_t$ starting at $- \text{id}$
by looking at a family of generating functions $F_t$
of the path $\Phi_t$ (starting at the identity).
For fibre critical points of $F_t$ we have $q = \frac{z_1 + \Phi_t (z_1)}{2}$.
Thus critical points of the restriction of $F_t$ to the fibre over $q=0$
correspond to fixed points of $- \Phi_t$.
So, instead of looking at the whole function $F_t$,
Givental only considers
the restriction of $F_t$ to the fibre over $q = 0$.
\eor
\end{rmk}

\subsection*{Relation with the method of broken trajectories}

We now discuss the relation between
the composition formula of
Proposition \ref{p:composition formula}
and the construction of generating functions
via the method of broken trajectories,
due to Chaperon, Laudenbach and Sikorav
\cite{Ch1, LS, Sikorav 1, Sikorav 2}.

The method of broken trajectories
is used to construct generating functions
for Lagrangian submanifolds of a cotangent bundle $T^{\ast}B$
that are Hamiltonian isotopic to the zero section.
The idea is to first interpret the symplectic action functional 
on a space of paths in $T^{\ast}B$
as a generating function with infinite dimensional domain,
and then to construct a finite dimensional approximation.
Recall that the symplectic action functional
associated to a time-dependent Hamiltonian function $H_t$
on an exact symplectic manifold $(W,\omega = - d\lambda)$
is the functional $\mathcal{A}_H$
on the space of paths $\gamma \colon [0,1] \rightarrow W$
which is defined by
$$
\mathcal{A}_H(\gamma) =
\int_{0}^{1} \lambda \Big(\frac{\partial \gamma}{\partial t} \Big) - H_t\,\big(\gamma(t)\big)\:dt\,.
$$
A path $\gamma$ is a critical point of $\mathcal{A}_H$ 
with respect to variations with fixed endpoints
if and only if it is a trajectory of the Hamiltonian flow of $H_t$. 
Consider now the case where $W$
is a cotangent bundle $T^{\ast}B$.
Let $E$ be the space of paths $\gamma \colon [0,1] \rightarrow T^{\ast}B$
that begin at the zero section,
and see it as the total space of a fibre bundle over $B$
with projection $p \colon E \rightarrow B$ given by
$p(\gamma) = \pi\big(\gamma(1)\big)$,
where $\pi$ is the projection of $T^{\ast}B$ into $B$. 
Given a time-dependent Hamiltonian function $H_t: T^{\ast}B \rightarrow \mathbb{R}$,
consider the functional $F \colon E \rightarrow \mathbb{R}$ defined by 
$$
F(\gamma) = \mathcal{A}_H(\gamma)\,.
$$
The fibre critical points of $F \colon E \rightarrow \mathbb{R}$
are the trajectories of the Hamiltonian flow of $H_t$,
and the covector $v^{\ast}(\gamma)$
associated to a fibre critical point $\gamma$ is $\gamma(1)$.
Thus, $F$ generates the image of the  zero section
by the time-1 map of the Hamiltonian flow of $H_t$.
Although $F$ is not a generating function in the usual sense,
because its domain is infinite dimensional,
a finite dimensional reduction can be obtained as follows.
Let $N$ be an integer.
Consider the direct sum
$\bigoplus_{i=1}^{N-1} TB \oplus T^{\ast}B$,
and denote its elements by expressions of the form $e = (q,X,P)$,
where $q$ is a point of $B$,
$X = (X_1,\cdots,X_{N-1})$ is an $(N-1)$-tuple of vectors $X_j \in T_qB$
and $P = (P_1,\cdots,P_{N-1})$ is
an $(N-1)$-tuple of covectors $P_j \in T_{q}^{\ast}B$. 
Let $\mathcal{U}$ be a neighborhood of the zero section of $TB$,
and consider the subspace $E_N$ of $\bigoplus_{i=1}^{N-1} TB \oplus T^{\ast}B$
that is formed by those elements $e = (q,X,P)$
such that all $X_j$ belong to $\mathcal{U}$.
If $\mathcal{U}$ is sufficiently small then
an element $e = (q,X,P)$ of $E_N$ can be interpreted
as a broken Hamiltonian trajectory of $H_t$,
with $N$ smooth pieces and $N-1$ jumps, as follows.
The first smooth piece $\gamma_1$ is obtained
by following the Hamiltonian flow
of $H_t$ for $t\in[0,\frac{1}{N}]$
from the point $(q,0)$
to a point of $T^{\ast}B$ that we denote by $\eta_1^{-}$. 
The second smooth piece $\gamma_2$
starts from a point $\eta_1^{+}$,
which is uniquely determined 
by $\eta_1^{-}$, $X_1$ and $P_1$ in a way that we describe later,
and follows the flow of $H_t$ for $t\in[\frac{1}{N}, \frac{2}{N}]$
to a point $\eta_2^-$.
We continue in this way to obtain the whole broken trajectory. 
In order to describe the jumps we fix a Riemannian metric on $B$, 
and consider the associated Levi--Civita connections on $TB$ and $T^{\ast}B$. 
The point $\eta_1^{+} = (q_1^{+},p_1^{+})$ is determined 
by $\eta_1^{-} = (q_1^{-},p_1^{-})$, $X_1$
and $P_1$ in the following way. 
Denote by $\overline{X}_1\in T_{q_1^{-}}B$ 
and $\overline{P}_1\in T_{q_1^{-}}^{\phantom{q_1}\ast}B$
the vector and the covector obtained by parallel transport
of $X_1$ and $P_1$
along the projection to $B$ of the path $\gamma_1$.
Since $X_1$ is in $\mathcal{U}$,
which is assumed to be sufficiently small,
we can then define $q_1^{+} = \text{exp}_{q_1^{-}} (\overline{X}_1)$ and 
$p_1^{+} = \big( (d \, \text{exp}_{q_1^-}\lvert_{\overline{X}_1})^*\big)^{-1} (P_1)$. 
The other jumps are defined similarly.
Consider the projection $p \colon E_N\rightarrow B$
that sends a point $e = (q,X,P)$
to the projection $q_N^-$ to $B$ of the endpoint $\eta_N^-$ 
of the broken Hamiltonian trajectory associated to $e$. 
Define a function $F \colon E_N \rightarrow \mathbb{R}$ by
\begin{equation}\label{e: F}
F(e) = \sum_{j = 1}^{N} \mathcal{A}_H (\gamma_j) + \sum_{j = 1}^{N-1} P_j (X_j)\,.
\end{equation}
Denote the flow of $H_t$ by $\{\varphi_t\}_{t \in [0,1]}$,
and assume that, for all $j = 1$, $\cdots$, $N$,
the symplectomorphism $\varphi_{\frac{j}{N}} \circ (\varphi_{\frac{j-1}{N}})^{-1}$
is sufficiently $\mathcal{C}^1$-small.
Then \cite{Sikorav 1} the fibre critical points of $F$ are the unbroken trajectories,
and the covector $v^{\ast}(e)$
associated to a fibre critical point $e$ is given by $v^{\ast}(e) = p_N^-$.
Thus, $F \colon E_N \rightarrow \mathbb{R}$ is a generating function 
for the image of the zero section by $\varphi_1$. 

If $\varphi_1$ is already sufficiently $\mathcal{C}^1$-small
then the above construction (for $N = 1$) reduces to the following.
The space $E_1$ can be identified with $B$,
by associating to a point $q$ of $B$
the Hamiltonian trajectory $\gamma_{q}$ of $H_t$ starting at $q$.
We see $E_1$ as the total space of a fibre bundle over $B$
by the diffeomorphism
\begin{equation}\labell{equation: fibre bundle one piece}
E_1 \rightarrow B \quad , \quad q \mapsto \pi \big(\gamma_{q}(1)\big)\,.
\end{equation}
Then the function $F \colon E_1 \rightarrow \mathbb{R}$, $F(q) = \mathcal{A}_H(\gamma_{q})$
is a generating function,
with respect to the projection \eqref{equation: fibre bundle one piece},
for the image of the zero section by $\varphi_1$.
In other words,
the function on $B$ obtained by precomposing $F$
with the inverse of \eqref{equation: fibre bundle one piece}
is a generating function
(with respect to the projection $B \rightarrow B$
given by the identity)
for the image of the zero section by $\varphi_1$.

Returning to a general $N$,
in the case when $B$ is $\mathbb{R}^n$ with the Euclidean metric
we can identify $E_N$ with the product
$\mathbb{R}^n \times (\mathbb{R}^n)^{N-1} \times (\mathbb{R}^n)^{N-1}$.
For an element 
$e = (q,X_1,\ldots,X_{N-1},P_1,\ldots,P_{N-1})$ 
we then have $q_j^+ = q_j^- + X_j$, $p_j^+ = P_j$
and $P_j(X_j) = \langle P_j,X_j \rangle$.

We now discuss how the composition formula
of Proposition \ref{p:composition formula}
is related to this construction.
We are interested in Lagrangians of $T^{\ast}\mathbb{R}^{2n}$
of the form $\Gamma_{\Phi} = \tau \big(\text{gr} (\Phi)\big)$,
where
$\tau \colon \overline{\mathbb{R}^{2n}} \times \mathbb{R}^{2n} \rightarrow T^{\ast}\mathbb{R}^{2n}$
is the identification \eqref{e:identification}.
Any such Lagrangian is Hamiltonian isotopic to the zero section
by the Hamiltonian isotopy of $T^{\ast}\mathbb{R}^{2n}$
that corresponds under $\tau$
to a Hamiltonian isotopy of the form $\text{id} \times \Phi_t$.
The Hamiltonian function thus satisfies
\begin{equation}\label{e: H}
H_t(\tau(z,Z)) = H_t \big(\tau(z+a,Z)\big) =
H_t \big(\tau(z,Z) + (\tfrac a 2,ia)\big) \quad \text{ for all } a \in \R^{2n} \equiv \mathbb{C}^n\,.
\end{equation}
Suppose now that the flow $\{\Psi_t\}_{t \in [0,1]}$ of $H_t$
is the concatenation of pieces
$\{\Psi_t\}_{t \in [\frac{j-1}{N},\frac{j}{N}]}$
so that each $\Psi^{(j)} := \Psi_{\frac{j}{N}} \circ {\Psi_{\frac{j-1}{N}}}^{-1}$
has a generating function $F_j: \mathbb{R}^{2n} \rightarrow \mathbb{R}$.
Using \eqref{e: H} we can relate the action terms in \eqref{e: F}
with the action of Hamiltonian trajectories
starting at the zero section,
and thus with the functions $F_j$.
Observe first that $\mathcal{A}_H(\gamma_1) = F_1 (\zeta_1)$
where $\zeta_1 := q_1^- = \frac{q + \Psi^{(1)}(q)}{2}$.
For $j = 2, \cdots, N$ set
$$
\widetilde{\gamma}_{j}(t) = \gamma_{j}(t) + (\tfrac{iP_{j-1}}{2}, - P_{j-1}) \,.
$$
The $\widetilde{\gamma}_{j}$ are Hamiltonian trajectories
starting at $(q_{j-1}^{-} + X_{j-1} + \frac{i P_{j-1}}{2}, 0)$,
and
$$
\mathcal A_H (\widetilde{\gamma}_{j}) - \mathcal A_H(\gamma_{j})
= \int_0^1 \lambda \Big(\frac{\partial \widetilde{\gamma}_{j}}{\partial t}\Big) - \lambda \Big(\frac{\partial \gamma_{j}}{\partial t}\Big) \, dt
= - \, \langle P_{j-1} \,, \, q_{j}^- - (q_{j-1}^- + X_{j-1}) \rangle \,.
$$
Set
$$
\zeta_j = \frac{u_j + \Psi^{(j)}(u_j)}{2}
$$
where
$u_j = q_{j-1}^- + X_{j-1} + \frac{i P_{j-1}}{2}$.
Then $\mathcal{A}_H (\widetilde{\gamma}_j) = F_j (\zeta_j)$,
and so the function \eqref{e: F} reduces to
\begin{align*}
F(e) &= F_1 (\zeta_1) + \cdots + F_N (\zeta_N)
+ \sum_{j = 2}^{N} \, \langle P_{j-1} \,,\,  q_{j}^- - (q_{j-1}^- + X_{j-1}) \rangle
+ \sum_{j = 1}^{N-1} \, \langle P_j, X_j \rangle \\
&= F_1 (\zeta_1) + \cdots + F_N (\zeta_N)
+ \sum_{j=1}^{N-1} \, \langle P_j, q_{j+1}^- - q_j^- \rangle \,.
\end{align*}
Set $q_j := q_{j+1}^- = \zeta_{j+1} - \frac{i P_j}{2}$
for $j=1$, $\cdots$, $N-1$,
and consider the change of variables
$$
e = (q, X_1,\cdots,X_{N-1},P_1,\cdots,P_{N-1})
\mapsto (\zeta_1, \cdots,\zeta_N,q_1,\cdots,q_{N-1}) \,.
$$
Then 
\begin{align*}
\sum_{j=1}^{N-1} \, \langle P_j, q_{j+1}^- - q_j^- \rangle
&= \langle -2i (\zeta_2 - q_1), q_1 - \zeta_1\rangle
+ \sum_{j=2}^{N-1} \, \langle -2i (\zeta_{j+1} - q_j), q_j - q_{j-1} \rangle\\
&= -2 \, \langle \zeta_2 - q_1, i(q_1 - \zeta_1) \rangle
-2 \, \sum_{j = 2}^{N-1} \, \langle \zeta_{j+1} - q_j, i(q_j - q_{j-1}) \rangle
\end{align*}
and so the function \eqref{e: F}
reduces to the function
$\big(\big( (F_1 \, \sharp \, F_2) \, \sharp \ldots \big)\, \sharp \, F_N \big)$
obtained  by iteratively applying the composition formula
of Proposition \ref{p:composition formula}.


\section{The homology join}
\labell{section: additivity under join}

Recall that the $\Z_k$-join $A \ast_{\Z_k} B$  of subsets $A$ of $L_k^{2M-1}(\underline w)$ 
and $B$ of $L_k^{2M'-1}(\underline w')$ 
is the subset of $L_k^{2(M+M')-1}(\underline w, \underline w')$
defined by~\eqref{join1}.
In this section we complete the proof of the join quasi-additivity property
of the cohomological index 
(Part (v) of Proposition \ref{proposition: cohomological index})
by proving the following lower bounds on the index
of the equivariant join: if $A$ and $B$ are closed, then
\begin{equation} \labell{lower bounds}
\text{ind} (A \ast_{\zk} B) \geq \begin{cases} 
\text{ind} (A) + \text{ind}(B) & \text{ if at least one of the indices is even} \\
\text{ind} (A) + \text{ind}(B)-1 & \text{ if both indices are odd} .
\end{cases}
\end{equation}

To prove this lower bound
we develop an equivariant join operation on homology.

The join stability property
of Proposition \ref{proposition: cohomological index}(v),
\begin{equation} \labell{equation: js in app}
\ind \big(A \ast_{\mathbb{Z}_k} L_k^{2K-1}(\ul{w}')\big) 
 = \ind (A) + 2K \,,
\end{equation}
is a special case of the join quasi-additivity property,
but since several of our applications
of the non-linear Maslov index only need this special case,
we also give a short direct proof of it.

\begin{Remark} \labell{open sets}
By continuity of the index 
(Proposition \ref{proposition: cohomological index}(i)),
and since for every neighborhood $\calO$ of $A*_{\zk}B$
there exist neighborhoods $U$ of $A$ and $V$ of $B$
such that $\calO$ contains $U *_{\zk} V$,
it is enough to prove \eqref{lower bounds}
and \eqref{equation: js in app}
when $A$ and $B$ are \emph{open} subsets of lens spaces.
\end{Remark}

\begin{lemma} \labell{ind in homology}
Let $A$ be an open subset of a lens space $L_k^{2M-1}(\ul{w})$.
Then the map 
$$
\iota_* \colon H_j(A; \zk) \to H_j \big(L_k^{2M-1}(\ul{w}); \zk\big)
$$
that is induced by the inclusion 
$\iota \colon A \hookrightarrow L_k^{2M-1}(\ul{w})$
is surjective (with image $\cong \zk$) for all $j < \ind(A)$
and is the zero map for all $j \geq \ind(A)$.
\end{lemma}

\begin{proof}
Because $A$ is a manifold, its \v{C}ech cohomology 
agrees with its singular cohomology.
The result then follows from Lemma \ref{no holes in index}
by the duality between homology and cohomology with field coefficients.
\end{proof}

\subsection*{Proof of the join stability property 
\eqref{equation: js in app}} \ 
\labell{proof join stability} 

Recall that the Thom space
of a real vector bundle $\pi \colon E \rightarrow X$
is the space $\textrm{Th}(\pi) = D(E)/S(E)$,
where $S(E)$ and $D(E)$ denote the total spaces
of the corresponding unit sphere bundle and closed disk bundle.
An orientation of $E$ gives rise
to the Thom isomorphism
$$
T \colon \widetilde H_{j+m} \big(\textrm{Th}(\pi)\big) \to H_j (X)
\quad \text{ for } j \geq 0 \,,
$$
where the tilde denotes reduced homology,
and $m$ is the rank of $\pi \colon E \rightarrow X$
(see for instance \cite[Theorem 5.7.10]{Spanier}).

By Remark \ref{open sets}, we can assume that
$A$ is an \emph{open} subset of $L_k^{2M-1}(\ul{w})$.
The preimage $\tA$ of $A$ in $S_k^{2M-1}(\ul{w}')$
is a principal $\zk$-bundle.
Consider the associated vector bundle
\begin{equation} \labell{E}
 \pi \colon \big( \tA \times \R^{2K}(\ul{w}') \big) / \, \zk \to A \,.
\end{equation}
We claim that there is a cofibre sequence
\begin{equation} \labell{cofibre}
L^{2K-1}_k(\ul{w}') \hookrightarrow 
A \ast_{\zk} L^{2K-1}_k(\ul{w}') \stackrel{q}{\to} \textrm{Th}(\pi),
\end{equation}
where the first map is the canonical embedding into the join.
Indeed, there is a natural $\zk$-equivariant homeomorphism
$$
\psi \colon \big( \widetilde{A} \ast S^{2K-1}(\ul{w}') \big) 
   \ssminus S^{2K-1}(\ul{w}') 
   \, \to \, \widetilde{A} \times \textrm{int} \big( D^{2K}(\ul{w}') \big) 
$$
where $\textrm{int} \big(D^{2K}(\ul{w}')\big) \subset \C^k(\ul{w}')$ 
is the open unit disc,
and $\zk$ acts on $\wt{A} \times \textrm{int} \big(D^{2K}(\ul{w}')\big)$ diagonally.
The map that $\psi$ induces on the quotient spaces identifies
$\big( A \ast_{\zk} L^{2K-1}_k(\ul{w}') \big) \ssminus L^{2K-1}_k(\ul{w}')$
with the total space of the open disk bundle of the vector bundle \eqref{E}
and extends to a homeomorphism
$$
\left(A\ast_{\zk} L^{2K-1}_k(\ul{w}')\right)
   /L^{2K-1}_k(\ul{w}') \, \to \, \textrm{Th}(\pi),
$$
which expresses \eqref{cofibre} as a cofibre sequence.

The canonical embedding
$L^{2K-1}_k(\ul{w}') \hookrightarrow A \ast_{\zk} L^{2K-1}_k(\ul{w}')$
is injective in homology
(for example, because its composition 
with the classifying map of $A *_{\zk} L^{2K-1}_k(\ul{w}')$
is a classifying map for $L^{2K-1}_k(\ul{w}')$, which is injective in homology).
It follows that the long exact sequence
associated to~\eqref{cofibre} splits
into short exact sequences
$$ 0 \to H_j(L_k^{2K-1}(\ul{w}'))
    \to H_j(A *_{\zk} L_k^{2K-1}(\ul{w}'))
    \to H_j(\textrm{Th}(\pi))
    \to 0,$$
and thus the collapse map
$$
q \colon A \ast_{\zk} L^{2K-1}_k(\ul{w}') \to \textrm{Th}(\pi)
$$
induces an isomorphism in homology in degrees $\geq 2K$.
Consider the isomorphism $\ast \ell$ defined by
$$
\ast \ell \colon 
H_j (A)  \xrightarrow{T^{-1}}  H_{j+2K} \big(\textrm{Th}(\pi)\big)
\xrightarrow{(q_*)^{-1}} 
H_{j+2K} \left(A\ast_{\zk} L^{2K-1}_k(\ul{w}') \right)  
\text{ for } j \geq 0.
$$
Since $\ast \ell$ is natural, 
by applying it to $A$ and to $L_k^{2M-1}(\ul{w})$ we obtain
a commuting diagram
\begin{equation} \labell{commuting}
\xymatrix{
H_j (A) \ar[d]_{\iota_*} \ar[r]^{\ast \ell \qquad \qquad \ } 
   & H_{j+2K} \left( A \ast_{\zk} L^{2K-1}_k(\ul{w}')\right)
 \ar[d]^{(\iota \, \ast_{\zk} \text{id})_{\ast}} \\
H_j \big(L^{2M-1}_k(\underline w)\big)  \ar[r]_{\ast \ell \quad \qquad \ } 
 & H_{j+2K} \big(L^{2(M+K)-1}_k(\underline w, \ul{w}')\big) \,.
}
\end{equation}
Since $\ind(A)$ and $\ind(A *_{\zk} L_k^{2K}(\ul{w}') )$ 
are the lowest degrees of the homology groups on which 
the corresponding vertical arrows of \eqref{commuting} are the zero maps
(by Lemma \ref{ind in homology}),
and since the horizontal arrows of \eqref{commuting} are isomorphisms 
that increase the degree by $2K$, 
we conclude that
$
\text{ind} \big(A \ast_{\zk} L_k^{2K-1}(\ul{w}')\big) = \text{ind}(A) + 2K
$.

\subsection*{The join and the equivariant join}

In order to construct a join operation on singular homology we first need 
to discuss in some detail the join operation on topological spaces.
Here, all topological spaces are assumed to be Hausdorff. 
Let
$$
\Delta^m = \big\{(t_0,\ldots,t_m) \in \R^{m+1}_{\geq 0} \;\lvert\; 
   t_0 + \ldots + t_m = 1 \big\}
$$
be the standard $m$-simplex ($m\geq 0$).
The \emph{join} of two topological spaces $X$ and $Y$
is defined to be the quotient
$$
X \ast Y = (X \times Y \times \Delta^1) / \sim 
$$
where the only non-trivial relations are
$(x,y,(0,1)) \sim (x',y,(0,1))$ and $(x,y,(1,0)) \sim (x,y',(1,0))$
for all $x, x' \in X$ and $y,y' \in Y$.
We denote the equivalence class of $(x,y,(t_0,t_1))$ in $X \ast Y$ 
by $ t_0 x+t_1 y$.
Similarly,
the join of $(m+1)$ topological spaces $X_0,\ldots,X_m$
is defined by
$$ X_0 \ast \ldots \ast X_m = 
 \left( X_0 \times \ldots \times X_m \times \Delta^m \right) / \sim ,$$
with the class of $(x_0,\ldots,x_m,(t_0,\ldots,t_m))$
denoted by $t_0 x_0 + \ldots + t_m x_m$, and with 
$t_0 x_0 + \ldots + t_m x_m 
 = t_0' x_0' + \ldots + t_m' x_m'$
if and only if $t_j = t_j'$ for all $j$ and $x_j = x_j'$ whenever $t_j \neq 0$.

The join operation is natural,
in the sense that 
continuous functions $f_j \colon X_j \to X_j'$ induce a continuous function
$ f_0 \ast \cdots \ast f_m 
  \colon X_0 \ast \cdots \ast X_m \to X_0' \ast \cdots \ast X_m'$,
so that $(g_0 \circ f_0) \ast \cdots \ast (g_m \circ f_m)
 = (g_0 \ast \cdots \ast g_m) \circ (f_0 \ast \cdots \ast f_m)$.
In particular, when the spaces $X_j$ are equipped with an action 
of a group $G$, their join acquires the $G$-action 
given by $a \cdot (t_0 x_0 + \cdots + t_m x_m) = 
t_0(a \cdot x_0) + \cdots + t_m (a \cdot x_m)$.
The (equivariant) join of principal $G$-bundles is a principal $G$-bundle.
If $X$ and $Y$ are spaces equipped with principal $G$-bundles 
$\tX \to X$ and $\tY \to Y$, we define their $\mathit G${\it -join}
to be the space $X \ast_G Y := (\tX \ast \tY) / G$
equipped with the principal bundle $\tX \ast \tY \to X \ast_G Y$.

The natural maps
\begin{equation}\labell{assocjoin}
(X \ast Y) \ast Z \xrightarrow{\varphi_1} X \ast Y \ast Z 
 \xleftarrow{\varphi_2} X\ast(Y\ast Z)
\end{equation}
given by
$\varphi_1 \colon s_0 ( t_0 x + t_1 y) + s_1 z 
 \mapsto  (s_0 t_0) x + (s_0 t_1) y + s_1 z $
and similarly for $\varphi_2$
are homeomorphisms
(because they are continuous proper bijections of Hausdorff spaces).
We have similar maps for any parenthetization of any number of factors.
We also have the map 
\begin{equation}\labell{commjoin}
\tau \colon X\ast Y \to Y \ast X \ , \quad
   t_0 x + t_1 y \mapsto t_1y + t_0 x \,.
\end{equation}
These maps are natural, in the sense that 
for any continuous maps $f \colon X \to X'$,
\ $g \colon Y \to Y'$ and $h \colon Z \to Z'$
the diagrams
\begin{equation} \labell{naturality}
\xymatrix{
 (X \ast Y) \ast Z    \ar[r]^{\varphi_1} \ar[d]^{ (f \ast g) \ast h } 
  & X \ast Y \ast Z \ar[d]^{f \ast g \ast h} \\
 (X' \ast Y') \ast Z' \ar[r]^{\varphi_1} & X' \ast Y' \ast Z' 
}
\quad \text{ and } \quad
\xymatrix{ 
 X \ast Y \ar[d]^{f \ast g} \ar[r]^{\tau} & Y \ast X \ar[d]^{g \ast f} \\
 X' \ast Y' \ar[r]_{\tau} & Y' \ast X'
}
\end{equation}
commute, and similarly for $\varphi_2$. 
For principal $G$-bundles over $X$, $Y$, $Z$,
the associativity and commutativity homeomorphisms 
\eqref{assocjoin} and \eqref{commjoin}
pass to the quotients and yield homeomorphisms
$$
(X \ast_G Y) \ast_G Z \, \xrightarrow{} \,  X \ast_G Y \ast_G Z
    \, \xleftarrow{} \, X \ast_G ( Y \ast_G Z) 
\quad \text{ and } \quad X \ast_G Y \, \xrightarrow{} \, Y \ast_G X \,
$$
that are natural in the sense analogous to \eqref{naturality}.

\subsection*{Joins of standard simplices, spheres, and lens spaces}

For the standard unit spheres, we have identifications 
$$ \psi_{M,M'} \colon
   S^{2M-1} \ast S^{2M'-1} \, \xrightarrow{\cong} \, S^{2(M+M')-1} $$
given by
$$
 t (z_0,\ldots,z_M) + t' (z_0', \ldots, z_{M'}') \xrightarrow{}
 ( \sqrt{t} z_0 \, , \ldots, \, \sqrt{t} z_M \, , 
   \sqrt{t'} z_0' \, , \ldots, \, \sqrt{t'} z_{M'}')
$$
(these maps are continuous proper bijections, hence homeomorphisms).
These identifications descend to lens spaces, where we use the same notation
\begin{equation}\labell{identlens}
\psi_{M,M'} \colon 
L_k^{2M-1}(\underline w) \ast_{\zk} L_k^{2M'-1}(\underline w') 
\xrightarrow{\cong} L_k^{2(M+M')-1}(\underline w,\underline w').
\end{equation}
We have similar identifications for multiple joins.

These identifications are consistent,
in the sense that the diagrams
\begin{equation} \labell{consistent}
\xymatrix{
 L_k^{2M-1} \ast_\zk L_k^{2M'-1} 
                    \ast_\zk L_k^{2M''-1} 
\ar[rr]_{ \quad \psi_{M,M',M''} }
\ar[d]^{ (\varphi_1)^{-1} }
 && L_k^{2(M+M'+M'')-1} \\
\left( L_k^{2M-1} \ast_\zk L_k^{2M'-1} \right) 
          \ast_\zk L_k^{2M''-1} 
   \ar[rr]_{ \quad \psi_{M,M'} \ast_\zk \text{Id} }  
 && L_k^{2(M+M')-1} \ast_\zk L_k^{2M''-1} 
   \ar[u]_{ \psi_{M+M',M''} } 
}
\end{equation}
commute, and similarly for the other parenthetization,
where to simplify notation we omitted the weights.
By induction we get consistent identifications
of iterated multiple joins of lens spaces with higher lens spaces.

Given subsets $A \subset L_k^{2M-1}(\ul{w})$
and $B \subset L_k^{2M'-1}(\ul{w}')$, 
identifying $A \ast_{\zk} B$ with a subset of 
$L_k^{2M-1}(\ul{w}) \ast_{\zk} L_k^{2M'-1}(\ul{w}')$
by naturality and further with a subset of $L_k^{2(M+M')-1}(\ul{w},\ul{w}')$
by \eqref{identlens}, 
we get the same subset of $L_k^{2(M+M')-1}(\ul{w},\ul{w}')$
that was described in \eqref{join1}.

\begin{Remark*}
If $f$ and $g$ are smooth maps between spheres or lens spaces,
then $f \ast g$
(viewed as a map between higher dimensional spheres or lens spaces)
might not be smooth.
\end{Remark*}

For the standard simplices, the identifications
\begin{equation} \labell{identsimp}
\Delta^l \ast \Delta^m \xrightarrow{\cong} \Delta^{l+m+1}
\end{equation}
given by 
$$
u_0(t_0,\ldots,t_l) + u_1(s_0,\ldots,s_m) 
 \mapsto (u_0t_0,\cdots,u_0t_l,u_1s_0,\ldots,u_1s_m) 
$$
are consistent, in a sense similar to \eqref{consistent}.  
In particular, the composition 
\begin{equation} \labell{identity}
\Delta^{l+m+n+2} \, \cong \, \Delta^l \ast \Delta^m \ast \Delta^n
 \, \cong \, (\Delta^l \ast \Delta^m) \ast \Delta^n 
 \, \cong \, \Delta^{l+m+1} \ast \Delta^n
 \, \cong \, \Delta^{l+m+n+2} 
\end{equation}
is the identity map, and similarly with the other parenthetization.

\subsection*{The join operation on homology}

Given singular simplices $\sigma \colon \Delta^l \to X$
and $\mu \colon \Delta^m \to Y$ on $X$ and $Y$,
the identification \eqref{identsimp}
makes their join into a singular $(l+m+1)$-simplex
$$ 
\sigma \ast \mu \colon \Delta^{l+m+1} \to X \ast Y \,.
$$
Extending bilinearly, we obtain a map of singular chains
\begin{equation}
\labell{joinhom}
 \ast \colon C_l(X) \otimes C_m(Y) \xrightarrow{} C_{l+m+1}(X\ast Y)\,.
\end{equation}
Similarly, the triple join gives a map of singular chains
$$
C_j(X) \otimes C_l(Y) \otimes C_m(Z) \to C_{j+l+m+2}(X \ast Y \ast Z) \,.
$$
The definitions of the join and boundary operations
$\ast$ and $\partial$ directly imply that
for any two singular simplices $\sigma \colon \Delta^l \to X$ 
and $\mu \colon \Delta^m \to Y$ we have
\begin{equation}\labell{boundaryjoin}
\partial(\sigma\ast \mu) = \begin{cases}
(\partial \sigma) \ast \mu + (-1)^{l+1} \sigma \ast \partial \mu 
                  & \text{ if } \; l,m > 0, \\
(\partial \sigma) \ast \mu + (-1)^{l+1} \sigma 
                  & \text{ if } \; l>0, m=0, \\
\mu - \sigma \ast (\partial \mu) & \text{ if } \; l = 0, m > 0 \\
\mu - \sigma & \text{ if } \; l = m = 0\,.
\end{cases}
\end{equation}
It follows from \eqref{boundaryjoin}
that the chain join \eqref{joinhom} defines an operation
$$
\ast \colon H_l (X) \otimes H_m (Y) \to H_{l+m+1}(X \ast Y) 
 \quad \text{ when } l>0 \text{ and } m>0 \,.
$$
We now show that the same considerations go through
also for principal $G$-bundles for a finite group $G$,
and that moreover if we consider $\zk$-coefficients
with $k$ dividing the order of $G$
then the induced operation in homology is defined in all degrees.

Let $\tX \to X$ be a principal $G$-bundle and $C_*( \widetilde{X})^G$
the complex of $G$-invariant chains on $\tX$.
There is a canonical isomorphism of complexes
$$
\varphi \colon C_*(X) \xrightarrow{\cong} C_*(\widetilde{X})^G
$$
which on a singular simplex $\sigma \colon \Delta^m \to X$ is given by 
$\varphi(\sigma) = \sum_{g \in G} g \cdot \widetilde{\sigma}\,$,
where $\widetilde{\sigma} \colon \Delta^m \to \wt{X}$ 
is any lift of $\sigma$.
The join operation sends $G$-invariant chains
on $\widetilde{X}$ and $\widetilde{Y}$ to $G$-invariant chains
on $\widetilde{X} \ast \widetilde{Y}$.  Defining
$$
\sigma \ast_G \mu 
   := \varphi^{-1} \big( \varphi(\sigma) \ast \varphi(\mu)\big)
$$
on simplices, and extending bilinearly, we obtain an 
equivariant join operation on chains,
\begin{equation}\labell{joinequivhom}
\ast_G \colon C_l (X) \otimes C_m(Y) \to C_{l+m+1}(X \ast_G Y) \,.
\end{equation}

The definitions of the join operation on chains \eqref{joinhom}
and on equivariant chains \eqref{joinequivhom}
make sense with arbitrary ring coefficients.
We now consider $\zk$-coefficients, for $k$ dividing the order of $G$.

\begin{lemma} \labell{del of join}
Let $\widetilde{X} \to X$ and $\widetilde{Y} \to Y$ be principal $G$-bundles.
Assume that $k$ divides the order of $G$.
Then the equivariant join operation on chains \eqref{joinequivhom} satisfies
$$
\partial \, (\sigma \ast_G \mu) 
   = (\partial \sigma) \ast_G \mu + (-1)^{l+1} \sigma \ast_G \partial \mu\,.
$$ 
\end{lemma}

\begin{proof}
By bilinearity, 
it is enough to consider the case of simplices (rather than chains)
$\sigma \colon \Delta^l \to X$ and $\mu \colon \Delta^m \to Y$.
If $l > 0$ and $m > 0$ the result follows from \eqref{boundaryjoin}.
Assume now that $l > 0$ and $m = 0$ (the remaining cases are similar).
Let $\wt{\sigma} \colon \Delta^l \to \tX$
and $\wt{\mu} \colon \Delta^m \to \tY$
be any lifts of $\sigma$ and $\mu$.
Since we use $\zk$-coefficients, we have
\begin{eqnarray*}
\partial \, (\sigma \ast_G \mu) &=& \partial \, \Big(\varphi^{-1}\big(\varphi(\sigma) \ast \varphi(\mu)\big)\Big) \\
&=& \varphi^{-1} \left( \partial \, \Big( \sum_{g,h \in G} (g \cdot \wt \sigma) \ast (h \cdot \wt \mu) \Big) \right) \\
&=& \varphi^{-1} \left( \sum_{g,h \in G} \Big( \big( \partial \, (g \cdot \wt \sigma)\big) \ast (h \cdot \wt \mu) + (-1)^{l+1} g 
\cdot \wt \sigma \Big) \right) \\
&=& \varphi^{-1} \Big(\varphi(\partial\sigma) \ast \varphi(\mu)\Big)
+ \varphi^{-1} \Big(|G| (-1)^{l+1} \varphi(\sigma)\Big) \\
&=&  (\partial\sigma) \ast_G \mu\,.
\end{eqnarray*}
\end{proof}

Lemma \ref{del of join} implies that
if $k$ divides the order of $G$ then
the equivariant join operation on chains induces an operation on homology:
$$
\ast_G \colon H_l(X;\zk) \otimes H_m(Y;\zk) \to H_{l+m+1}(X \ast_G Y;\zk) 
 \quad \text{ for all } l \geq 0 \text{ and } m \geq 0 \,.
$$
Moreover, the naturality of the chain level formula
provides join operations on relative homology
$$ 
\ast_G \colon H_l(X,A;\zk) \otimes H_m(Y,B;\zk) 
 \to H_{l+m+1} \big(X \ast_G Y \, , \, (X \ast_G B) \cup (A \ast_G Y) \, ;\, \zk \big)
$$
such that $\ol{ x_l \ast_G y_m } = \ol{ x_l } \ast_G \ol{ y_m }$
for any $x_l \in H_l(X; \zk)$ and $y_m \in H_m(Y; \zk)$,
where $\ol{x_l}$, $\ol{y_m}$ and $\ol{x_l \ast_G y_m}$
denote the images of $x_l$, $y_m$ and $x_l \ast_G y_m$
in relative homology.

It follows from the consistency of the identifications 
of the standard simplices 
(specifically, from \eqref{identity} being the identity map)
that the join operation on chains \eqref{joinhom} 
is associative, in the following sense.
For any three singular simplices $\sigma \colon \Delta^j \to X$, \
$\mu \colon \Delta^l \to Y$ and $\nu \colon \Delta^m \to Z$ we have
\begin{equation} \labell{assocjoinchains}
\varphi_1 \circ \big( ( \sigma\ast\mu ) \ast \nu \big) 
 = \sigma \ast \mu \ast \nu 
 = \varphi_2 \circ \big( \sigma \ast (\mu\ast\nu)\big) \,,
\end{equation}
where $\varphi_1$ and $\varphi_2$ are the homeomorphisms of \eqref{assocjoin}.
This further implies that the equivariant join operation 
on homology is associative, in the following sense.
For homology classes $\alpha,\beta,\gamma$
on spaces $X,Y,Z$ equipped with principal $G$-bundles
with $k$ dividing the order of $G$, we have
$ (\varphi_1)_* \big( (\alpha \ast \beta) \ast \gamma \big)
 = \alpha \ast \beta \ast \gamma
 = (\varphi_2)_* \big( \alpha * (\beta * \gamma) \big)$.

The join operation in homology also satisfies a commutativity 
property.  We postpone this result to at the end of this appendix
(Proposition \ref{twist});
we do not need it for our applications.

\subsection*{Computations for lens spaces}

Lemma \ref{alternative} and Proposition \ref{complens}
contain computations of equivariant joins for lens spaces.
For our applications, we only need the ``if'' direction 
of Proposition \ref{complens} 
and we don't need Lemma~\ref{alternative}.

\begin{lemma}\labell{alternative}
Let $x_0 \in H_0(L^1_k;\zk)$ be the homology class of a point.
Then $x_0 \ast_\zk x_0 = 0$.
\end{lemma}

\begin{proof}
The class $x_0 \in H_0(L^1_k; \zk)$ is represented
by the singular simplex sending $\Delta^0$ to $[1] \in L^1_k$.
The class $x_0 \ast_{\zk} x_0 \in H_1(L^3_k;\zk)$
is represented by $\sum_{j=0}^{k-1} \sigma_j$,
where, for $j = 0, \ldots, k-1$,
$\sigma_j \colon \Delta^1=[0,1] \to L^3_k$ is the singular $1$-simplex 
given by 
$$
\sigma_j(t)=\left[ \sqrt{1-t} , \sqrt{t} \, e^{j\frac{2\pi i }{k}} \right] \,.
$$ 
The paths $\sigma_j$ all have initial point $[1,0]$ and end point $[0,1]$ 
in $L^3_k$.
The concatenation $\ol{\sigma_0} \sigma_j$ is a loop in $L_k^3$
in the homology class $[\sigma_j - \sigma_0]$.
As the lift of $\ol{\sigma_0} \sigma_j$ to $S^3$ starting at $(0,1)$ ends at $(0,e^{2 \pi i j / k})$
we see that, for $1 \leq j \leq k-1$, the loop $\ol{\sigma_0} \sigma_j$ 
generates $\pi_1(L_k^3)$ 
and $\ol{\sigma_0}\sigma_j = (\ol{\sigma_0} \sigma_1)^j$ in $\pi_1 (L_k^3) $.
In $H_1 (L_k^3)$ we have $[\sigma_j - \sigma_0] = j [\sigma_1 - \sigma_0]$.
So
$$
x_0 \ast_{\zk} x_0 = \Big[ \sum_{j=0}^{k-1} \sigma_j \Big] 
   =  \Big[ \sum_{j=0}^{k-1} (\sigma_j-\sigma_0) \Big]
   = \sum_{j=0}^{k-1} j [\sigma_1 - \sigma_0] 
   = \frac{k(k-1)}{2} \, [\sigma_1 - \sigma_0] = 0 \,.
$$
\end{proof}

\begin{prop}\labell{complens}
Suppose that  $x_m$ and $x_{m'}$ are non-zero elements
of $H_m \big(L^{2M-1}_k(\underline w); \zk\big)$
and $H_{m'} \big(L^{2M'-1}_k(\underline w'); \zk\big)$.
Then the join $x_{m} \ast_{\zk} x_{m'}$ is non-zero 
if and only if $m$ or $m'$ is odd.
\end{prop}

\begin{proof}
By functoriality, it suffices to consider the case
when the weights $\underline w$ and $\underline w'$ 
are of the form $(1,\ldots,1)$.
Indeed, for any $\ul{w}$ (and similarly for $\ul{w}'$), 
the classifying map of $L^{2M-1}_k(\ul{w})$
induces an injection in homology (even over $\Z$)
and can be obtained as the composition of a map
$L^{2M-1}_k(\ul{w}) \to L^{2K-1}_k := L^{2K-1}_k (1,\ldots,1)$
with the classifying map $L^{2K-1}_k \to L^\infty_k$
for some sufficiently large $K$.

Let $\sigma_0 \colon \Delta^0 \to L_k^1$ be the simplex $1 \mapsto [1]$
and let $\sigma_1 \colon \Delta^1 \to L_k^1$ 
be the simplex $(t_0,t_1) \mapsto [e^{2\pi i t_1/k} ]$.
As chains, $\sigma_0$ and $\sigma_1$ are closed;
denote their homology classes by $y_0 = [\sigma_0]$ and $y_1 = [\sigma_1]$.

The standard cell decomposition of $L^{2M+1}_k$
can be described as follows (see for instance \cite[Example 2.43]{Hatcher}).
There is one cell $e^j$ in each dimension $0 \leq j \leq 2M+1$.
The standard inclusion $L_k^{2M-1} \to L_k^{2M+1}$, $[z] \mapsto [z,0]$
takes the $j$th cell of $L_k^{2M-1}$
to the $j$th cell of $L_k^{2M+1}$
for all $0 \leq j \leq 2M-1$,
and is injective in homology.
Moreover, under the identification
$L^{2M+1}_k = L^{2M-1}_k \ast_{\zk} L^1_k$ of \eqref{identlens},
we have that $e^{2M} = e^{2M-1} \ast_{\zk} e^0$ 
and $e^{2M+1} = e^{2M-1}\ast_{\zk} e^1$.

When $j$ is odd, the $j$th skeleton of the cellular decomposition
of $L_k^{2M+1}$ is the lens subspace $L_k^j$,
embedded by the standard inclusion $[z] \mapsto [z,0]$.
We denote the $j$th skeleton of the cell complex by $L_k^j$
even when $j$ is even.

Let 
$$
\sigma_{2M-1} \colon (\Delta^{2M-1}, \partial \Delta^{2M-1}) 
              \to (L^{2M-1}_k,L^{2M-2}_k)
$$
be a characteristic map of $e^{2M-1}$.
Then $[\sigma_{2M-1}]$ is a generator of $H_{2M-1}(L^{2M-1}_k,L^{2M-2}_k;\zk)$.
Let $y_{2M-1}$ be the generator of $H_{2M-1}(L_k^{2M-1})$
that maps to $[\sigma_{2M-1}]$
under the isomorphism
$$ H_{2M-1}(L_k^{2M-1};\zk) \to H_{2M-1} (L_k^{2M-1}, L_k^{2M-2}; \zk ).$$  

For $j = 0,1$ the chains $\sigma_{2M-1} \ast_\zk \sigma_j$
are triangulations of the cells $e^{2M+j}$ relative to their boundary,
so they represent generators of $H_{2M+j}(L^{2M+j}_k, L^{2M+j-1}_k;\zk)$.
By the naturality of the equivariant join, $y_{2M-1} \ast_\zk y_j$ 
maps to $[\sigma_{2M-1} \ast_\zk \sigma_j]$ under the isomorphism 
$$
H_{2M+j}(L^{2M+j}_k; \zk) \to H_{2M+j}(L^{2M+j}_k,L^{2M+j-1}_k;\zk)  \,.
$$
It follows that $y_{2M-1} \ast_\zk y_j $ is a generator
of $H_{2M+j}(L_k^{2M+1};\zk)$.

Taking iterations, and using associativity to remove the brackets,
we conclude that each of the classes
$y_1 \ast_\zk \ldots \ast_\zk y_1$, \ 
$y_1 \ast_\zk \ldots \ast_\zk y_1 \ast_\zk y_0$
and (by a similar argument) 
$y_0 \ast_\zk y_1 \ast_\zk \ldots \ast_\zk y_1 $
is a generator of the homology group in the appropriate dimension.

Let $x_m \in H_m(L_k^{2M-1};\zk)$ 
and $x_{m'} \in H_{m'}(L_k^{2M'-1};\zk)$ be non-zero classes.
Suppose that $m$ and $m'$ are not both even.
Expressing $x_m$ as a non-zero scalar multiple 
of $y_1 \ast_\zk \ldots \ast_\zk y_1$
or $y_0 \ast_\zk y_1 \ast_\zk \ldots \ast_\zk y_1$, 
and expressing $x_{m'}$ as a non-zero scalar multiple 
of $y_1 \ast_\zk \ldots \ast_\zk y_1$
or $y_1 \ast_\zk \ldots \ast_\zk y_1 \ast_\zk y_0$,
we conclude (by associativity) that $x_m \ast x_{m'}$
(in which $y_0$ might occur as a first or last factor but not both)
is non-zero.

Since $(y_1 \ast_\zk \ldots \ast_\zk y_1 \ast_\zk y_0) \ast_\zk
(y_0 \ast_\zk y_1 \ast_\zk \ldots \ast_\zk y_1)$ is zero
(by associativity and by Lemma \ref{alternative}),
it similarly follows that if $m$ and $m'$ are both even
then $x_m \ast_\zk x_{m'}$ is zero.
\end{proof}

\subsection*{Proof of the lower bounds \eqref{lower bounds}
on the index of a join}

By Remark \ref{open sets}, we can assume that $A$ and $B$
are open subsets of $L_k^{2M-1}(\underline w)$ and of $L_k^{2M'-1}(\underline w')$.

First, suppose that $\ind(A)$ or $\ind(B)$ is even.
By Lemma \ref{ind in homology}, 
there exist classes $\alpha \in H_{\ind(A)-1}(A)$
and $\beta \in H_{\ind(B)-1}(B)$
whose images in $H_{\ind(A)-1}(L_k^{2M-1}(\ul{w}))$ 
and in $H_{\ind(B)-1}(L_k^{2M'-1}(\ul{w}'))$ are non-zero.
By Lemma \ref{complens} and the naturality of the equivariant join,
$\alpha \ast_\zk \beta$ is a class in 
$H_{\ind(A)+\ind(B)-1} (A \ast_\zk B)$ whose image in 
$H_{\ind(A)+\ind(B)-1} ( L_k^{2(M+M'-1)}(\ul{w},\ul{w}') )$
is non-zero.  By Lemma \ref{ind in homology}, 
this shows that $\ind(A \ast_\zk B) \geq \ind(A) + \ind(B)$.

If $\ind(A)$ or $\ind(B)$ are both odd,
we apply a similar argument to classes
$\alpha \in H_{\ind(A)-1}(A)$ and $\beta' \in H_{\ind(B)-2}(B)$
to conclude that
$\ind(A \ast_\zk B) \geq \ind(A) + \ind(B) - 1$.

\begin{rmk}\labell{remark: RP}
In the case of projective space,
any cell of the standard cellular decomposition
is the equivariant join of the cell in the previous degree with a 0-cell.
Therefore the proof of Proposition \ref{complens}
shows that in this case the join of two generators in any degree is non-zero.
It follows from this argument
and property (iv') in Remark \ref{remark: stronger properties for RP}
that in the case of projective spaces
the cohomological index is join additive:
for closed subsets $A$ of $\mathbb{RP}^M$
and $B$ of $\mathbb{RP}^{M'}$ we have
$\ind (A \ast_{\zt} B) = \ind (A) + \ind (B)$.
\eor
\end{rmk}

\subsection*{Commutativity of the homology join}

We complete our discussion of the join operation on homology
with a commutativity property of this operation.

\begin{prop}\labell{twist}
Let $\widetilde{X} \to X$ and $\widetilde{Y} \to Y$ be principal $G$-bundles
and suppose that $k$ divides the order of $G$.
Let $\tau \colon X \ast_G Y \to Y \ast_G X$
denote the homeomorphism~\eqref{commjoin}.
Then for $\alpha \in H_l(X;\zk)$ and $\beta \in H_m(Y;\zk)$ we have
$$
\tau_{\ast} (\alpha\ast_G \beta) = (-1)^{(l+1)(m+1)} \beta \ast_G \alpha\,.
$$
\end{prop}
 
\begin{proof}
We use the geometric interpretation of singular cycles
that appears in \cite[p. 108--109]{Hatcher}.
Let $x \in C_l(X; \zk)$ be a cycle, and write 
$x = \sum_i x_i \sigma_i$
with $\sigma_i \colon \Delta^l \to X$ and $x_i \neq 0$.
Let 
$$
K_x = \left(\coprod \Delta_i^l \right)/ \sim
$$
be the disjoint union of one $l$-simplex $\Delta_i^l$ for each $\sigma_i$,
quotiented by the identification of 
the facets of the $\Delta^l_i$s
that give rise (via the maps $\sigma_i$) to the same 
singular $(l-1)$-simplex.
Then the singular simplices $\sigma_i$'s assemble to give a map
$\sigma \colon K_x \to X$.
We denote by $\bar \sigma_i \colon \Delta^l_i \to K_x$
the inclusion in the coproduct followed by quotient.
Then 
$$
\bar{x} := \sum_i x_i \bar \sigma_i \in C_l (K_x; \zk)
$$
is a cycle and $ \sigma_*(\bar{x}) = x$.
Now let $x = \sum_i x_i \sigma_i \in C_l(X; \zk)$ and $y = \sum_j y_j \mu_j \in C_m(Y;\zk)$
be cycles representing the homology classes $\alpha$ and $\beta$ respectively.
Then 
$$
x \ast_G y = \varphi^{-1} \Big( \sum_{g,h \in G} \sum_{i,j} x_i y_j (g \cdot \wt \sigma_i) \ast (h \cdot \wt \mu_j)\Big)
$$
and
$$
y \ast_G x = \varphi^{-1} \Big( \sum_{g,h \in G} \sum_{i,j} x_i y_j (h \cdot \wt \mu_j) \ast (g \cdot \wt \sigma_i) \Big) \,
$$
where $\wt{\sigma}_i$ and $\wt{\mu}_j$ 
are some lifts of $\sigma_i$ and $\mu_j$.
Let $\nu \colon K_{x \ast_G y} \to X \ast_G Y$ and $\eta \colon K_{y \ast_G x} \to Y \ast_G X$
be the maps geometrically realizing the cycles $x \ast_G y$ and $y \ast_G x$
via the procedure described above,
and $\overline{x \ast_G y} \in C_{l+m+1}(K_{x \ast_G y}; \zk)$
and $\overline{y \ast_G x} \in C_{l+m+1}(K_{y \ast_G x}; \zk)$
be the corresponding cycles. 
Then the diagram
$$ 
\xymatrix{
K_{x \ast_G y} \ar[d]_{T} \ar[r]^{\nu} & X \ast_G Y \ar[d]^\tau \\
K_{y \ast_G x} \ar[r]_{\eta} & Y \ast_G X 
}
$$
commutes, where $T$ is the homeomorphism induced by the canonical homeomorphisms
$$
\Delta^{l+m+1} = \Delta^l \ast \Delta^m \xrightarrow{\tau} \Delta^m \ast \Delta^l = \Delta^{l+m+1}
$$
between the top cells of $K_{x \ast_G y}$ and $K_{y \ast_G x}$.
Since for cellular homology with $\zk$-coefficients we have
$T_*([\overline{x \ast_G y}]) = (-1)^{(m+1)(l+1)} [\overline{y \ast_G x}]$, 
the same holds in singular homology.
Hence
\begin{align*}
\tau_{\ast} (\alpha\ast_G \beta)
& = \tau_{\ast} \, \nu_{\ast} \big([\overline{x \ast_G y}]\big) = \eta_{\ast} \, T_* \big([\overline{x \ast_G y}]\big) \\
& = \eta_* \big( (-1)^{(l+1)(m+1)}[\overline{y \ast_G x}] \big) = (-1)^{(l+1)(m+1)} \, \beta \ast_G \alpha \,. 
\end{align*}
\end{proof}



\end{document}